\documentclass[a4paper,11pt]{amsart}
\usepackage{amssymb,xspace,amscd}
\usepackage{hyperref}
\title{\mbox{Orbits of real forms}
\mbox{in complex flag manifolds}}
\author[A.~Altomani]{Andrea Altomani}
\address{A.\ Altomani:
University of Luxembourg \\ 162a, avenue de 
la Fa\"{\i}encerie\\\mbox{L-2309} Luxembourg}
\email{andrea.altomani@uni.lu}
\author[C.~Medori]{Costantino Medori}
\address{C.\ Medori:
Dipartimento di Matematica\\ Universit\`a di Parma\\ Viale G.P.
Usberti, 53/A
\\ 43100 Parma (Italy)} \email{costantino.medori@unipr.it}
\author[M.~Nacinovich]{Mauro Nacinovich}
\address{M.\ Nacinovich:
Dipartimento di Matematica\\ II Universit\`a di Roma
``Tor Ver\-ga\-ta''\\ Via della Ricerca Scientifica\\ 00133 Roma
(Italy)}
\email{nacinovi@mat.uniroma2.it}
\date{October 20, 2008}
\numberwithin{equation}{section}
\theoremstyle{plain}
\newtheorem{thm}{Theorem}[section]
\newtheorem{lem}[thm]{Lemma}
\newtheorem{cor}[thm]{Corollary}
\newtheorem{prop}[thm]{Proposition}
\theoremstyle{definition}
\newtheorem{exam}[thm]{Example}

\newtheorem{dfn}[thm]{Definition}
\newtheorem{rmk}[thm]{Remark}
\begin{document}
\begin{abstract} We investigate the $CR$ geometry of the orbits
$M$ of a real form ${{\mathbf{G}_0}}$ of a complex semisimple Lie group 
${{{\mathbf{G}}}}$
in a complex flag manifold ${X}={{{\mathbf{G}}}}/\mathbf{Q}$.
We are mainly concerned with finite type and holomorphic nondegeneracy
conditions, canonical ${{\mathbf{G}_0}}$-equivariant 
and Mostow fibrations, and topological
properties of the orbits.
\end{abstract}
\maketitle
\section*{Introduction}
In this paper we study a large class of homogeneous $CR$ manifolds,
that 
come up
as orbits of real forms in complex flag manifolds.
These objects, that we call here \emph{parabolic $CR$ manifolds},
were first considered 
by J. A. Wolf (see  \cite{Wolf69} and
also \cite{FHW06} for a comprehensive 
introduction to this topic).
He studied the action of 
a real form ${{\mathbf{G}_0}}$
of a complex semisimple Lie group ${{{\mathbf{G}}}}$
on a flag manifold ${X}$
of ${{{\mathbf{G}}}}$.
He showed that $\mathbf{G}_0$ has 
finitely many orbits in ${X}$, so that the
union of the open orbits is dense in $X$, and
that there is just one that is closed. 
He systematically investigated 
their properties, especially for
the open orbits and for the holomorphic
arc components of general orbits, outlining a framework
that includes, as special cases, the 
bounded symmetric domains.\par
We aim here
to give a contribution to the study of the general
${{\mathbf{G}_0}}$-orbits in ${X}$, 
by considering and utilizing their natural
$CR$ structure. 
\par
In \cite{AMN06} we began this program by investigating
the closed orbits. 
These can be combinatorially described
in terms of their \emph{cross-marked Satake diagrams}. 
This approach was possible 
because their isotropy subgroup always contains a maximally noncompact
Cartan subgroup of ${{\mathbf{G}_0}}$. This is no longer the case for
general orbits, and therefore 
we need here
to deal  with general
Cartan subgroups. 
This different situation lead us to
introduce
the notions of 
\emph{adapted Cartan pairs}
and \emph{fit Weyl chambers}, 
that we utilized to extend to
general ${{\mathbf{G}_0}}$-orbits, and even simplify, 
the criteria of 
\cite{AMN06}, that characterize finite type (see \cite{BG77}) and
holomorphic nondegeneracy (see \cite{BER99}).
\par
A deeper analysis of the equivariant ${{\mathbf{G}_0}}$-maps between
${{\mathbf{G}_0}}$-orbits of various complex flag manifolds of
${{\mathbf{G}}}$ yields an accurate description of the smooth
structure of the orbits. We show here 
(Theorem \ref{thm:ha})
that each orbit $M$ is
${{\mathbf{G}_0}}$-equivariantly
equivalent to a tower of fibrations over a 
canonically associated \emph{real} flag manifold
$M_{\mathfrak{e}}$, with fibers that are products of
Euclidean complex spaces 
and open orbits in complex flag manifolds. This result
can be utilized
to investigate some 
topological properties of $M$. It turns out, for instance, that the 
fundamental group $\pi_1(M)$ of $M$ 
only depends on 
$M_{\mathfrak{e}}$,
and on the conjugacy class of the
maximally noncompact Cartan subgroups
of the isotropy of the action of $\mathbf{G}_0$ on $M$
(see Theorem \ref{thm:ia}).
This explains why 
the fundamental group of a closed orbit 
$M$ is always isomorphic to that of~$M_{\mathfrak{e}}$ 
(see e.g.~\mbox{\cite[\S 8]{AMN06}}).\par
Let $M$ be any $\mathbf{G}_0$-orbit in $X$.
Fix a maximal compact subgroup
$\mathbf{K}_0$ 
of ${{\mathbf{G}_0}}$, containing a
maximal compact subgroup of
the isotropy subgroup $\mathbf{I}_0$ of $M$, and let
$\mathbf{K}\subset{{\mathbf{G}}}$ be
its complexification. The Matsuki-dual (see e.g. 
\cite{Mats88}, \cite{BrLo02})
complex $\mathbf{K}$-orbit $M^*$ intersects
$M$ into a $\mathbf{K}_0$-orbit $N$, that is the basis of
a Mostow fibration of $M$ (see \cite{Most55, Most62}).
We have $N=M$ when $M$ is closed, but otherwise 
$N$ is a $\mathbf{K}_0$-homogeneous
$CR$ manifold with the same $CR$-codimension of $M$, but 
with a smaller $CR$-dimension.
From the structure Theorem \ref{thm:ha}
we obtain that $N$ is 
$\mathbf{K}_0$-equivariantly
a tower of fiber bundles, with basis $M_{\mathfrak{e}}$, and
fibers that are complex flag manifolds.
Since $N$ is a deformation retract of $M$, the
$\mathbf{K}_0$-equivariant fibration $N\to{M}_{\mathfrak{e}}$
yields information on the topology of $M$ in terms
of that of real and complex flag manifolds.
The latter
has been investigated by several
Authors, see e.g. \cite{BCF01, CS99, F71, T89, TK68}.
\par
We stress the fact that our methods and results 
are effective and permit to 
explicitly describe the objects involved and actually compute
some of their invariants. \par
Let us turn now to a more detailed description
of
the contents of this paper.
In the first section we collect 
some general
definitions and results on 
$CR$ manifolds $M$, homogeneous for the transitive action of a real
Lie group ${{\mathbf{G}_0}}$ of $CR$ automorphisms, and on their
corresponding $CR$-\emph{algebras}.
These objects 
were defined in
\cite{MN05}: they are pairs $(\mathfrak{g}_0,\mathfrak{q})$, where
$\mathfrak{g}_0$ is the \emph{real} Lie algebra
of ${{\mathbf{G}_0}}$, and $\mathfrak{q}$ a \emph{complex}
Lie subalgebra
of its complexification $\mathfrak{g}=\mathfrak{g}_0^{\mathbb{C}}$.
Having fixed a point $x_0\in{M}$, the elements of $\mathfrak{q}$
correspond to the left invariant elements of the lift
to $\mathbf{G}_0$, by
the principal fibering \mbox{${{\mathbf{G}_0}}\ni{g}\to g\cdot{x}_0\in{M}$},
of the complex tangent vector fields 
of type $(0,1)$ \mbox{on $M$}.
\par
The notions of finite type and holomorphic degeneracy of
$CR$ geometry 
translate, in the case of homogeneous $CR$ manifolds,
into fundamentality and weak degeneracy
of their $CR$ algebras (see Definition \ref{dfn:ac} and
Proposition \ref{prop:ac}).
When $M$ is not of finite type, or
is holomorphically degenerate, then it was shown in
\cite{MN05} that
there are nontrivial canonical ${\mathbf{G}_0}$-equivariant 
$CR$-fibrations $M\to{M}'$,
called the fundamental and the weakly
nondegenerate reductions, respectively. 
\par
We need to make a remark to fill a gap between the results of
\cite{MN05} and the statement of Proposition \ref{prop:ac}
of this article.
G. Fels (see \cite{fels06}) pointed out to us that
the notion of \emph{holomorphic degeneracy}
of \cite{BER99}, in the case of
homogeneous $CR$ manifolds,  coincides with
our notion of \emph{weak
degeneracy}.
This equivalence is used in
point (3) of Proposition \ref{prop:ac}.
\par
In this paper, since we consider orbits of real forms in complex
flag manifolds, we shall mostly restrict to
$CR$ algebras
$(\mathfrak{g}_0,\mathfrak{q})$ that consist
of a real form $\mathfrak{g}_0$ of a complex semisimple Lie algebra 
$\mathfrak{g}$ and of a complex parabolic Lie subalgebra
$\mathfrak{q}$ of
$\mathfrak{g}$. These were called \emph{parabolic $CR$ algebras}
in \cite{AMN06, MN05}.
\par
In \S\ref{sec:b} we collect some facts about complex flag manifolds,
for which we refer to \cite{Kn:2002, Wolf69}. 
We describe the Chevalley
decomposition of a complex parabolic subgroup 
$\mathbf{Q}$ of a semisimple
complex linear group $\mathbf{G}$ (Proposition \ref{prop:bc}),
to prepare for the
description of the isotropy subgroup
$\mathbf{I}_0$ of a $\mathbf{G}_0$-orbit $M$ in
\S\ref{sec:c}. There we obtain a suitable Chevalley decomposition of
$\mathbf{I}_0$ (Proposition \ref{prop:cb})  
from that of
its Lie algebra
$\mathfrak{i}_0$ (Proposition \ref{prop:ca}).
Note that, 
in general, $\mathbf{I}_0$ is not a parabolic subgroup
of $\mathbf{G}_0$, while this is true in the \textit{standard} case,
corresponding to the 
semisimple Levi-Tanaka algebras (see e.g. \cite{MN97, MN98}).
This Chevalley decomposition is fundamental to  
understanding the structure of the orbits. Our proof
strongly relies on the semialgebraic nature of the objects under
consideration, and on 
the properties of complex parabolic subgroups.
In \S\ref{sec:c} we also consider the
Cartan decomposition of the reductive factor of 
the Chevalley decomposition of $\mathbf{I}_0$, from which
we derive information about
the group of connected components of $\mathbf{I}_0$, 
in terms of its maximally noncompact Cartan subgroups
 (\mbox{Theorem \ref{thm:cf}}).
\par
The $\mathbf{G}_0$-orbits of the complex flag manifold $X$ of
$\mathbf{G}$ are completely described by their associated $CR$
algebras $(\mathfrak{g}_0,\mathfrak{q})$. 
Recall that a complex parabolic subalgebra $\mathfrak{q}$
of $\mathfrak{g}$ 
contains a 
Borel subalgebra $\mathfrak{b}$ of $\mathfrak{g}$, and hence
a Cartan subalgebra $\mathfrak{h}$
of $\mathfrak{g}$. Then $\mathfrak{q}$
is conveniently characterized in terms
of the root system $\mathcal{R}$ defined by $\mathfrak{h}$.
In fact, $\mathfrak{b}$ corresponds to a Weyl chamber $C$ of $\mathcal{R}$,
and the 
parabolic $\mathfrak{q}$'s containing
$\mathfrak{b}$
are in one-to-one correspondence with the subsets $\Phi$ of the basis
$\mathcal{B}$ of $C$-positive simple roots
(see Formula
\eqref{eq:bg}). Next, we observe that,
for every real form $\mathfrak{g}_0$ of $\mathfrak{g}$,
a complex parabolic subalgebra $\mathfrak{q}$ 
of $\mathfrak{g}$ always
contains a Cartan subalgebra $\mathfrak{h}_0$ of 
$\mathfrak{g}_0$.
It is convenient to employ
the root system
$\mathcal{R}$ of $\mathfrak{g}$ defined by
the complexification $\mathfrak{h}$
of $\mathfrak{h}_0$, and   
Weyl chambers $C$ of $\mathcal{R}$ yielding a
$\mathfrak{b}\subset\mathfrak{q}$.
They are called \emph{fit} for
$(\mathfrak{g}_0,\mathfrak{q})$ (Definition \ref{dfn:ba}). Then the
conjugation rules for the set $\mathcal{B}$
of simple $C$-positive roots, and the 
datum of the $\Phi\subset\mathcal{B}$ corresponding to~$\mathfrak{q}$,
encode all relevant information for the
orbit $M$ associated to $(\mathfrak{g}_0,\mathfrak{q})$.
\par
In \cite{AMN06} it was natural to
restrict to
maximally noncompact Cartan subalgebras $\mathfrak{h}_0$ of
$\mathfrak{g}_0$ and to fit Weyl chambers $C$
yielding
Satake diagrams of $\mathfrak{g}_0$ (see \cite{Ara62}). Here 
the consideration of \emph{general} orbits leads 
us first to
consider \emph{all} Cartan subalgebras 
$\mathfrak{h}_0$ of $\mathfrak{g}_0$, and then, 
after having fixed $\mathfrak{h}_0$, 
to
select, among the fit 
Weyl chambers $C$,
those for which $\mathcal{B}$ enjoys the best conjugation properties
(see Lemma \ref{lem:dd}). These are called
$S$-fit or $V$-fit, as 
they yield conjugation
rules for $\mathcal{B}$ that are \textit{as close as possible} to those 
of a Satake or a Vogan diagram, respectively (see Definition \ref{dfn:de}).
\par
An important feature of 
parabolic $CR$ manifolds 
is the fact
that to every $\mathfrak{g}_0$-equivariant
morphism of parabolic $CR$ algebras (see \cite{MN05})
always corresponds a smooth 
$\mathbf{G}_0$-equivariant  
fibration of the associated 
$\mathbf{G}_0$-orbits.
We prove indeed  that
the inclusions between the isotropy subgroups of the $\mathbf{G}_0$-orbits
are equivalent to those between their corresponding Lie subalgebras
(Theorem \ref{thm:ea}). Besides providing the
desired fibrations, this result also 
allows to consider
\emph{changes of
$CR$ structures}. This amounts to considering the given orbit $M$ 
simply as a homogeneous ${{\mathbf{G}_0}}$-manifold, and looking for
its possible realizations as a real submanifold of a
complex flag manifold of ${{\mathbf{G}}}$. In particular,
we can inquire about maximal and minimal 
${{\mathbf{G}_0}}$-homogeneous parabolic
$CR$ structures on $M$. In fact, 
one of these,
the \emph{weakening} of the $CR$ structure (see 
Definition \ref{dfn:ed}) 
will be one of the main ingredients in the proof of the
structure Theorem \ref{thm:ha}. These 
$\mathbf{G}_0$-equivariant fibrations are not, in general,
$CR$ maps. 
When they are 
$CR$, 
they are restrictions of $\mathbf{G}$-equivariant
fibrations of flag manifolds,
and their fibers $F$ have special features.
Indeed, $F$ consists 
of finitely many connected components, each being a 
copy of a Cartesian product $F'\times{F}''$, in which 
$F'$ is the orbit of a real form in a suitable complex
flag manifold $Y$,
and $F''$ a complex nilmanifold
(Theorem \ref{thm:eh}).
\par
The $S$-fit Weyl chambers are especially suited to discuss finite type
(see Theorem \ref{thm:fb}), while the $V$-fit Weyl chambers
are the proper setting for weak degeneracy
(see Proposition \ref{prop:fc} and Theorem \ref{thm:fd}). Thus the results of
\S\ref{sec:f} give a combinatoric way of constructing,
in the special case of the orbits of a real form in
a complex flag manifold, the 
basis of the fundamental
and weakly nondegenerate
reductions of a $CR$ algebra, that where discussed in general
in \cite[\S 5]{MN05},
and for the minimal orbits in \cite[\S 9,10,11]{AMN06}.
The basis of these reductions can indeed be
obtained by first representing the parabolic $\mathfrak{q}$
that characterizes $M$ by a subset $\Phi$ of $C$-positive simple roots
for an $S$- or $V$-fit Weyl chamber $C$, and then obtaining a new set
$\Psi\subset\Phi$ by dropping roots of $\Phi$ according to some
fairly easy
rules, that are stated in Theorems \ref{thm:fb}, \ref{thm:fd}, respectively.
These $\Psi$ give then the parabolic $\mathfrak{q}'$ for which the
$(\mathfrak{g}_0,\mathfrak{q}')$ are the basis of the desired reductions.\par
The construction of the fibration in
the structure 
Theorem \ref{thm:ha}
is obtained by alternating weakenings
of the
$CR$ structures, that
produce weak degeneracy, and the corresponding canonical
weakly nondegenerate reductions, that have
complex fibers
(see Theorem \ref{thm:fd}).
This process ends by yielding 
a fibering over
a real flag manifold $M_{\mathfrak{e}}$. 
An important feature is that, although
the fibering $M\to{M}_{\mathfrak{e}}$ may not be $CR$, at each step
the weakly nondegenerate reductions are $CR$ fibration, and therefore
we obtain a better control of the fibers (by Theorem \ref{thm:eh}).
Our approach is a substitute for the use of holomorphic arc components
in \cite[Theorem 8.15]{Wolf69}.
\par
If $M'\to{M}$ is a ${{\mathbf{G}_0}}$-equivariant
$CR$ fibering, with typical fiber $F$,
the number of connected components of $F$ is proved 
 to be equal to the quotient of the orders of
the analytic Weyl
groups of the reductive and of the semisimple part of the isotropy
$\mathbf{I}_0$
of the basis, computed  with respect to
a maximally noncompact Cartan
subgroup $\mathbf{H}_0$ of the reductive part of 
the isotropy ${{\mathbf{I}_0'}}$ of the total space $M'$
(Theorem \ref{thm:gc}).\par
The results of \S{\ref{sec:h}} and \S{\ref{sec:g}} are used,
within \S{\ref{sec:i}} and the following \S\ref{sec:j},
to study various topological properties of the ${{\mathbf{G}_0}}$-orbits
$M$ and of the basis $N$ of their Mostow fibrations.
In particular, we compute  the first homotopy and homology groups of
$M$ and $N$ (Theorem \ref{thm:ia}, Corollary \ref{cor:id}
and Theorem \ref{thm:ig}).  
In \S\ref{sec:j} we show that the basis $N$ of the
Mostow fibration $M\to{N}$ is a $CR$ manifold with the same
$CR$ codimension of $M$ (Proposition \ref{prop:jb}), and then
construct a $\mathbf{K}_0$-equivariant map $N\to{M}_{\mathfrak{e}}$
from $N$ to the real flag manifold
$M_{\mathfrak{e}}$ (the \emph{real core} of $M$ of 
\S{\ref{sec:h}}),
as a tower of fibrations, with fibers that are diffeomorphic to
complex flag manifolds
(Theorem \ref{thm:jc}). This allows, in Corollary \ref{cor:jd},
to relate the Euler-Poincar\'e characteristics of  
$X$, $M_{\mathfrak{e}}$ and $N$ (see \cite{AMN08} for related results
concerning the minimal orbits).\par
In the final section \S\ref{sec:k} we compare our construction of
the \emph{real core} $M_{\mathfrak{e}}$ of $M$ with the space
$M_a$ of its \emph{algebraic arc components},
introduced in \cite{Wolf69}.
In many cases, e.g. when $M$ is a closed orbit, we have $M_a=M_{\mathfrak{e}}$,
but it may happen for some $M$ that  $M_{\mathfrak{e}}$ and ${M}_a$
are not diffeomorphic. 
We point out that 
the $\mathbf{G}_0$-equivariant fibration $M\to{M}_{\mathfrak{e}}$
has always simply connected fibers, while the fibers 
of $M\to{M}_a$ may
not be simply connected.
\section{Homogeneous $CR$ manifolds and $CR$ algebras}
\label{sec:a}
Let $M$ be a smooth manifold.
A \emph{$CR$ structure} on $M$ is the datum of
a smooth complex subbundle $T^{0,1}M$ of its complexified tangent bundle
$T^{\mathbb{C}}M$, with $T^{0,1}M\cap\overline{T^{0,1}M}=0$, 
that is \emph{formally integrable}, i.e. satisfies:
\begin{equation}
  \label{eq:aa}
  [\Gamma(M,T^{0,1}M),\Gamma(M,T^{0,1}M)]\subset\Gamma(M,T^{0,1}M).
\end{equation}
The complex rank $n$ of $T^{0,1}M$ is called the 
$CR$-dimension, and $k=\mathrm{dim}_{\mathbb{R}}{M}-2n$ the
$CR$-codimension of $M$.
If $n=0$, we say that $M$ is totally real; if $k=0$,
$M$ is a complex 
manifold in view of the Newlander-Nirenberg theorem.\par
Let $M$ be a real submanifold of a complex manifold ${X}$.
For $x\in{M}$ set $T^{0,1}_xM=T^{0,1}_x{X}\cap{T}^{\mathbb{C}}_xM$.
When the dimension of $T^{0,1}_xM$ is independent of $x\in{M}$,
then $T^{0,1}M=\bigcup_{x\in{M}}T^{0,1}_xM$ is a formally integrable
complex subbundle of $T^{\mathbb{C}}M$, defining on $M$ the structure of a
\emph{$CR$ submanifold} of ${X}$.
If the complex dimension of ${X}$ is the sum of the
$CR$ dimension and the $CR$ codimension of $M$, we say that
the inclusion $M\hookrightarrow{X}$ is \emph{generic}.
\par
If $M$ and $M'$ are $CR$ manifolds, a smooth map $f:M'\to{M}$ is
$CR$ if $df^{\mathbb{C}}(T^{0,1}M')\subset{T}^{0,1}M$.
In an obvious way, we define the notions of $CR$ immersion, submersion
and diffeomorphism. In particular a smooth bundle $\pi:M'\to{M}$ for
which $\pi$ is a $CR$ submersion (i.e. $\pi$ is a $CR$ map and
$d\pi^{\mathbb{C}}(T^{0,1}M')=T^{0,1}M$) is called a \emph{$CR$ bundle}
or a \emph{$CR$ fibration}.
\begin{dfn} \label{dfn:aa}
Let ${{\mathbf{G}_0}}$ be a Lie group.
A ${{\mathbf{G}_0}}$-homogeneous $CR$ manifold is a
${{\mathbf{G}_0}}$-homogeneous smooth manifold endowed with a 
${{\mathbf{G}_0}}$-invariant
$CR$ structure.
\end{dfn}
Fix a point $x_0$ of a ${{\mathbf{G}_0}}$-homogeneous $CR$ manifold
${M}$, let $\mathbf{I}_0\subset{{\mathbf{G}_0}}$ 
be the isotropy subgroup of $x_0$ and
$\pi:{{\mathbf{G}_0}}\to{M}$ the 
corresponding principal $\mathbf{I}_0$-bundle. 
Denote by $\mathfrak{Z}({{\mathbf{G}_0}})$ the space of smooth
sections of the pullback by $\pi$ 
of 
$T^{0,1}M$ to ${{\mathbf{G}_0}}$, i.e.
the set of complex valued
vector fields $Z$ in ${{\mathbf{G}_0}}$ such that 
$d\pi^{\mathbb{C}}(Z_g)\in{T}_{\pi(g)}^{0,1}{M}$ for all $g\in{{\mathbf{G}_0}}$.
By \eqref{eq:aa}, also $\mathfrak{Z}({{\mathbf{G}_0}})$ is formally integrable,
i.e. 
$[\mathfrak{Z}({{\mathbf{G}_0}}),\mathfrak{Z}({{\mathbf{G}_0}})]
\subset\mathfrak{Z}({{\mathbf{G}_0}})$. Moreover, 
$\mathfrak{Z}({{\mathbf{G}_0}})$ is invariant by left translations in
${{\mathbf{G}_0}}$. Hence
its left invariant vector
fields generate $\mathfrak{Z}({{\mathbf{G}_0}})$
as a left $\mathcal{C}^{\infty}({{\mathbf{G}_0}},\mathbb{C})$-module.\par
Thus, denoting by $\mathfrak{g}$ the complexification of the
Lie algebra $\mathfrak{g}_0$ of
${{\mathbf{G}_0}}$, by
the formal integrability condition \eqref{eq:aa}, the subspace
\begin{equation}\label{eq:ab}
\mathfrak{q}=(d\pi^{\mathbb{C}})^{-1}(T^{0,1}_{x_0}M)\subset{\mathfrak{g}}
=T^{\mathbb{C}}_e{{\mathbf{G}_0}}
\end{equation} 
is a complex subalgebra of
$\mathfrak{g}$. We can summarize these observations by:
\begin{lem} 
Let $\mathfrak{i}_0$ be the Lie algebra of the isotropy
$\mathbf{I}_0$. 
Then \eqref{eq:ab} establishes a one-to-one correspondence
between 
the ${{\mathbf{G}_0}}$-homogeneous $CR$ structures on 
$M={{\mathbf{G}_0}}/\mathbf{I}_0$
and the complex Lie subalgebras
$\mathfrak{q}$ of $\mathfrak{g}$ such that
$\mathfrak{q}\cap\mathfrak{g}_0=\mathfrak{i}_0$.\qed
\end{lem}
This was our motivation to introduce and discuss 
\emph{$CR$ algebras} in \cite{MN05}. We rehearse some definitions.
\begin{dfn} \label{dfn:ad}
A \emph{$CR$ algebra} is a pair $(\mathfrak{g}_0,\mathfrak{q})$, consisting
of a real Lie algebra $\mathfrak{g}_0$ and of a complex Lie subalgebra
$\mathfrak{q}$ of its complexification $\mathfrak{g}$, such that
the quotient $\mathfrak{g}_0/(\mathfrak{q}\cap\mathfrak{g}_0)$ is a finite
dimensional real vector space.\par
If $M$ is a ${{\mathbf{G}_0}}$-homogeneous $CR$ manifold and $\mathfrak{q}$
is defined by
\eqref{eq:ab}, 
we say that the $CR$ algebra
$(\mathfrak{g}_0,\mathfrak{q})$ is \emph{associated} to $M$.
\end{dfn}
\begin{rmk}\label{rmk:ac}
The $CR$-dimension and $CR$-codimension of $M$ can be computed in terms
of its associated 
$CR$ algebra $(\mathfrak{g}_0,\mathfrak{q})$. We have indeed
\begin{align}
  CR\text{-}\mathrm{dim}\,{M}&=\mathrm{dim}_{\mathbb{C}}\mathfrak{q}-
\mathrm{dim}_{\mathbb{C}}(\mathfrak{q}\cap\bar{\mathfrak{q}}),\\
CR\text{-}\mathrm{codim}\,{M}&=\mathrm{dim}_{\mathbb{C}}\mathfrak{g}-
\mathrm{dim}_{\mathbb{C}}(\mathfrak{q}+\bar{\mathfrak{q}}).
\end{align}
\end{rmk}
The $CR$ algebra
$(\mathfrak{g}_0,\mathfrak{q})$, is said to be \emph{totally real} when
$CR\text{-}\mathrm{dim}(M)=0$,
\emph{totally complex} when
$CR\text{-}\mathrm{codim}(M)=0$. 
\begin{dfn}\label{dfn:ac}
A $CR$ algebra $(\mathfrak{g}_0,\mathfrak{q})$ is
called:
\begin{itemize}
\item \emph{fundamental} if there is no complex Lie subalgebra
$\mathfrak{q}'$ of $\mathfrak{g}$ with:
\begin{equation}
  \label{eq:ac}
  \mathfrak{q}+\bar{\mathfrak{q}}\subset\mathfrak{q}'\subsetneqq
\mathfrak{g};
\end{equation}
\item \emph{strictly}, or \emph{Levi-nondegenerate} if
\begin{equation}
\{Z\in\mathfrak{q}\,|\,\mathrm{ad}(Z)(\bar{\mathfrak{q}}) \subset
\mathfrak{q}+\bar{\mathfrak{q}}\}=\mathfrak{q}\cap\bar{\mathfrak{q}}.
\end{equation}
\item \emph{weakly degenerate} if there is a complex Lie subalgebra
$\mathfrak{q}'$ of $\mathfrak{g}$ with:
\begin{equation}
  \label{eq:ad}
  \mathfrak{q}\subsetneqq \mathfrak{q}'\subset\mathfrak{q}+\bar{\mathfrak{q}};
\end{equation}
\end{itemize}
\end{dfn}
We have:
\begin{prop}\label{prop:ac}
Let $M$ be a ${{\mathbf{G}_0}}$-homogeneous $CR$ manifold with associated
$CR$ algebra $(\mathfrak{g}_0,\mathfrak{q})$. Then:
\begin{enumerate}
\item $(\mathfrak{g}_0,\mathfrak{q})$ is fundamental if and only
if $M$ is of finite type in the sense of \cite{BG77}.\smallskip
\item[]
\qquad Assume that $M$ is of finite type in the sense of \cite{BG77}.
Then:
\item $(\mathfrak{g}_0,\mathfrak{q})$ is strictly nondegenerate
if and only if the vector valued Levi form 
(see e.g. \cite{MN00}) on $M$ is nondegenerate.
Strict nondegeneracy implies weak nondegeneracy.
\item $(\mathfrak{g}_0,\mathfrak{q})$ is weakly nondegenerate
if and only if the group of germs of $CR$ diffeomorphisms at $x_0\in{M}$
that stabilize $x_0$
is a finite dimensional Lie group, i.e. $M$ is holomorphically 
nondegenerate
(see e.g.~\cite{BER99}, \cite{fels06}).
\item $(\mathfrak{g}_0,\mathfrak{q})$ is weakly degenerate
if and only if there exists a local 
$\mathbf{G}_0$-equivariant $CR$ fibration $M\to{M}'$, with
nontrivial complex fibers.
\end{enumerate}
\end{prop}
\begin{proof}
  All the statements, except (3), were proved in \cite[Proposition
  13.3]{MN05}
and \cite[Proposition 4.1]{AMN06}. 
Item (3) follows from the observation in \cite{fels06}
  that, for homogeneous $CR$ manifolds, weak nondegeneracy is
  equivalent to the finite nondegeneracy of \cite{fels06}, which is in
  turn equivalent to the holomorphic nondegeneracy of \cite{BER99}.
\end{proof}
\section{Complex flag manifolds}\label{sec:b}
In this section we collect some 
notions on complex flag manifolds.\par
A {\it complex flag manifold} is the quotient
${X}=\mathbf{G}/\mathbf{Q}$ of a 
connected
complex semisimple Lie group
$\mathbf{G}$ by a \emph{parabolic} subgroup $\mathbf{Q}$.
We recall that 
$\mathbf{Q}$ is {parabolic} in $\mathbf{G}$
if and only if
its Lie algebra $\mathfrak{q}$ is a parabolic Lie subalgebra of 
 the Lie algebra $\mathfrak{g}$ of
$\mathbf{G}$. This means that $\mathfrak{q}$
contains a Borel subalgebra, i.e. a maximal solvable Lie
subalgebra of $\mathfrak{g}$. We also note that $\mathbf{G}$ 
is necessarily
a linear group, and that $\mathbf{Q}$ is connected, contains the center
of $\mathbf{G}$ and equals the normalizer of $\mathfrak{q}$ in
$\mathbf{G}$\,:
\begin{equation}\label{eq:ba}
\mathbf{Q}\,=\,\left.
\left\{g\in\mathbf{G}\,\right|\, \mathrm{Ad}(g)
(\mathfrak{q})=\mathfrak{q}\right\}\,.
\end{equation}
The isotropy subgroup of a point $x=g\mathbf{Q}$ of ${X}$ is
$\mathrm{ad}(g)(\mathbf{Q})$. Since $\mathbf{Q}$ is its own normalizer
in $\mathbf{G}$, we can identify 
$x$ with $\mathrm{ad}(g)(\mathbf{Q})$.
Moreover, since $\mathbf{Q}$ is connected, 
$\mathrm{ad}(g)(\mathbf{Q})$ is completely determined by
its Lie algebra $\mathrm{Ad}(g)(\mathfrak{q})$. Thus ${X}$
can be viewed as the space of parabolic subalgebras of 
$\mathfrak{g}$ that are $\mathrm{Ad}(\mathbf{G})$-conjugate to
$\mathfrak{q}$. Hence  we obtain an isomorphic
${X}$ by substituting $\mathbf{Int}_{\mathbb{C}}(\mathfrak{g})$
to $\mathbf{G}$.
\par
In particular, a different choice of 
a connected $\mathbf{G}'$
and of a parabolic $\mathbf{Q}'$, 
with Lie algebras $\mathfrak{g}'$ and
$\mathfrak{q}'$ isomorphic to $\mathfrak{g}$ and
$\mathfrak{q}$, yields a complex flag manifold $X'$ that is
complex-projectively isomorphic to ${X}$.
Thus a flag manifold ${X}$ is 
completely described by the
pair consisting of the Lie
algebras~$\mathfrak{g}$ and~${\mathfrak{q}}$. \par
Fix a Cartan subalgebra 
$\mathfrak{h}$ of
$\mathfrak{g}$, contained in $\mathfrak{q}$, and
let $\mathcal{R}=\mathcal{R}(\mathfrak{g},\mathfrak{h})$ be the corresponding
root system. For each $\alpha\in\mathcal{R}$, let 
$\mathfrak{g}^{\alpha}=\{Z\in\mathfrak{g}\,|\, [H,Z]=\alpha(H)Z,
\; 
\forall{H}\in\mathfrak{h}\}$ be
the root subspace of $\alpha$.
Denote by
$\mathfrak{h}_{\mathbb{R}}$ the real subspace of $\mathfrak{h}$
on which the roots are real valued. The choice of a Weyl chamber
$C\subset\mathfrak{h}_{\mathbb{R}}$ defines a partial
order\, 
$\prec$\, in the dual space 
$\mathfrak{h}^*_{\mathbb{R}}$
of $\mathfrak{h}_{\mathbb{R}}$.
Since $\mathfrak{q}$ contains $\mathfrak{h}$, it is the direct
sum of $\mathfrak{h}$ and of the root spaces contained in
$\mathfrak{q}$,
i.e.
\begin{equation}
  \mathfrak{q}= \mathfrak{h} + \sum_{\alpha\in\mathcal{Q}}\mathfrak{g}^{\alpha},
\quad\text{where}\quad
  \mathcal{Q}= \{\alpha\in\mathcal{R} \mid
  \mathfrak{g}^{\alpha}\subset\mathfrak{q}\}.   
\end{equation}
Then the fact that $\mathfrak{q}$ is parabolic means that one can choose
$C$ in such a way that:
\begin{equation}\label{eq:bb}
\alpha\in\mathcal{Q}\quad\text{for all}\quad
\alpha\succ0.
\end{equation}
The set
$\mathcal{Q}$ is {\it parabolic}, i.e. 
is closed under root addition
and \mbox{$\mathcal{Q}\cup\left(-\mathcal{Q}\right)=\mathcal{R}$}.
\begin{dfn} \label{dfn:ba}
A Weyl chamber $C$ for which \eqref{eq:bb} holds true is called
\emph{fit} for $\mathcal{Q}$. We denote by $\mathfrak{C}(\mathcal{R})$
the set of all Weyl chambers for $\mathcal{R}$ and by
$\mathfrak{C}(\mathcal{R},\mathcal{Q})$ the subset of those that
are fit for $\mathcal{Q}$.
\end{dfn}
Let $C\in\mathfrak{C}(\mathcal{R})$
and $\mathcal{B}=\mathcal{B}(C)$ be the corresponding system of 
$C$-positive simple roots.
All $\alpha\in\mathcal{R}$ are linear combinations of elements 
of the basis
${\mathcal{B}}$\,:
\begin{equation}\label{eq:bc}
\alpha=\sum_{\beta\in{\mathcal{B}}}{k_{\alpha}^{\beta}\beta}\,,\quad 
k_{\alpha}^{\beta}\in\mathbb{Z}.
\end{equation}
\begin{dfn}\label{dfn:bb}
We define 
the \emph{support} $\mathrm{supp}(\alpha)$ 
of $\alpha$ with respect to ${\mathcal{B}}$ as the set
of $\beta\in{\mathcal{B}}$ for which $k^{\beta}_{\alpha}\neq 0$.
\end{dfn}
Having fixed $C\in\mathfrak{C}(\mathcal{R},\mathcal{Q})$, 
we associate to
$\mathfrak{q}$ the subset $\Phi$ of $\mathcal{B}$,
consisting of the 
simple $C$-positive roots $\alpha$ for which 
${\mathfrak{g}}^{-\alpha}\not\subset\mathfrak{q}$. \par
Then $\mathcal{Q}$ and
$\mathfrak{q}$ are completely determined by $\Phi$. Indeed:
\begin{align} \label{eq:bd}
&&\mathcal{Q}\; &=\mathcal{Q}_{\Phi}=\{\alpha\! \succ\!  0\}
\cup\{\alpha \!\prec\!  0\mid
\mathrm{supp}(\alpha)\!
\cap\!\Phi\!=\!\emptyset\}=\mathcal{Q}^n_{\Phi}\cup\mathcal{Q}^r_{\Phi},
\;\text{with:}\\
\label{eq:bf}
&&\mathcal{Q}^n&=\mathcal{Q}_{\Phi}^n=\{\alpha\in\mathcal{R}\,|\,
\alpha \succ  0\;\mathrm{and}\;\mathrm{supp}(\alpha)
\cap\Phi\neq\emptyset\},\\
\label{eq:be}
&&\mathcal{Q}^r&=\mathcal{Q}_{\Phi}^r=\{\alpha\in\mathcal{R}\,|\,
\mathrm{supp}(\alpha)
\cap\Phi=\emptyset\},\\
\intertext{and for the parabolic 
subalgebra $\mathfrak{q}$ we have the decomposition:}\label{eq:bg}
&&\mathfrak{q}\; &=\mathfrak{q}_{\Phi}={\mathfrak{h}}+{\sum}_{\alpha
\in\mathcal{Q}_{\Phi}}{{\mathfrak{g}}^{\alpha}}=\mathfrak{q}^r
\oplus\mathfrak{q}^n,\quad\text{where:}\\\label{eq:bi}
&&\mathfrak{q}^n&=\mathfrak{q}^n_{\Phi}={\sum}_{\alpha\in\mathcal{Q}^n_{\Phi}}
{\mathfrak{g}^{\alpha}}\qquad
\text{is the nilradical of $\mathfrak{q}$ and }\\
\label{eq:bj}
&&\mathfrak{q}^r&=\mathfrak{q}^r_{\Phi}={\mathfrak{h}}+\!\!
{\sum}_{\alpha\in\mathcal{Q}^r_{\Phi}}
{\mathfrak{g}^{\alpha}}\;\,\text{is a reductive
complement of $\mathfrak{q}^n_{\Phi}$ in $\mathfrak{q}_{\Phi}$}.\\
\intertext{We explicitly note
that $\Phi$ is related to $\mathcal{Q}$ and $C$ by}
\label{eq:bfa}
&&\Phi&=\mathcal{Q}^n\cap\mathcal{B}.
\end{align}
\par\smallskip
All Cartan subalgebras ${\mathfrak{h}}$
of ${\mathfrak{g}}$ are equivalent, modulo 
inner automorphisms. After ${\mathfrak{h}}$, and
hence $\mathcal{R}$, have been  fixed,
all basis of simple roots of
$\mathcal{R}$ are equivalent for the
transpose of inner automorphisms of ${\mathfrak{g}}$ normalizing
${\mathfrak{h}}$. Thus, after picking a Weyl chamber $C$, and
having fixed in this way a system $\mathcal{B}$ of $C$-positive simple
roots,
the correspondence $\Phi\leftrightarrow
\mathfrak{q}_{\Phi}$ is one-to-one between subsets $\Phi$ of 
${\mathcal{B}}$ and 
complex parabolic Lie subalgebras
of ${\mathfrak{g}}$, modulo inner automorphisms.
In other words, the subsets $\Phi$ of our fixed $\mathcal{B}$
parametrize all different flag manifolds $X$ of 
a connected semisimple
complex Lie group $\mathbf{G}$ with Lie algebra ${\mathfrak{g}}$.
\par
The choice of a Cartan subalgebra $\mathfrak{h}$ of
$\mathfrak{g}$ contained in $\mathfrak{q}$ yields a canonical
Chevalley decomposition of the parabolic subgroup $\mathbf{Q}$\,:
\begin{prop}\label{prop:bc}
Let $\mathbf{Q}$ be a parabolic subgroup of $\mathbf{G}$,
corresponding to the complex parabolic
Lie subalgebra $\mathfrak{q}$ of $\mathfrak{g}$.
With the notation above, we have a Chevalley decomposition\,:
\begin{equation}\label{eq:bk}
\mathbf{Q}=\mathbf{Q}^n\rtimes\mathbf{Q}^r
\end{equation}
where the unipotent radical
$\mathbf{Q}^n$ is the connected and
simply connected Lie subgroup 
of ${{{\mathbf{G}}}}$ with Lie algebra $\mathfrak{q}^n$
given by \eqref{eq:bi},
and $\mathbf{Q}^r$ is 
the reductive\footnote{According to \cite{Kn:2002} we call
reductive a linear Lie group ${{\mathbf{G}_0}}$, 
having finitely many connected components,
with 
a reductive Lie algebra
$\mathfrak{g}_0$, 
and such that,
for the complexification $\mathfrak{g}$ of $\mathfrak{g}_0$, we have
 $\mathrm{Ad}_{\mathfrak{g}}(\mathbf{G}_0)
\subset\mathbf{Int}_{\mathbb{C}}(\mathfrak{g})$.
While stating here the
reductiveness of $\mathbf{Q}^r$,
we consider the given group $\mathbf{G}$ as a real linear group.}
complement of $\mathbf{Q}^n$ in $\mathbf{Q}$, whose
Lie algebra $\mathfrak{q}^r$ is given by \eqref{eq:bj}.\par
Let
$\mathfrak{c}\subset\mathfrak{h}$ 
be the center of $\mathfrak{q}^r$:
\begin{equation}
  \label{eq:bl}
  \mathfrak{c}=\mathfrak{z}({\mathfrak{q}}^r)=
\left\{H\in\mathfrak{h}\mid
\mathrm{ad}(H)(\mathfrak{q}^r)=0\right\}.
\end{equation}
Then the reductive complement $\mathbf{Q}^r$ is characterized by\,:
\begin{equation}\label{eq:bm}
\mathbf{Q}^r
=\mathbf{Z}_{\mathbf{G}}
\left(\mathfrak{c}\right)=
\left\{g\in\mathbf{G}\mid
\mathrm{Ad}(g)(H)=H\quad\forall{H}\in
\mathfrak{c}\right\} .
\end{equation}
\par
Moreover,
$\mathbf{Q}^r$ is a subgroup of finite index in
the normalizer 
$\mathbf{N}_{\mathbf{G}}(\mathfrak{q}^r)$
of $\mathfrak{q}^r$ in $\mathbf{G}$,
and
$\mathbf{Q}\cap\mathbf{N}_{\mathbf{G}}(\mathfrak{q}^r)=
\mathbf{Q}^r$.
\end{prop}
\begin{proof}
A complex parabolic subgroup can also be considered as a \emph{real}
parabolic subgroup. The Chevalley decomposition  \eqref{eq:bk}
reduces then to the Langlands decomposition 
$\mathbf{Q}=\mathbf{MAN}$, with $\mathbf{N}=\mathbf{Q}^n$ and
$\mathbf{MA}=\mathbf{Q}^r$. Thus our statement is a consequence of
\cite[Proposition 7.82(a)]{Kn:2002}. \par
Note that $\mathfrak{q}^r$ is the centralizer of
$\mathfrak{c}$ 
in $\mathfrak{g}$ and is its own
normalizer. This yields the inclusion
$\mathbf{Q}^r\subset
\mathbf{N}_{\mathbf{G}}(\mathfrak{q}^r)$.
Since $\mathbf{N}_{\mathbf{G}}(\mathfrak{q}^r)$ is
semi-algebraic, it has finitely many connected components.
Thus its intersection with $\mathbf{Q}^n$ is discrete and
finite, and thus trivial because $\mathbf{Q}^n$ is connected,
simply connected and unipotent. 
\end{proof}
\section{${{\mathbf{G}_0}}$-orbits and their isotropy subgroups}
\label{sec:c}
We keep the notation of \S\ref{sec:b}. Let
$X=\mathbf{G}/\mathbf{Q}$ be a complex flag manifold, and
${{\mathbf{G}_0}}$ a connected real form of  $\mathbf{G}$.
Note that ${{\mathbf{G}_0}}$ is semi-algebraic, being a connected
component
of an algebraic group for the Euclidean topology.
We know from \cite{Wolf69} that there are 
finitely many ${{\mathbf{G}_0}}$-orbits
in ${X}$.
Fix any such orbit $M$ and a point $x\in{M}$. We can assume that
$\mathbf{Q}\subset\mathbf{G}$ is the stabilizer of $x$ 
for the action of $\mathbf{G}$ in ${X}$. 
Let
$\mathbf{I}_0=\mathbf{Q}\cap{{\mathbf{G}_0}}$ be the stabilizer of $x$ in
${{\mathbf{G}_0}}$, so that $M\simeq{{\mathbf{G}_0}}/\mathbf{I}_0$, as
a smooth manifold. 
The Lie algebra
$\mathfrak{g}_0$  of ${{\mathbf{G}_0}}$
is a real form of $\mathfrak{g}$ and
the Lie algebra of the isotropy subgroup $\mathbf{I}_0$
is the intersection
$\mathfrak{i}_0=\mathfrak{q}\cap\mathfrak{g}_0$,
where $\mathfrak{q}$ is the Lie algebra of $\mathbf{Q}$. \par
Since $\mathbf{I}_0$ always contains the center of $\mathbf{G}_0$,
the orbit $M$ actually only depends 
on the pair of Lie subalgebras $\mathfrak{g}_0$ 
and $\mathfrak{q}$ of $\mathfrak{g}$. 
\begin{dfn} \label{dfn:ca}
A pair $(\mathfrak{g}_0,\mathfrak{q})$, 
consisting of a 
semisimple real Lie algebra $\mathfrak{g}_0$ and of a complex
parabolic Lie subalgebra $\mathfrak{q}$ of its complexification
$\mathfrak{g}$ is a \emph{parabolic $CR$ algebra}.
The corresponding ${{\mathbf{G}_0}}$-orbit $M$ in 
${X}=\mathbf{G}/\mathbf{Q}$
is said to be a \emph{parabolic $CR$ manifold}.
\end{dfn}
We summarize the results of
\cite[p.491]{AMN06} by stating the following\,:
\begin{prop}[Decomposition of the isotropy subalgebra] \label{prop:ca}
Let $(\mathfrak{g}_0,\mathfrak{q})$ be a parabolic $CR$ algebra,
and \, $\mathfrak{i}_0=\mathfrak{q}\cap\mathfrak{g}_0$ the corresponding
isotropy subalgebra.
Then
$\mathfrak{i}_0$ 
contains a Cartan subalgebra $\mathfrak{h}_0$ of
$\mathfrak{g}_0$. \par
If $\mathfrak{h}_0$ is any Cartan subalgebra of $\mathfrak{g}_0$
contained in $\mathfrak{i}_0$, there is a Cartan involution
$\vartheta:\mathfrak{g}_0\to\mathfrak{g}_0$, 
with $\vartheta(\mathfrak{h}_0)=\mathfrak{h}_0$,
yielding 
a decomposition
\begin{equation}\label{eq:ca}
\mathfrak{i}_0=\mathfrak{n}_0\oplus\mathfrak{l}_0=\mathfrak{n}_0\oplus
\mathfrak{s}_0\oplus\mathfrak{z}_0
\end{equation} 
such that\,:
\begin{enumerate}
\item
$\mathfrak{n}_0$ is the nilpotent ideal 
of $\mathfrak{i}_0$, consisting of the elements $Y\in\mathfrak{i}_0$
for which $\mathrm{ad}_{\mathfrak{g}_0}(Y)$
is nilpotent;
\item
$\mathfrak{l}_0=\mathfrak{s}_0\oplus\mathfrak{z}_0$ is reductive;
\item
$\mathfrak{z}_0\subset\mathfrak{h}_0$
is the center of $\mathfrak{l}_0$ and
$\mathfrak{s}_0=[\mathfrak{l}_0,\mathfrak{l}_0]$ its semisimple ideal;
\item $\mathfrak{n}_0=[\mathfrak{z}_0,\mathfrak{n}_0]=
[\mathfrak{z}_0,\mathfrak{i}_0]$;
\item $\mathfrak{l}_0=\mathfrak{i}_0\cap\vartheta(\mathfrak{i}_0)$
is a $\vartheta$-invariant Lie subalgebra of $\mathfrak{g}_0$,
and also
$\mathfrak{z}_0$ and $\mathfrak{s}_0$ are $\vartheta$-invariant;
\item $\mathfrak{s}_0$ and $\mathfrak{z}_0$ are orthogonal
for the Killing form of $\mathfrak{g}_0$; 
\item the complexifications $\mathfrak{n}$ of $\mathfrak{n}_0$ 
and $\mathfrak{l}$ of $\mathfrak{l}_0$ are:
\begin{equation}
  \label{eq:ck}
  \mathfrak{n}=\mathfrak{q}^n\cap\bar{\mathfrak{q}}+
\bar{\mathfrak{q}}^n\cap\mathfrak{q},\quad \mathfrak{l}=\mathfrak{q}^r\cap
\bar{\mathfrak{q}}^r.\qed
\end{equation}
\end{enumerate}
\end{prop}
\begin{dfn}\label{dfn:cc}
Let $(\mathfrak{g}_0,\mathfrak{q})$
be a parabolic $CR$ algebra. A pair $(\vartheta,\mathfrak{h}_0)$,
consisting of a Cartan involution $\vartheta$ of $\mathfrak{g}_0$
and of a $\vartheta$-invariant
Cartan subalgebra $\mathfrak{h}_0$ of $\mathfrak{i}_0$
will be said to be
a Cartan pair \emph{adapted} to $(\mathfrak{g}_0,\mathfrak{q})$.
\end{dfn}
We have the following\,:
\begin{prop}[Chevalley decomposition of the isotropy]\label{prop:cb}
Keep the notation introduced above.
The isotropy subgroup $\mathbf{I}_0$ is the closed real
semi-algebraic
subgroup of ${{\mathbf{G}_0}}$\,:
\begin{equation}\label{eq:cb}
\mathbf{I}_0=\mathbf{N}_{{{\mathbf{G}_0}}}(\mathfrak{q})=\{
g\in{{\mathbf{G}_0}}\mid
\mathrm{Ad}_{\mathfrak{g}}(g)(\mathfrak{q})=
\mathfrak{q}\}\,.\end{equation}
The isotropy subgroup $\mathbf{I}_0$ admits a Chevalley decomposition
\begin{equation}\label{eq:cc}
\mathbf{I}_0=\mathbf{L}_0\ltimes\mathbf{N}_0
\end{equation}
where\,:
\begin{enumerate}
\item
$\mathbf{N}_0$ is a unipotent, closed, connected, and simply connected
normal
subgroup of $\mathbf{I}_0$, with Lie algebra $\mathfrak{n}_0$;
\item
$\mathbf{L}_0$ is a reductive Lie subgroup,  
with Lie algebra $\mathfrak{l}_0$,
and is the centralizer of $\mathfrak{z}_0$ in ${{\mathbf{G}_0}}$\,:
\begin{equation}\label{eq:cd}
\mathbf{L}_0=\mathbf{Z}_{{{\mathbf{G}_0}}}(\mathfrak{z}_0)=
\{g\in{{\mathbf{G}_0}}\mid \mathrm{Ad}_{\mathfrak{g}_0}(g)(H)=
H\quad\forall{H}\in\mathfrak{z}_0\}\, .
\end{equation}
\end{enumerate}
\end{prop} 
\begin{proof} Formula \eqref{eq:cb} holds because $\mathbf{Q}$ is 
the normalizer of $\mathfrak{q}$ in $\mathbf{G}$.\par
Fix a Cartan pair $(\vartheta,\mathfrak{h}_0)$ adapted to
$(\mathfrak{g}_0,\mathfrak{q})$. We consider the complexification
$\mathfrak{h}$ of $\mathfrak{h}_0$ and use the notation
of \eqref{eq:bd}-\eqref{eq:bj}. Let us
construct a parabolic $\mathbf{Q}'\subset\mathbf{Q}$ 
whose reductive part ${\mathbf{Q}'}^r$ is  a complexification
of $\mathbf{L}_0$. To this aim we set
\begin{equation*}
\mathfrak{q}'=\mathfrak{q}^n
\oplus\left(\mathfrak{q}^r\cap\bar{\mathfrak{q}}\right).
\end{equation*}
We claim that
$\mathfrak{q}'$  is a parabolic Lie subalgebra
of $\mathfrak{g}$. 
It contains the Cartan subalgebra $\mathfrak{h}$, because 
$\mathfrak{h}\subset\mathfrak{q}^r
\cap\bar{\mathfrak{q}}^r$. Since $\mathfrak{q}'$
is clearly a complex Lie subalgebra of $\mathfrak{q}$, 
to show that $\mathfrak{q}'$
is parabolic it suffices to verify that, for each $\alpha\in\mathcal{R}$
(the root system of $\mathfrak{g}$ with respect to
$\mathfrak{h}$)
either $\mathfrak{g}^{\alpha}$ or
$\mathfrak{g}^{-\alpha}$ is contained in $\mathfrak{q}'$.
This is obvious if $\{\alpha,-\alpha\}\cap\mathcal{Q}^n\neq
\emptyset$. We need only to consider roots $\alpha$ with
$\pm\alpha\in\mathcal{Q}^r$, and then observe that
either  
$\mathfrak{g}^{\alpha}$ or
$\mathfrak{g}^{-\alpha}$ is contained in 
$\bar{\mathfrak{q}}$, because $\bar{\mathfrak{q}}$ is parabolic
and contains $\mathfrak{h}$. Moreover, one easily verifies
that
\begin{equation*}
  \mathfrak{q}'{}^{\,r}=\bar{\mathfrak{q}}'{}^{\,r}\quad\text{and}\quad
\mathfrak{q}'=\mathfrak{q}^n+\left(\mathfrak{q}\cap
\bar{\mathfrak{q}}\right).
\end{equation*}
Hence, the center $\mathfrak{c}'$ of ${\mathfrak{q}'}^{\,r}$
coincides with
the complexification $\mathfrak{z}$ of $\mathfrak{z}_0$.
\par\smallskip
If $g\in\mathbf{I}_0$, then $\mathrm{Ad}_{\mathfrak{g}_0}(g)(\mathfrak{l}_0)$
is a reductive complement of $\mathfrak{n}_0$ in $\mathfrak{i}_0$.
Since all reductive complements of $\mathfrak{n}_0$ 
in $\mathfrak{i}_0$ 
are conjugated
by an inner automorphism from 
$\mathrm{Ad}_{\mathfrak{i}_0}(\mathbf{N}_0)$, we can find a $g_n\in\mathbf{N}_0$
such that $\mathrm{Ad}_{\mathfrak{i}_0}(g_n^{-1}g)(\mathfrak{l}_0)=
\mathfrak{l}_0$. Consider the element $g_r=g_n^{-1}g$. 
We have\,:
\begin{align*} 
&\mathrm{Ad}_{\mathfrak{g}_0}(g_r)(\mathfrak{g}_0)=\mathfrak{g}_0,&
\mathrm{Ad}_{{\mathfrak{g}}}(g_r)(\mathfrak{q})&=
\mathfrak{q},\\
&\mathrm{Ad}_{\mathfrak{g}}(g_r)(\mathfrak{q}^n)=
\mathfrak{q}^n ,&
\mathrm{Ad}_{\mathfrak{g}}(g_r)(\bar{\mathfrak{q}})&=
\bar{\mathfrak{q}},
\end{align*}
because $g_r\in\mathbf{Q}\cap\bar{\mathbf{Q}}$.
Thus $\mathrm{Ad}_{\mathfrak{g}}(g_r)(\mathfrak{q}')=\mathfrak{q}'$.
Moreover, $\mathrm{Ad}_{\mathfrak{g}}(g_r)({\mathfrak{q}^r})=
\mathfrak{q}^r$
and 
$\mathrm{Ad}_{\mathfrak{g}}(g_r)(\bar{\mathfrak{q}}^r)=
\bar{\mathfrak{q}}^r$. Since ${\mathfrak{q}'}^{\,r}=\mathfrak{q}^r\cap
\bar{\mathfrak{q}}^r$, we obtain that
$\mathrm{Ad}_{\mathfrak{g}}(g_r)({\mathfrak{q}'}^{\,r})=
{\mathfrak{q}'}^{\,r}$. Hence, 
 by Proposition \ref{prop:bc}, 
$g_r$ belongs to ${\mathbf{Q}'}^r\cap{\mathbf{I}_0}$,
and the statement follows.
Note that in fact we obtain
$\mathbf{L}_0={\mathbf{Q}'}^r\cap\mathbf{G}_0$,
so that \eqref{eq:cd}
is a consequence of Proposition \ref{prop:bc}.
\end{proof}
\begin{cor} \label{cor:ce}
We keep the notation and the assumptions 
of Propositions \ref{prop:ca} and \ref{prop:cb}. 
Let $(\vartheta,\mathfrak{h}_0)$ be a Cartan pair adapted to
$(\mathfrak{g}_0,\mathfrak{q})$ and
\begin{equation}\label{eq:ce}
\mathfrak{g}_0=\mathfrak{k}_0\oplus\mathfrak{p}_0
\quad\text{and}\quad
{{\mathbf{G}_0}}=\mathbf{K}_0\times\exp({\mathfrak{p}}_0)
\end{equation} 
the Cartan decompositions of $\mathfrak{g}_0$ and $\mathbf{G}_0$,
respectively,
corresponding to $\vartheta$. Then\,:
\begin{align}\label{eq:cf}
&\mathfrak{l}_0=\mathfrak{k}_{00}\oplus\mathfrak{p}_{00}\,,
&&\text{with}\quad
\mathfrak{k}_{00}=\mathfrak{k}_0\cap\mathfrak{l}_0\,,\quad
\mathfrak{p}_{00}=\mathfrak{p}_0\cap\mathfrak{l}_0\,,\\
\label{eq:cg}
&{{\mathbf{L}_0}}=\mathbf{K}_{00}\times\exp(\mathfrak{p}_{00})\,,
&&\text{with}\quad
\mathbf{K}_{00}=\mathbf{K}_0\cap\mathbf{L}_0\,.
\end{align}
\par
Let $\mathbf{S}_0$ be the analytic Lie subgroup of ${{\mathbf{G}_0}}$
generated by $\mathfrak{s}_0=[\mathfrak{l}_0,\mathfrak{l}_0]$. 
Then $\mathbf{S}_0$ is a closed Lie subgroup of $\mathbf{L}_0$
and has a finite center.
\end{cor}
\begin{proof}
The fact that
the Lie subgroup $\mathbf{L}_0$ 
is reductive in the sense of \cite{Kn:2002} 
and the validity of the
decompositions \eqref{eq:cf} and \eqref{eq:cg}
are straightforward
consequences of \cite[Proposition 7.25]{Kn:2002},
because of the characterization of $\mathbf{L}_0$ 
given in \eqref{eq:cd} of Proposition \ref{prop:cb}.\par
The last statement follows because
$\mathbf{S}_0$ is an analytic subgroup of a linear group
and has a semisimple Lie algebra
(see e.g. \cite[Proposition 7.9]{Kn:2002}).
\end{proof}
We conclude this section by proving a theorem about the set
of connected components of the isotropy.
\begin{thm}\label{thm:cf}
Let $M$ be a $\mathbf{G}_0$-orbit, 
with corresponding parabolic
$CR$ algebra $(\mathfrak{g}_0,\mathfrak{q})$. Let
$(\vartheta,\mathfrak{h}_0)$ be an adapted Cartan pair. 
We keep the notation of Propositions \ref{prop:ca},
\ref{prop:cb} and Corollary \ref{cor:ce}.\par
Then the Cartan subgroup
\begin{equation}\label{eq:ch}
\mathbf{H}_0
=\mathbf{Z}_{{{\mathbf{G}_0}}}(\mathfrak{h}_0)=\{g\in{{\mathbf{G}_0}}\,|\,
\mathrm{Ad}_{\mathfrak{g}_0}(g)(H)=H,\quad\forall{H}\in\mathfrak{h}_0\}
\end{equation}
of ${{\mathbf{G}_0}}$
is contained in $\mathbf{L}_0$. 
\par
Let 
$\mathfrak{h}_0$ be maximally noncompact 
in $\mathfrak{i}_0=\mathfrak{q}\cap\mathfrak{g}_0$.
Then the intersection
$\mathfrak{t}_0=\mathfrak{h}_0\cap\mathfrak{s}_0$ is a maximally noncompact
Cartan subalgebra of $\mathfrak{s}_0$. The corresponding
Cartan subgroup $\mathbf{T}_0$ of $\mathbf{S}_0$ is given by\,:
\begin{equation}
\mathbf{T}_0=\mathbf{Z}_{\mathbf{S}_0}(\mathfrak{t}_0)=
\mathbf{Z}_{\mathbf{S}_0}(\mathfrak{h}_0).
\end{equation}
Denoting by $\pi_0$ the group of
connected components, we obtain\,:
\begin{equation}\label{eq:ci}
\pi_0(\mathbf{I}_0)\simeq\pi_0(\mathbf{K}_{00})\simeq
\pi_0(\mathbf{L}_0)
\simeq\pi_0\left(\frac{\;\mathbf{H}_0\;}{\;\mathbf{T}_0}\right)
\simeq \frac{\mathbf{H}_0\;\;}{\mathbf{T}_0\mathbf{H}^0_0}\,,
\end{equation}
where $\mathbf{H}^0_0$ is the connected component of the identity
of $\mathbf{H}_0$. In particular, 
$\pi_0(\mathbf{I}_0)$ is an Abelian group.
\end{thm}
\begin{proof} From
$\mathfrak{z}_0\subset\mathfrak{h}_0$, we have
$\mathbf{H}_0=\mathbf{Z}_{{{\mathbf{G}_0}}}(\mathfrak{h}_0)\subset
\mathbf{Z}_{{{\mathbf{G}_0}}}(\mathfrak{z}_0)=\mathbf{L}_0$. \par
Assume that $\mathfrak{h}_0$ is maximally noncompact in $\mathfrak{i}_0$.
We know, see e.g. \cite[Proposition 7.90 (a)]{Kn:2002},
that $\mathbf{H}_0$ intersects all connected components of $\mathbf{L}_0$
and hence of $\mathbf{I}_0$ and of $\mathbf{K}_{00}$.
Take $h_0\in\mathbf{H}_0$ in the connected component $\mathbf{I}_0^0$
of the identity of $\mathbf{I}_0$. We 
write $h_0=k_0\exp(X_0)$ with
$k_0\in\mathbf{K}_{00}$ and $X_0\in\mathfrak{p}_{00}$. 
By \cite[Lemma 7.22]{Kn:2002}, both $k_{0}$ and
$\exp(X_0)$ centralize $\mathfrak{h}_0$, and hence belong to
$\mathbf{H}_0$. In particular, $\exp(X_0)$ belongs to the connected
component $\mathbf{H}_0^0$ of $\mathbf{H}_0$. 
We note now that $k_0$ belongs to 
$\mathbf{K}_{0}\cap\mathbf{I}_0^0$, that is
the connected component 
$\mathbf{K}_{00}^0$ of the identity of $\mathbf{K}_{00}$.
Since the exponential $\exp:\mathfrak{k}_{00}\to\mathbf{K}_{00}^0$ 
is surjective,
$k_0=\exp(Y)$ for some
$Y\in\mathfrak{k}_{00}$. 
Since $\mathfrak{z}_0$ is $\vartheta$-invariant, 
the orthogonal 
of $\mathfrak{k}_{00}\cap\mathfrak{z}_0$ 
in $\mathfrak{k}_{00}$ 
for the Killing form of $\mathfrak{g}_0$ is contained
in the orthogonal $\mathfrak{s}_0$ 
of $\mathfrak{z}_0$ in $\mathfrak{l}_0$,
i.e. $\mathfrak{k}_{00}\cap\mathfrak{z}_0^\perp\subset\mathfrak{s}_0$,
so that 
we have the
decomposition\,:
\begin{equation*}
\mathfrak{k}_{00}=\left(\mathfrak{k}_{00}\cap\mathfrak{s}_0\right)\oplus
\left(\mathfrak{k}_{00}\cap\mathfrak{z}_0\right)\,.
\end{equation*}
Let $Y=S+Z$, with $S\in\mathfrak{k}_{00}\cap\mathfrak{s}_0$ and 
$Z\in\mathfrak{k}_{00}\cap\mathfrak{z}_0$. 
Then $k_{0}=\exp(S)\exp(Z)$. Since $\exp(Z)\in\mathbf{H}^0_0$,
we obtain that $\exp(S)\in\mathbf{S}_0\cap\mathbf{H}_0=\mathbf{T}_0$
and hence\,:
\begin{equation*}
h_0=\exp(S)\,\left(\exp(Z)\exp(X_0)\right) ,
\end{equation*}
with $\exp(S)\in\mathbf{T}_0$, and
$\left(\exp(Z)\exp(X_0)\right)\in\mathbf{H}_0^0$\,.\par
Finally, we note that $\pi_0(\mathbf{I}_0)$ 
is Abelian, being a quotient of the Abelian group
$\mathbf{H}_0$.
The proof is complete.
\end{proof}
\section{Adapted Weyl chambers}
\label{sec:d}
We keep the notation of \S \ref{sec:c}. 
Let $(\vartheta,\mathfrak{h}_0)$ be a Cartan pair adapted to
the parabolic $CR$ algebra
$(\mathfrak{g}_0,\mathfrak{q})$,
and $\mathfrak{g}_0=\mathfrak{k}_0\oplus\mathfrak{p}_0$
the Cartan decomposition corresponding
to $\vartheta$. We still denote by $\vartheta$
its $\mathbb{C}$-linear extension to $\mathfrak{g}$ and
by $\sigma$ and $\tau$ the conjugations of $\mathfrak{g}$
with respect to its real form $\mathfrak{g}_0$ and
its compact form
$\mathfrak{u}=\mathfrak{k}_0\oplus{i}\mathfrak{p}_0$, respectively.
The $\mathbb{C}$-linear map $\vartheta$
and the anti-$\mathbb{C}$-linear maps $\sigma$ and $\tau$
pairwise commute and:
\begin{equation}\label{eq:da}
\tau=\vartheta\circ\sigma=\sigma\circ\vartheta,\quad
\sigma=\vartheta\circ\tau=\tau\circ\vartheta,\quad
\vartheta=\sigma\circ\tau=\tau\circ\sigma.
\end{equation}
Let $\mathcal{R}$ be the root system of $\mathfrak{g}$
with respect to $\mathfrak{h}=\mathfrak{h}_0^{\mathbb{C}}$. Set
\begin{equation}\label{eq:db}
\mathfrak{h}_0^+=\mathfrak{h}_0\cap\mathfrak{k}_0,\quad
\mathfrak{h}_0^-=\mathfrak{h}_0\cap\mathfrak{p}_0,\quad
\mathfrak{h}_{\mathbb{R}}=\mathfrak{h}_0^-\oplus{i}\mathfrak{h}_0^+.
\end{equation}
The root system $\mathcal{R}$ is a subset of the dual
$\mathfrak{h}_{\mathbb{R}}^*$ of $\mathfrak{h}_{\mathbb{R}}$.
The subspace $\mathfrak{h}_{\mathbb{R}}$ of $\mathfrak{g}$
is $\vartheta$, $\sigma$ and $\tau$-invariant. Hence the corresponding
involutions of $\mathfrak{h}_{\mathbb{R}}^*$ define involutions
of $\mathcal{R}$, that are given by\,:
\begin{equation}
\sigma^*(\alpha)=\bar\alpha\,,\quad
\tau^*(\alpha)=-\alpha\,,\quad
\vartheta^*(\alpha)=-\bar\alpha.
\end{equation}
\begin{dfn}\label{dfn:db}
We denote by $\mathcal{R}_{\mathrm{re}}$ the set of \emph{real}
roots in $\mathcal{R}$, i.e. those for which $\bar\alpha=\alpha$,
by $\mathcal{R}_{\mathrm{im}}$ the set of \emph{imaginary} roots
in $\mathcal{R}$, i.e. those for which $\bar\alpha=-\alpha$,
and by $\mathcal{R}_{\mathrm{cpx}}$ the set of \emph{complex} roots
in $\mathcal{R}$, i.e. those for which~$\bar\alpha\neq\pm\alpha$.
\end{dfn}
\begin{lem} Let $(\mathfrak{g}_0,\mathfrak{q})$ be a parabolic
$CR$ algebra and
$(\vartheta,\mathfrak{h}_0)$ a Cartan pair
adapted to $(\mathfrak{g}_0,\mathfrak{q})$.
Then the reductive factor $\mathfrak{q}^r$ in the decomposition
\eqref{eq:bg} with respect to the Cartan subalgebra
$\mathfrak{h}=\mathbb{C}\otimes_{\mathbb{R}}\mathfrak{h}_0$ is
\begin{equation}
\mathfrak{q}^r=\mathfrak{q}\cap\tau(\mathfrak{q}).
\end{equation}
\end{lem}
\begin{proof}
Indeed, if $\mathcal{R}$ is the root system of $\mathfrak{g}$
corresponding to the Cartan subalgebra $\mathfrak{h}$, then
$\tau$ transforms the eigenspace $\mathfrak{g}^{\alpha}$
of $\alpha\in\mathcal{R}$ into $\mathfrak{g}^{-\alpha}$. 
\end{proof}
We have the following
\begin{lem} \label{lem:dd}
Let $(\mathfrak{g}_0,\mathfrak{q})$ be a parabolic
$CR$ algebra and
$(\vartheta,\mathfrak{h}_0)$ a Cartan pair
adapted to $(\mathfrak{g}_0,\mathfrak{q})$. Then:
\par\noindent
$(1)$\quad There exists a fit Weyl chamber 
$C'\in\mathfrak{C}(\mathcal{R},\mathcal{Q})$ such that,
if
$\,\prec\,$ is the partial order in $\mathfrak{h}_{\mathbb{R}}^*$
defined by $C'$, $\mathcal{B}'$ the basis of $C'$-positive
simple roots, and $\Phi'=\mathcal{Q}^n\cap\mathcal{B}'$, 
the two 
following equivalent
conditions are satisfied:
\begin{align}\label{eq:de}
&\text{if $\alpha\in\mathcal{R}_{\mathrm{cpx}}$,
$\alpha\,\succ \,{0}$ 
and
$\bar\alpha\,\prec\,{0}$, then $\alpha$ and $-\bar{\alpha}$
both belong to $\mathcal{Q}^n$,}\\ \label{eq:df}
&\text{$\bar\alpha\,{\succ}\,{0}$ for all
$\alpha
\in\mathcal{B}'
\setminus\left(\Phi'\cup \mathcal{R}_{\mathrm{im}}\right)$.}
\end{align}
\par\smallskip\noindent
$(2)$\quad 
 There exists a fit Weyl chamber $C''\in\mathfrak{C}(\mathcal{R},\mathcal{Q})$
such that,
if
$\,\prec\,$ is the partial order in $\mathfrak{h}_{\mathbb{R}}^*$
defined by $C''$, $\mathcal{B}''$ the basis of $C''$-positive
simple roots, and $\Phi''=\mathcal{Q}^n\cap\mathcal{B}''$, 
the two 
following equivalent
conditions are satisfied:
\begin{align}\label{eq:dg}
&\text{if $\alpha\in\mathcal{R}_{\mathrm{cpx}}$, 
$\alpha\,\succ\,{0}$ and
$\bar\alpha\,\succ\,{0}$, then $\alpha$ and $\bar{\alpha}$
both belong to $\mathcal{Q}^n$,}\\ \label{eq:dh}
&\text{$\bar\alpha\,{\prec}\,{0}$ for all
$\alpha\in\mathcal{B}''\setminus
\left(\Phi''\cup\mathcal{R}_{\mathrm{re}}\right)$.}
\end{align}
\end{lem}
\begin{proof} For $C\in\mathfrak{C}(\mathcal{R})$,
we denote by $\nu(C)$ the number of 
$C$-positive roots $\alpha$ 
for which also $\bar\alpha$ is $C$-positive.
Choose $C'\in\mathfrak{C}(\mathcal{R},\mathcal{Q})$ in such a way that:
\begin{equation}
\nu(C')
=\max_{C\in\mathfrak{C}(\mathcal{R},\mathcal{Q})}
\nu(C).
\end{equation}
If there was a complex $C'$-positive root
$\alpha\in\mathcal{B}'\setminus\Phi'$ with a $C'$-negative
$\bar\alpha$, 
the symmetry $s_{\alpha}$ would transform $C'$ into a 
$C^*$ that still belongs to $\mathfrak{C}(\mathcal{R},\mathcal{Q})$.
Since the set $\mathcal{R}^+(C^*)$ of $C^*$-positive roots
is $\left(\mathcal{R}^+(C')\setminus\{\alpha\}\right)\cup\{-\alpha\}$,
we should have $\nu(C^*)=\nu(C')+1$ and hence a contradiction.
This shows that \eqref{eq:df} is valid for $C'$.
\par
Next we observe that obviously \eqref{eq:de}
implies \eqref{eq:df}. Vice versa, 
assume that \eqref{eq:df} is true and let 
$\alpha=\sum_{\beta\in\mathcal{B}'}{k_\alpha^\beta\beta}$
be a complex $C'$-positive root with 
a $C'$-negative $\bar\alpha$. Then $\bar\beta$ should be
$C'$-negative for some complex
$\beta$ in $\mathrm{supp}(\alpha)$, and 
hence $\mathrm{supp}(\alpha)\cap\Phi'\neq\emptyset$
by \eqref{eq:df}, implying that $\alpha\in\mathcal{Q}^n$.
The same argument, applied to the complex $C'$-positive root
$(-\bar\alpha)$, shows then that also
$(-\bar\alpha)\in\mathcal{Q}^n$.\par
Similarly, by taking a $C''\in\mathfrak{C}(\mathcal{R},\mathcal{Q})$
that minimizes $\nu(C)$ in $\mathfrak{C}(\mathcal{R},\mathcal{Q})$
we obtain a fit Weyl chamber that satisfies the conditions
\eqref{eq:dg} and \eqref{eq:dh}.
\end{proof}
\begin{dfn}\label{dfn:de}
Let $(\vartheta,\mathfrak{h}_0)$ be a Cartan pair 
adapted to a parabolic
$CR$ algebra $(\mathfrak{g}_0,\mathfrak{q})$,
$\mathcal{R}$ the root system of $\mathfrak{g}$
with respect to 
the complexification $\mathfrak{h}$
of $\mathfrak{h}_0$, $\mathcal{Q}$ the parabolic
set of $\mathfrak{q}$ in $\mathcal{R}$. A
fit Weyl chamber $C$ for $\mathcal{Q}$
is:\begin{itemize}
\item[\emph{$S$-fit}] for $(\mathfrak{g}_0,\mathfrak{q})$ if 
$C$ satisfies
the equivalent conditions \eqref{eq:de}, \eqref{eq:df};
\item[\emph{$V$-fit}] for $(\mathfrak{g}_0,\mathfrak{q})$ if 
$C$ satisfies
the equivalent conditions \eqref{eq:dg}, \eqref{eq:dh}.
\end{itemize}
\end{dfn}
\section{${{\mathbf{G}_0}}$-equivariant fibrations}\label{sec:e}
In this section we shall discuss ${{\mathbf{G}_0}}$-equivariant
fibrations of 
${{\mathbf{G}_0}}$-orbits. We begin by proving 
that their structures of $\mathbf{G}_0$-homogeneous space is completely
determined by the \textit{Lie algebras} of their isotropy subgroups.
\begin{thm} \label{thm:ea}
Let $M$, $M'$ be ${{\mathbf{G}_0}}$-orbits, corresponding to
$CR$ algebras $(\mathfrak{g}_0,\mathfrak{q})$,
$(\mathfrak{g}_0,\mathfrak{q}')$, where
$\mathfrak{q}$, $\mathfrak{q}'$ are two complex parabolic
subalgebras of $\mathfrak{g}$.
Let $\mathfrak{i}_0=\mathfrak{q}\cap\mathfrak{g}_0$ and
$\mathfrak{i}_0'=\mathfrak{q}'\cap\mathfrak{g}_0$ be
the Lie subalgebras of the isotropy subgroups $\mathbf{I}_0$
and $\mathbf{I}_0'$ of $M$ and $M'$, respectively.
Then:
\begin{equation}\label{eq:ea}
\mathfrak{i}'_0\subset\mathfrak{i}_0\Longrightarrow
\mathbf{I}_0'\subset\mathbf{I}_0,
\end{equation}
and therefore, if $\mathfrak{i}'_0\subset\mathfrak{i}_0$, 
there is a canonical ${{\mathbf{G}_0}}$-equivariant fibration:
\begin{equation}\label{eq:eb}
\begin{CD} M'@>{\phi}>>M\,. \end{CD}
\end{equation}
Assume that $\mathfrak{i}'_0\subset\mathfrak{i}_0$. We have: 
\begin{enumerate}
\item If $\mathfrak{i}_0'$
contains a Cartan subalgebra $\mathfrak{h}_0$
that is maximally noncompact as a  Cartan subalgebra of 
$\mathfrak{i}_0$, then 
the fiber $F$ of \eqref{eq:eb} is connected.
\item The projection $\phi$  is
a $CR$ map if and only if
$\mathfrak{q}'\subset\mathfrak{q}$.
\item The projection $\phi$  is
a $CR$ submersion, and hence \eqref{eq:eb} is a $CR$ fibration,
if and only if \; 
$\mathfrak{q}=\mathfrak{q}'+(\mathfrak{q}\cap\bar{\mathfrak{q}}).$
\end{enumerate}
\end{thm}
\begin{proof} Assume that $\mathfrak{i}'_0\subset\mathfrak{i}_0$.
 To prove that $\mathbf{I}_0'\subset\mathbf{I}_0$, it suffices
to prove that each connected component of $\mathbf{I}_0'$
intersects $\mathbf{I}_0$. \par
Fix a
Cartan pair $(\vartheta,\mathfrak{h}_0)$ adapted to both 
$(\mathfrak{g}_0,\mathfrak{q}')$, 
and $(\mathfrak{g}_0,\mathfrak{q})$. 
This can be obtained by first taking 
the Cartan involution
$\vartheta$ associated to a maximally compact subalgebra
$\mathfrak{k}_{0}$ of $\mathfrak{g}_0$, with
$\mathfrak{k}_{00}'=\mathfrak{k}_0\cap{\mathfrak{i}_0'}$
and 
$\mathfrak{k}_{00}=\mathfrak{k}_0\cap\mathfrak{i}_0$ 
maximally compact in ${\mathfrak{i}_0'}$ and $\mathfrak{i}_0$,
respectively, and then choosing any $\vartheta$-invariant
Cartan subalgebra $\mathfrak{h}_0$ of  $\mathfrak{i}_0'$.
\par
Consider 
the centers $\mathfrak{z}_0'$ of $\mathfrak{l}_0'=\mathfrak{i}_0'
\cap\vartheta(\mathfrak{i}_0')$ and
$\mathfrak{z}_0$ of $\mathfrak{l}_0=\mathfrak{i}_0
\cap\vartheta(\mathfrak{i}_0)$. Since $\mathfrak{z}_0\subset\mathfrak{l}_0'$,
we have
$\mathfrak{z}'_0\supset\mathfrak{z}_0$. Thus, 
by \eqref{eq:cd}, 
$\mathbf{L}_0'=\mathbf{Z}_{{{\mathbf{G}_0}}}(\mathfrak{z}'_0)\subset
\mathbf{Z}_{{{\mathbf{G}_0}}}(\mathfrak{z}_0)=\mathbf{L}_0$.
By \eqref{eq:cc}, ${{\mathbf{L}_0'}}$ is a deformation retract
of $\mathbf{I}_0'$. Then each
connected component of ${\mathbf{I}_0'}$ intersects ${{\mathbf{L}_0'}}$
and thus also $\mathbf{I}_0$.
This proves \eqref{eq:ea}.
\par\smallskip
$(1)$. Assume that $\mathfrak{h}_0$ is a Cartan subalgebra of
$\mathfrak{i}_0'$ that is maximally noncompact both in
$\mathfrak{i}_0'$ and in $\mathfrak{i}_0$, and
consider the corresponding Cartan subgroup 
$\mathbf{H}_0=\mathbf{Z}_{{{\mathbf{G}_0}}}(\mathfrak{h}_0)$.
By \cite[Proposition 7.90]{Kn:2002},
each connected component of $\mathbf{L}_0$ and of $\mathbf{L}_0'$,
and, because of \eqref{eq:cc}, of $\mathbf{I}_0'$ and of
$\mathbf{I}_0$, 
intersects $\mathbf{H}_0$. Being 
$\mathbf{H}_0\subset\mathbf{L}_0'\subset\mathbf{L}_0$,
this shows that each connected component of $\mathbf{I}_0$
contains points of $\mathbf{I}_0'$, and hence fore
$\mathbf{I}_0/\mathbf{I}_0'$
is connected.\par\smallskip
$(2)$.  
Let 
$\mathbf{I}_0'$ and $\mathbf{I}_0$ be the isotropy subgroups
of $M'$ and $M$ at the points $x_0'$ and $x_0=\phi(x_0')$, respectively.
Denote by $d\pi':\mathfrak{g}\to T^{\mathbb{C}}_{x'_0}M'$
and by $d\pi:\mathfrak{g}\to T^{\mathbb{C}}_{x_0}M$ the
complexification of the differential at the identity of the action
of ${{\mathbf{G}_0}}$ on $M'$ and $M$,
with base points $x'_0$ and $x_0$. By \eqref{eq:ab}
$\mathfrak{q}'={d\pi'}^{-1}(T^{0,1}_{x'_0}M')$ and
$\mathfrak{q}={d\pi}^{-1}(T^{0,1}_{x_0}M)$. 
Let ${d\phi}^{\mathbb{C}}_{x_0'}$ be the complexification of the differential
of $\phi$ at $x'_0$.  When $\phi$ is a $CR$ map, we have:
\begin{equation*}
d\pi(\mathfrak{q}')=
{d\phi}^{\mathbb{C}}_{x'_0}(d\pi'(\mathfrak{q}'))=
{d\phi}^{\mathbb{C}}_{x'_0}(T^{0,1}_{x'_0}M')\subset T^{0,1}_{x_0}M
=d\pi(\mathfrak{q})\,.
\end{equation*} 
Since $\mathfrak{q}=(d\pi)^{-1}(d\pi(\mathfrak{q}))$, this implies
that $\mathfrak{q}'\subset\mathfrak{q}$. Vice versa, when 
$\mathfrak{q}'\subset\mathfrak{q}$, the map $\phi:M'\to{M}$ is
$CR$, being the restriction of the $\mathbf{G}$-equivariant
holomorphic projection $X'=\mathbf{G}/\mathbf{Q}'
\to\mathbf{G}/\mathbf{Q}={X}$.
\par\smallskip
Finally, $(3)$ is a consequence of \cite[Lemma 4.5]{MN05}.
\end{proof}
\begin{rmk}\label{rmk:eb}
By Theorem \ref{thm:ea}, orbits
$M$ and $M'$, corresponding to 
parabolic
$CR$ algebras $(\mathfrak{g}_0,\mathfrak{q})$
and $(\mathfrak{g}_0,\mathfrak{q}')$ with
$\mathfrak{q}'\cap\bar{\mathfrak{q}}'=\mathfrak{q}\cap\bar{\mathfrak{q}}$,
are ${{\mathbf{G}_0}}$-equivariantly
diffeomorphic. Thus, having fixed the real subalgebra
$\mathfrak{i}_0$, we can regard the complex parabolic $\mathfrak{q}$'s
with $\mathfrak{q}\cap\mathfrak{g}_0=\mathfrak{i}_0$ as defining
different
${{\mathbf{G}_0}}$-invariant
$CR$ structures on a same ${{\mathbf{G}_0}}$-homogeneous smooth manifold $M$,
each
corresponding to a different $CR$-generic embedding of $M$
into a complex flag manifold. \end{rmk} 
In the next theorem we construct a somehow
\textit{minimal} $CR$ structure on a ${{\mathbf{G}_0}}$-orbit $M$.
Note that the subalgebra
$\mathfrak{q}_w$ of the theorem below is equal to
the $\mathfrak{q}'$ used in the proof of Proposition \ref{prop:cb}.
\begin{thm}\label{thm:ec} 
Let $(\mathfrak{g}_0,\mathfrak{q})$ be a parabolic $CR$ algebra. Then:
\begin{equation}\label{eq:ee}
\mathfrak{q}_w=\mathfrak{q}^n+\mathfrak{q}\cap\bar{\mathfrak{q}}
\end{equation}
is a complex parabolic subalgebra of $\mathfrak{g}$,
contained in $\mathfrak{q}$. It is the minimal complex parabolic
subalgebra of $\mathfrak{g}$ with the properties that:
\begin{equation}\label{eq:ef}
\mathfrak{q}_w\cap\bar{\mathfrak{q}}_w=\mathfrak{q}\cap\bar{\mathfrak{q}}
\quad\text{and}\quad \mathfrak{q}_w\subset\mathfrak{q}.
\end{equation}
Any Cartan pair $(\vartheta,\mathfrak{h}_0)$ 
adapted to $(\mathfrak{g}_0,\mathfrak{q})$
is also adapted to 
$(\mathfrak{g}_0,\mathfrak{q}_w)$. 
Fix a Cartan pair $(\vartheta,\mathfrak{h}_0)$
adapted to $(\mathfrak{g}_0,\mathfrak{q})$ and hence also to
$(\mathfrak{g}_0,\mathfrak{q}_w)$. Let $\tau$ 
be the conjugation with respect to the 
$\vartheta$-invariant compact form of $\mathfrak{g}$
(see \eqref{eq:ca}), and
set
\begin{equation}\label{eq:ep}
  \mathfrak{q}^r=\mathfrak{q}\cap\tau(\mathfrak{q}),\quad
\mathfrak{q}^r_w=\mathfrak{q}_w\cap\tau(\mathfrak{q}_w).
\end{equation}
Then we have:
\begin{align}
\label{eq:eg}
\mathfrak{q}_w&=\mathfrak{q}^n+\mathfrak{q}^r\cap\bar{\mathfrak{q}},
\\\label{eq:em}
\mathfrak{q}_w^n&=\mathfrak{q}^n+\mathfrak{q}^r\cap\bar{\mathfrak{q}}^n,\\
\label{eq:eh}
\mathfrak{q}_w^r=\bar{\mathfrak{q}}_w^r&=\mathfrak{q}_w^r\cap
\bar{\mathfrak{q}}_w^r=\mathfrak{q}^r\cap
\bar{\mathfrak{q}}^r\,.
\end{align}\par
Let $M_w$ be the ${{\mathbf{G}_0}}$-orbit corresponding
to $(\mathfrak{g}_0,\mathfrak{q}_w)$. 
Then the canonical
${{\mathbf{G}_0}}$-equivariant fibration $M_w\to{M}$ is a $CR$ map and
a smooth diffeomorphism.
\end{thm}
\begin{proof} Let $(\vartheta,\mathfrak{h}_0)$ be any Cartan pair
adapted to $(\mathfrak{g},\mathfrak{q})$.
Let $\mathcal{R}$ be the root system of $\mathfrak{g}$, with respect
to the complexification $\mathfrak{h}$ of $\mathfrak{h}_0$.
We fix an $A\in\mathfrak{h}_{\mathbb{R}}$ that defines the parabolic
set $\mathcal{Q}$ of $\mathfrak{q}$, i.e. such that
(cf. e.g. \cite[VIII \S 4]{Bou75}: $A$ is any element of the interior
of the facet associated to $\mathfrak{q}$ and $\mathfrak{h}$):
\begin{equation*}
\mathcal{Q}=\{\alpha\in\mathcal{R}\,|\,\alpha(A)\geq{0}\}.
\end{equation*} 
If $\epsilon>0$ is so small that
$\alpha(A)>\epsilon |\bar\alpha(A)|$ for all $\alpha\in\mathcal{Q}^n$,
then $\mathfrak{q}_w$ is the complex parabolic Lie subalgebra of
$\mathfrak{g}$ that corresponds to the parabolic set:
\begin{equation*}
\mathcal{Q}_w=\{\alpha\in\mathcal{R}\,|\, \alpha(A+\epsilon\bar{A})\geq 0\}\,.
\end{equation*}
This observation easily yields \eqref{eq:ef}, \eqref{eq:eg},
\eqref{eq:em},
\eqref{eq:eh}. The last statement follows from
\eqref{eq:ef}, Theorem \ref{thm:ea}, and Remark \ref{rmk:eb}.
The fact that $\mathfrak{q}_w$ is minimal with the properties
of \eqref{eq:ef} follows from \eqref{eq:eh}:
in fact any parabolic subalgebra containing 
$\mathfrak{q}\cap\bar{\mathfrak{q}}$ must contain the reductive
subalgebra $\mathfrak{q}^r\cap\bar{\mathfrak{q}}^r$.
\end{proof}
\begin{dfn}\label{dfn:ed}
The parabolic $CR$ algebra $(\mathfrak{g}_0,\mathfrak{q}_w)$, and the 
corresponding ${{\mathbf{G}_0}}$-orbit $M_w$, are called
the $CR$-\emph{weakening} of $(\mathfrak{g}_0,\mathfrak{q})$, and of $M$,
respectively. 
\end{dfn}
\begin{rmk}\label{rmk:ee}
  In \cite[\S{9.1}]{Wolf69} the orbits $M$ corresponding to 
parabolic $CR$ algebras $(\mathfrak{g}_0,\mathfrak{q})$ with
$\mathfrak{q}^r=\bar{\mathfrak{q}}^r$ are called \emph{polarized}. 
Thus the $CR$ weakening of $M$ is polarized, and, vice versa,
if $M$ is polarized, $M$ coincides with its $CR$ weakening $M_w$.
\end{rmk}
We have
\begin{lem} \label{lem:ej}
Let $(\mathfrak{g}_0,\mathfrak{q})$ be a 
polarized parabolic $CR$ algebra, i.e. 
assume that $\mathfrak{q}$ contains a 
maximal reductive subalgebra $\mathfrak{q}^r$ with
$\mathfrak{q}^r=\bar{\mathfrak{q}}^r$.
Then $(\mathfrak{g}_0,\mathfrak{q})$ is either totally real, or
weakly degenerate.
\end{lem}
\begin{proof}
Let $(\vartheta,\mathfrak{h}_0)$ be a Cartan pair adapted to
$(\mathfrak{g}_0,\mathfrak{q})$ and $\mathcal{R}$ the root system
of $\mathfrak{g}$ with respect to the complexification $\mathfrak{h}$
of $\mathfrak{h}_0$. We fix an $S$-fit Weyl chamber for 
$(\mathfrak{g}_0,\mathfrak{q})$. By the assumption, the parabolic set
$\mathcal{Q}$ of $\mathfrak{q}$ has the property that
$\mathcal{Q}^r=\bar{\mathcal{Q}}^r$. 
This is equivalent to the fact
that 
\begin{equation*}
\mathrm{supp}(\bar\alpha)\cap\Phi\neq\emptyset \quad\forall\alpha\in
\mathcal{Q}^n,
\end{equation*}
where $\Phi$ is the set of $C$-positive simple roots in $\mathcal{Q}$.
If $\bar\alpha\succ{0}$ for all $\alpha\in\Phi$, 
then $\bar{\mathcal{Q}}^n=\mathcal{Q}^n$, and $\mathfrak{q}$ is totally real. 
Assume that there is 
$\alpha\in\Phi$ with $\bar\alpha\prec{0}$ and consider
the parabolic $\mathfrak{q}'=\mathfrak{q}_{\Psi}$, with 
$\Psi=\Phi\setminus\{\alpha\}$. Then 
\begin{equation*}\mathcal{Q}_{\Psi}
\supset\mathcal{Q}\quad\text{and}\quad
\mathcal{Q}_{\Psi}\setminus\mathcal{Q}=
\{\beta\in\mathcal{R}\mid \beta\prec{0},\;
\mathrm{supp}{(\beta)}\cap\Phi=\{\alpha\}\}.
\end{equation*}
If $\beta\in(\mathcal{Q}_{\Psi}\setminus\mathcal{Q})$, then 
$\bar\beta\succ{0}$. Indeed, 
$\mathrm{supp}(\bar\alpha)\cap\Phi\neq\emptyset$, and 
$\mathrm{supp}(\bar\gamma)\cap\Phi=\emptyset$ if 
$\gamma$ is a $C$-positive simple root not belonging to $\Phi$.
Thus,  
the decomposition of $\bar\beta$ into a linear combination
of $C$-positive simple roots contains some element of $\Phi$
with a positive coefficient and hence is positive.
This shows that $({\mathcal{Q}_{\Psi}\setminus\mathcal{Q}})\subset
\bar{\mathcal{Q}}$. Therefore $\mathcal{Q}_{\Psi}\subset\mathcal{Q}\cup
\bar{\mathcal{Q}}$. Thus we obtained 
$\mathfrak{q}\subsetneq\mathfrak{q}_{\Psi}\subset\mathfrak{q}+
\bar{\mathfrak{q}}$, showing that $(\mathfrak{g}_0,\mathfrak{q})$
is weakly degenerate.
\end{proof}
From Remark \ref{rmk:ee} and Lemma \ref{lem:ej}, and the characterization
of the $CR$-weakening in the proof of Theorem \ref{thm:ec}, we obtain:
\begin{prop} \label{prop:ee}
Let $(\mathfrak{g}_0,\mathfrak{q})$ be a parabolic $CR$ algebra, and
$(\mathfrak{g}_0,\mathfrak{q}_w)$ its $CR$-weakening.
Then $(\mathfrak{g}_0,\mathfrak{q}_w)$ is either totally real, or
weakly degenerate. \par
Let $(\vartheta,\mathfrak{h}_0)$ be an adapted 
Cartan pair for 
$(\mathfrak{g}_0,\mathfrak{q})$, and hence also for 
$(\mathfrak{g}_0,\mathfrak{q}_w)$. Denote by
$\mathcal{R}$ the root system
of $\mathfrak{g}$ with respect to the complexification 
$\mathfrak{h}$ of $\mathfrak{h}_0$, and 
by $\mathcal{Q},\;\mathcal{Q}_w$
the parabolic sets of $\mathfrak{q}$, $\mathfrak{q}_w$,
respectively.
Then:
\begin{enumerate}
\item
$C\in\mathfrak{C}(\mathcal{R},\mathcal{Q})$ 
is $S$-fit
for $(\mathfrak{g}_0,\mathfrak{q})$ if and only if it is $S$-fit
for $(\mathfrak{g}_0,\mathfrak{q}_w)$. 
\item Let $\mathcal{B}$ be the basis of $C$-positive simple roots
for an $S$-fit 
$C\in\mathfrak{C}(\mathcal{R},\mathcal{Q})$, and 
$\Phi=\mathcal{B}\cap\mathcal{Q}^n$, so that $\mathfrak{q}=\mathfrak{q}_{\Phi}$.
Then $\mathfrak{q}_w=\mathfrak{q}_{\Phi_w}$ with
\begin{align}\quad\qquad &&
\Phi_w=\Phi\cup\{\alpha\in\mathcal{B}\mid \bar\alpha\succ 0,\;\,
\mathrm{supp}(\bar\alpha)\cap\Phi\neq\emptyset\}. &&\quad\qquad\qed
\end{align}
\end{enumerate}
\end{prop}
We introduce the following: 
\begin{dfn}
Let $M$ be a ${{\mathbf{G}_0}}$-homogeneous $CR$ manifold with associated
$CR$ algebra $(\mathfrak{g}_0,\mathfrak{q})$. A 
\emph{strengthening} of the
$CR$ structure of $M$ is the datum of a complex Lie subalgebra
$\mathfrak{q}'$ with:
\begin{equation}\label{eq:ej}
\mathfrak{q}'\cap\bar{\mathfrak{q}}'=\mathfrak{q}\cap\bar{\mathfrak{q}}
\quad\text{and}\quad  \mathfrak{g}\supset\mathfrak{q}'\supset\mathfrak{q}. 
\end{equation}
We say that the ${{\mathbf{G}_0}}$-homogeneous $CR$ structure defined by
$(\mathfrak{g}_0,\mathfrak{q})$ is \emph{maximal} if
$\mathfrak{q}'=\mathfrak{q}$ for all complex Lie 
subalgebras $\mathfrak{q}'$ of $\mathfrak{g}$ satisfying
\eqref{eq:ej}.
\end{dfn}
If $(\mathfrak{g}_0,\mathfrak{q})$ is a parabolic $CR$ algebra and 
$M_s$ is the ${{\mathbf{G}_0}}$-orbit associated to a strengthening 
$(\mathfrak{g}_0,\mathfrak{q}_s)$ of
$(\mathfrak{g}_0,\mathfrak{q})$, then the ${{\mathbf{G}_0}}$-equivariant
map $M\to{M}_s$ is a diffeomorphism and a $CR$ map. We have:
\begin{prop} Let $M$ be the ${{\mathbf{G}_0}}$-orbit associated to
the parabolic $CR$ algebra $(\mathfrak{g}_0,\mathfrak{q})$. 
Fix an adapted Cartan pair $(\vartheta,\mathfrak{h}_0)$ and
let $\mathfrak{q}=\mathfrak{q}_{\Phi}$ for a system $\Phi$ 
of $C$-positive simple roots of an $S$-fit Weyl chamber $C$.
Then:
\begin{enumerate}
\item  The necessary and sufficient condition for the
$CR$ structure defined by $(\mathfrak{g}_0,\mathfrak{q}_{\Phi})$
to be maximal is that:
\begin{equation}\label{eq:el}
\alpha\in\Phi\;\text{and}\; \bar\alpha \succ 0
\;\Longrightarrow\; \mathrm{supp}(\bar\alpha)\cap\Phi
\subset\{\alpha\}.
\end{equation}
\item 
There are maximal ${{\mathbf{G}_0}}$-homogeneous 
$CR$ structures 
$(\mathfrak{g}_0,\mathfrak{q}')$ on $M$, and for each of them
$\mathfrak{q}'=\mathfrak{q}_{\Psi}$ for some system of simple roots
\mbox{$\Psi\subset\Phi$.}
\end{enumerate}
\end{prop}
\begin{proof} Let $\alpha\in\Phi$
and set $\Psi=\Phi\setminus\{\alpha\}$. 
If $\bar\alpha \prec 0$,
then $-\alpha\in(\mathcal{Q}_{\Psi}\cap\bar{\mathcal{Q}}_{\Psi})
\setminus(\mathcal{Q}_{\Phi}\cap\bar{\mathcal{Q}}_{\Phi})$,
and hence $\mathfrak{q}_{\Psi}$ does not satisfy \eqref{eq:ej}.
If $\bar\alpha \succ 0$ and 
$\mathrm{supp}(\bar\alpha)\cap\Phi\subset\{\alpha\}$,
again $-\alpha\in(\mathcal{Q}_{\Psi}\cap\bar{\mathcal{Q}}_{\Psi})
\setminus(\mathcal{Q}_{\Phi}\cap\bar{\mathcal{Q}}_{\Phi})$,
showing that also in this case
$\mathfrak{q}_{\Psi}$ does not satisfy \eqref{eq:ej}.
Since all complex subalgebras $\mathfrak{q}'$ that satisfy
\eqref{eq:ej} are of the form $\mathfrak{q}'=\mathfrak{q}_{\Psi}$
for some $\Psi\subset\Phi$, this proves
that \eqref{eq:el} is a necessary condition.\par
Vice versa, if there is $\alpha\in\Phi$ with
$\bar\alpha\succ 0$ and $\mathrm{supp}(\bar\alpha)
\cap\Phi\not\subset\{\alpha\}$, then, for $\Psi=\Phi
\setminus\{\alpha\}$,
the complex subalgebra $\mathfrak{q}_{\Psi}$ satisfies 
\eqref{eq:ej}, and hence $(\mathfrak{g}_0,\mathfrak{q}_{\Phi})$
is not maximal. This proves~$(1)$.\par
To prove (2), it suffices to take any maximal 
element of the set
of all complex Lie subalgebras $\mathfrak{q}'$ with 
$\mathfrak{q}\subset\mathfrak{q}'\subset\mathfrak{g}$ and
$\mathfrak{q}'\cap\bar{\mathfrak{q}}'=
\mathfrak{q}\cap\bar{\mathfrak{q}}$.
\end{proof} 
We conclude this section by proving a
theorem that describes the structure of the fiber $F$ of
a ${{\mathbf{G}_0}}$-equivariant $CR$ fibration.
\begin{thm}\label{thm:eh} 
Let $M$, $M'$ be $\mathbf{G}_0$-orbits, corresponding to the
parabolic $CR$ algebras $(\mathfrak{g}_0,\mathfrak{q})$,
$(\mathfrak{g}_0,\mathfrak{q}')$.
Assume that
$\mathfrak{q}'\subset\mathfrak{q}$,
so that $\mathbf{I}_0\supset\mathbf{I}_0'$ for the isotropy
subgroups with Lie algebras $\mathfrak{i}_0=\mathfrak{q}\cap\mathfrak{g}_0$
and $\mathfrak{i}_0'=\mathfrak{q}'\cap\mathfrak{g}_0$.\par
Fix a Cartan pair $(\vartheta,\mathfrak{h}_0)$, adapted to
$(\mathfrak{g}_0,\mathfrak{q}')$, and hence also to
$(\mathfrak{g}_0,\mathfrak{q})$.
\par
Consider the
canonical ${{\mathbf{G}_0}}$-equivariant fibration
$M'\to{M}$, with typical fiber $F=\mathbf{I}_0/{\mathbf{I}_0'}$. 
We use the notation of Propositions \ref{prop:ca} and \ref{prop:cb}.
\par
The fiber
 $F$ is
an $\mathbf{I}_0$-homogeneous $CR$ manifold. Its associated
$CR$ algebra $(\mathfrak{i}_0,\bar{\mathfrak{q}}\cap
\mathfrak{q}')$ is the semidirect product of the
$CR$ algebras $(\mathfrak{l}_0,\mathfrak{q}'\cap\mathfrak{l})$
and 
\mbox{$(\mathfrak{n}_0,\mathfrak{q}'\cap\mathfrak{n})$}, where:
\begin{enumerate}
\item
$(\mathfrak{l}_0,\mathfrak{q}'\cap\mathfrak{l})$
is a parabolic $CR$ algebra;
\item
$(\mathfrak{n}_0,\mathfrak{q}'\cap{\mathfrak{n}})$
is nilpotent and totally complex, i.e. 
${\mathfrak{n}}
=\mathfrak{n}_0+(\mathfrak{q}'\cap{\mathfrak{n}})$.
\end{enumerate}
The fiber
$F$ 
is $CR$ diffeomorphic to a Cartesian product:
\begin{equation}\label{eq:en}
F=F'\times{F}'',\vspace{-10pt}\end{equation} 
\qquad\qquad where: 
\begin{enumerate}
\item[($1'$)]
$F'$ has finitely many connected components, 
each 
isomorphic to the
$\mathbf{L}_0^0$-orbit in the flag manifold 
${{X}'}={\mathbf{L}}/\left(
\mathbf{Q}'\cap\mathbf{L}\right)$
corresponding to the parabolic $CR$ algebra 
$(\mathfrak{l}_0,\mathfrak{q}'\cap\mathfrak{l})$.
Here we denoted by
$\mathbf{L}_0^0$ the connected component
of the identity of $\mathbf{L}_0$.
\item[($2'$)]
$F''$ is a Euclidean complex 
$\mathbf{N}_0$-nilmanifold,
with associated $CR$ algebra 
$(\mathfrak{n}_0,\mathfrak{q}'\cap{\mathfrak{n}})$.
\end{enumerate}
\par
If $\mathfrak{q}\subset\mathfrak{q}'+\bar{\mathfrak{q}}'$,
then the fibers of the 
$\mathbf{G}_0$-equivariant fibration $M'\to{M}$
are complex and simply connected (but not necessarily connected).
\end{thm}
\begin{proof} The $CR$ structure of the fiber $F$ is defined by
the embedding into the complex flag manifold 
${{X}''}=\mathbf{Q}/\mathbf{Q}'$. However, the embedding 
$F\hookrightarrow{{X}''}$ is, in general, not $CR$-generic.
Thus we begin by considering a natural generic embedding of $F$.
\par
The algebraic subgroup $\mathbf{Q}\cap\bar{\mathbf{Q}}$ 
decomposes into the semidirect product
\begin{equation}
\mathbf{Q}\cap\bar{\mathbf{Q}}=\mathbf{L}\ltimes
\mathbf{N},
\end{equation}
where ${\mathbf{N}}$ is its unipotent radical and 
$\mathbf{L}=\mathbf{Q}^r\cap\bar{\mathbf{Q}}^r$
is reductive.\par
The intersection
$\mathbf{Q}\cap\bar{\mathbf{Q}}$
contains a Cartan subgroup of $\mathbf{G}$, and therefore
is connected. Thus 
also the groups $ \mathbf{L} $ and
$ \mathbf{N} $ are connected. Their Lie algebras are
$\mathfrak{l}$ and
${\mathfrak{n}}$, respectively. Moreover, 
$ \mathbf{N} $ is also simply connected, being conjugate,
in the linear group $\mathbf{G}$, to a group of 
unipotent upper triangular
matrices (see e.g. \cite[\S{17.5}]{Humph}).\par
The connected component 
$\mathbf{L}^0_0$
of the identity
of $\mathbf{L}_0$ is a real form of 
$ \mathbf{L} $.
The parabolic subalgebra $\mathfrak{q}'$ of
$\mathfrak{g}$, containing a
Cartan subalgebra of $\mathfrak{l}$, intersects
$\mathfrak{l}$ into the
complex parabolic subalgebra 
$\mathfrak{q}'\cap\mathfrak{l}$ of $\mathfrak{l}$.
The intersection $\mathbf{Q}'\cap \mathbf{L} $
is the parabolic subgroup of $ \mathbf{L} $
corresponding to $\mathfrak{q}'\cap\mathfrak{l}$.
The quotient $F'=\mathbf{L}_0/\left(\mathbf{Q}'\cap\mathbf{L}_0\right)$
is therefore the union of finitely many copies of an
$\mathbf{L}_0^0$-orbit in the flag manifold 
${{X}'}= \mathbf{L} /\left({\mathbf{Q}'}\cap{\mathbf{L}}\right)$.
\par
Next we note that the intersection 
of $\mathbf{Q}'$ with
the unipotent subgroup $ \mathbf{N} $ is a subgroup
of its unipotent radical ${\mathbf{Q}'}^{\,n}$. Thus it is
connected and simply connected and the quotient 
$Y=
 \mathbf{N} /(\mathbf{Q}'\cap \mathbf{N} )$
is a connected and simply connected complex nilmanifold. 
Since both $\mathbf{Q}'\cap\mathbf{G}_0$ and
$\mathbf{Q}'\cap{\mathbf{N}}$ are closed and connected,
and we have a Lie algebras semidirect sum decomposition:
\begin{equation}
\mathfrak{q}'\cap\bar{\mathfrak{q}}=(\mathfrak{q}'
\cap\mathfrak{l})
\ltimes (\mathfrak{q}'\cap{\mathfrak{n}}),
\end{equation}
we also obtain a semidirect product decomposition\,:
\begin{equation}
\mathbf{Q}'\cap\bar{\mathbf{Q}}=(\mathbf{Q}'\cap
 \mathbf{L} )\ltimes
(\mathbf{Q}'\cap
 \mathbf{N} ).
\end{equation}
Hence the fiber $F=\mathbf{I}_0/{\mathbf{I}_0'}$
is $CR$ diffeomorphic to the Cartesian product $F=F'\times{F}''$,
with the $F'$ described above,
and where $F''$ is the orbit of $\mathbf{N}_0$ in $Y$.
Since
$
\mathfrak{q}'\cap{\mathfrak{n}}\supset
{\mathfrak{q}'}^{\,n}\cap\mathfrak{q}^n\cap\bar{\mathfrak{q}}=
\mathfrak{q}^n\cap\bar{\mathfrak{q}}$,
we have
\begin{gather*}
\mathfrak{n}_0+(\mathfrak{q}'\cap{\mathfrak{n}})\supset
(\mathfrak{q}'\cap{\mathfrak{n}})+
(\overline{\mathfrak{q}'\cap{\mathfrak{n}}})\supset
(\mathfrak{q}^n\cap\bar{\mathfrak{q}})+(\bar{\mathfrak{q}}^n\cap\mathfrak{q})
={\mathfrak{n}}.
\end{gather*}
Thus $(\mathfrak{n}_0,\mathfrak{q}'\cap{\mathfrak{n}})$ 
is
totally complex.
The orbit $F''$ of
$\mathbf{N}_0$ in $Y$, being an open
Euclidean complex submanifold of $Y$,
coincides with it, because $\mathbf{N}_0$ is nilpotent.
\par
Finally, if 
$\mathfrak{q}\subset\mathfrak{q}'\subset\mathfrak{q}+\bar{\mathfrak{q}}$,
the factor $F'$ in the decomposition
\eqref{eq:en} is an open orbit in $X'$
and hence simply connected by \cite[Theorem 5.4]{Wolf69}.
 \end{proof}
\section{${{\mathbf{G}_0}}$-equivariant 
$CR$ fibrations and fit Weyl chambers}\label{sec:f}
In this section we describe the fundamental and the weakly nondegenerate
reductions (see \cite{MN05}) of a parabolic $CR$ manifold $M$,
with associated $CR$ algebra $(\mathfrak{g}_0,\mathfrak{q})$. 
This description will be obtained in terms of  representations
$\mathfrak{q}=\mathfrak{q}_{\Phi}$ of the parabolic subalgebra
$\mathfrak{q}$ with respect to 
systems
$\Phi$ of $C$-positive simple roots
for $S$-fit and $V$-fit Weyl chambers $C$.
Since its fundamental and weakly nondegenerate reductions
share with $(\mathfrak{g}_0,\mathfrak{q})$ the same
adapted Cartan pairs, we can fix throughout this section
a Cartan pair $(\vartheta,\mathfrak{h}_0)$, adapted to the parabolic
$CR$ algebra $(\mathfrak{g}_0,\mathfrak{q})$.
\begin{lem}\label{lem:fa}
Let $(\mathfrak{g}_0,\mathfrak{q}_{\Phi})$
be a parabolic $CR$ algebra, with $\Phi\subset\mathcal{B}$,
where $\mathcal{B}$ is the set
of $C$-positive simple roots
for
an $S$-fit Weyl chamber $C$. 
Set
\begin{equation}\label{eq:fn}
  \Phi^-=\{\alpha\in\Phi\mid \bar\alpha\prec{0}\}.
\end{equation}
\par
Let 
$\alpha_0\in\mathcal{B}$. A necessary
and sufficient condition in order that:
\begin{gather}\label{eq:fa}
\mathfrak{q}_{\Phi}+\bar{\mathfrak{q}}_{\Phi}
\subset\mathfrak{q}_{\{\alpha_0\}}
\intertext{
is that}
\label{eq:fb} 
\alpha_0\in\Phi\cap\bar{\mathcal{Q}}^n_{\Phi}
\quad\text{and}\quad
\alpha_0\not\in
{\bigcup}_{\beta\in(\mathcal{B}\setminus\Phi)\cup\Phi^-}\mathrm{supp}(\bar\beta).
\end{gather}
\end{lem}
\begin{proof}
First we show that \eqref{eq:fb} implies \eqref{eq:fa}.
If $\alpha_0\in\Phi\cap\bar{\mathcal{Q}}^n_{\Phi}$, then 
$\bar\alpha_0\succ{0}$, because all roots
in $\mathcal{Q}^n_{\Phi}$ are $C$-positive, and hence, in particular,
$\alpha_0\notin
\mathcal{R}_{\mathrm{im}}$.\par
We have $\mathfrak{q}_{\Phi}\subset\mathfrak{q}_{\{\alpha_0\}}$,
because $\{\alpha_0\}\subset\Phi$. To prove that also
$\bar{\mathfrak{q}}_{\Phi}\subset\mathfrak{q}_{\{\alpha_0\}}$,
it suffices to show that $\mathfrak{q}_{\{\alpha_0\}}^n\subset
\bar{\mathfrak{q}}_{\Phi}^n$. Assume by contradiction that
this inclusion is false. Then there is a root $\alpha$ with
$\alpha\succ \alpha_0$ and
$\alpha\notin\bar{\mathcal{Q}}^n_{\Phi}$, 
i.e. with $\bar\alpha\notin\mathcal{Q}^n_{\Phi}$.
If $\bar\alpha \succ 0$, then
$\mathrm{supp}(\bar\alpha)\cap\Phi=\emptyset$.
Being: \begin{equation}\label{eq:fd}
\alpha_0\in\mathrm{supp}(\alpha)\subset
{\bigcup}_{\beta\in\mathrm{supp}(\bar\alpha)}\mathrm{supp}(\bar\beta),
\end{equation} this would
imply that $\alpha_0\in\mathrm{supp}(\bar\beta)$ for
some $\beta\in\mathcal{B}\setminus\Phi$. Hence, by \eqref{eq:fb},
$\bar\alpha \prec 0$, and, from 
\eqref{eq:fd}, we obtain that $\alpha_0$ belongs to 
$\mathrm{supp}(\bar\beta)$ for some 
$\beta\in\left(\mathcal{B}\setminus\mathcal{R}_{\mathrm{im}}\right)$
with $\bar\beta\prec{0}$.
But then, because
$C$ is $S$-fit, $\beta\in\Phi^-$, yielding, by \eqref{eq:fb},  
a contradiction.
This completes the proof  that \eqref{eq:fb} implies \eqref{eq:fa}.
\par
Let us prove the opposite implication.
From
$\mathfrak{q}_{\Phi}\subset\mathfrak{q}_{\{\alpha_0\}}$,
we have that
$\alpha_0\in\Phi$. Condition \eqref{eq:fa} 
is equivalent to the inclusion 
$\mathfrak{q}^n_{\{\alpha\}}\subset\mathfrak{q}^n_{\Phi}\cap
\bar{\mathfrak{q}}^n_{\Phi}$. In particular, 
$\alpha_0\in\bar{\mathcal{Q}}^n_{\Phi}$, and, 
as we already have $\alpha_0\in\Phi$,
this brings the first half
of \eqref{eq:fb}. Let $\beta\in\mathcal{B}$
and assume that $\alpha_0\in\mathrm{supp}(\bar\beta)$.
If $\bar\beta\succ 0$, then $\bar\beta\in\mathcal{Q}_{\{\alpha_0}^n
\subset\mathcal{Q}^n\cap\bar{\mathcal{Q}}^n$ yields
$\beta\in\mathcal{Q}^n$ and hence, being simple, 
$\beta\in(\Phi\setminus\Phi^-)$.
Otherwise, $\bar\beta\prec 0$, and 
$-\bar\beta\in\mathcal{Q}_{\{\alpha_0\}}^n
\subset\mathcal{Q}^n_{\Phi}\cap\bar{\mathcal{Q}}^n_{\Phi}$. 
This would
imply that $-\beta\in\mathcal{Q}^n_{\Phi}$, 
but this is impossible
because $\mathcal{Q}^n_{\Phi}$
consists of $C$-positive roots.
\end{proof}
We recall that the \emph{basis of the fundamental reduction} of a $CR$ algebra
$(\mathfrak{g}_0,\mathfrak{q})$ is the $CR$ algebra 
$(\mathfrak{g}_0,\mathfrak{q}')$, where $\mathfrak{q}'$ is the smallest
complex Lie subalgebra of $\mathfrak{g}$ with
$\mathfrak{q}+\bar{\mathfrak{q}}\subset\mathfrak{q}'$
(see \cite[\S 5B]{MN05}).
\begin{thm}[Fundamental reduction]\label{thm:fb}
Let $(\mathfrak{g}_0,\mathfrak{q}_{\Phi})$
be a parabolic $CR$ algebra, with $\Phi\subset\mathcal{B}$,
where $\mathcal{B}$ is the set of $C$-positive simple roots for
an $S$-fit Weyl chamber $C$. Let:
\begin{gather}
\label{eq:fc}
\Phi^-=\{\alpha\in\Phi\,|\,\bar\alpha \prec 0\},\\
\label{eq:ff}
\Psi=\big(\Phi\cap\bar{\mathcal{Q}}_{\Phi}^n\big)\setminus
\big(\;{\bigcup}_{\alpha\in
(\mathcal{B}\setminus\Phi)
\cup{\Phi^-}}\mathrm{supp}(\bar\alpha)
\big).
\end{gather}
Then: 
\begin{enumerate}
\item
$(\mathfrak{g}_0,\mathfrak{q}_{\Psi})$ is the 
basis of the fundamental reduction
of  $(\mathfrak{g}_0,\mathfrak{q}_{\Phi})$.
\item
A necessary and sufficient condition
for $(\mathfrak{g}_0,\mathfrak{q}_{\Phi})$ to be fundamental is that
$\Psi=\emptyset$, i.e. that:
\begin{equation}\label{eq:fg}
\Phi\cap\bar{\mathcal{Q}}_{\Phi}^n\subset
{\bigcup}_{\alpha\in
(\mathcal{B}\setminus\Phi)\cup{\Phi^-}}
\mathrm{supp}(\bar\alpha).
\end{equation}
\end{enumerate}
Let $M_{\Phi}$ and $M_{\Psi}$ be the 
${{\mathbf{G}_0}}$-orbits associated to
$(\mathfrak{g}_0,\mathfrak{q}_{\Phi})$,
$(\mathfrak{g}_0,\mathfrak{q}_{\Psi})$, respectively. 
Then the ${{\mathbf{G}_0}}$-equivariant fibration 
$M_{\Phi}\xrightarrow{\pi}{M}_{\Psi}$ 
is $CR$ and all connected components of 
its fibers are parabolic $CR$ manifolds 
 (cf. Definition \ref{dfn:ca})
of finite type. In particular, $M_{\Phi}$ is of finite type
if and only if \eqref{eq:fg} holds true.
\end{thm}
\begin{proof}
The complex subalgebra $\mathfrak{q}'$ yielding the
fundamental reduction $(\mathfrak{g}_0,\mathfrak{q}')$ of
$(\mathfrak{g}_0,\mathfrak{q}_{\Phi})$ is a complex 
subalgebra of $\mathfrak{g}$ that contains $\mathfrak{q}_{\Phi}$,
and hence is parabolic and
of the form $\mathfrak{q}'=\mathfrak{q}_{\Phi'}$ for
some $\Phi'\subset\Phi$.
Thus $\mathfrak{q}'$ is the intersection
of all $\mathfrak{q}_{\{\alpha_0\}}$,
with $\alpha_0\in\Phi$, for which
$\mathfrak{q}_{\Phi}+\bar{\mathfrak{q}}_{\Phi}\subset
\mathfrak{q}_{\{\alpha_0\}}$.  
By Lemma \ref{lem:fa}, we have $\Phi'=\Psi$.
This proves (1) and (2). \par
The last statement is a consequence of
\cite[Theorem 5.3]{MN05}.
\end{proof}
Next we turn our consideration to weak nondegeneracy. First we prove:
\begin{prop}\label{prop:fc}
Let $(\mathfrak{g}_0,\mathfrak{q}_{\Phi})$
and $(\mathfrak{g}_0,\mathfrak{q}_{\Psi})$
be parabolic $CR$ algebras, with $\Psi\subset\Phi\subset\mathcal{B}$,
where $\mathcal{B}$ is the set of $C$-positive simple roots for
a $V$-fit Weyl chamber $C$.
Let $M_{\Phi}$, $M_{\Psi}$ be the 
corresponding ${{\mathbf{G}_0}}$-orbits. Then the
${{\mathbf{G}_0}}$-equivariant fibration:
\begin{equation}\label{eq:fh}
\begin{CD}
M_{\Phi}@>\pi>>M_{\Psi}
\end{CD}
\end{equation}
is a $CR$ fibration with complex fibers if and only if:
\begin{equation}\label{eq:fi}
\bar\alpha \prec 0\quad\forall\alpha\in\Phi
\setminus\Psi.
\end{equation}
\end{prop}
\begin{proof}
The necessary and sufficient condition for \eqref{eq:fh} being
a $CR$ fibration with complex fiber is that:
\begin{equation}\label{eq:fm}
\mathfrak{q}_{{\Phi}}\subset\mathfrak{q}_{\Psi}\subset
\mathfrak{q}_{{\Phi}}+\bar{\mathfrak{q}}_{{\Phi}}\end{equation}
(see \cite[Corollary 5.6]{MN05}). We need to show that this condition
is equivalent to \eqref{eq:fi}. It suffices to discuss
the situation where ${\Phi}\setminus{\Psi}$ consists of
a single simple root. So we assume that $\alpha_0\in{\Phi}$ and
${\Psi}={\Phi}\setminus\{\alpha_0\}$.
\par
Assume that \eqref{eq:fm} holds true. Then 
$-\alpha_0\in\bar{\mathcal{Q}}_{{\Phi}}$,
i.e. $-\bar\alpha_0\in\mathcal{Q}_{{\Phi}}$. In particular,
since $-\alpha_0\notin\mathcal{Q}_{{\Phi}}$, $\alpha_0$ is not real.
Thus, if $\bar\alpha_0\in\mathcal{Q}^r_{{\Phi}}$, then 
$\bar\alpha_0\prec{0}$, 
because $C$ is $V$-fit for $(\mathfrak{g}_0,\mathfrak{q}_{{\Phi}})$.
Otherwise,
$-\bar\alpha_0\in\mathcal{Q}^n_{{\Phi}}$
implies that $\bar\alpha_0 \prec 0$,
because $\mathcal{Q}^n_{{\Phi}}$
consists of $C$-positive roots.
\par
Vice versa, let us show that, if $\bar\alpha_0 \prec 0$,
then the parabolic set $\mathcal{Q}_{{\Psi}}$ of $\mathfrak{q}_{{\Psi}}$
is contained in $\mathcal{Q}_{{\Phi}}\cup\bar{\mathcal{Q}}_{{\Phi}}$.
This is
equivalent to 
$\mathcal{Q}^n_{{\Psi}}\supset\mathcal{Q}^n_{{\Phi}}
\cap\bar{\mathcal{Q}}^n_{{\Phi}}$.
Assume by contradiction that there is some
$\alpha\in(\mathcal{Q}^n_{{\Phi}}\cap\bar{\mathcal{Q}}^n_{{\Phi}})
\setminus\mathcal{Q}^n_{{\Psi}}$. Then 
$\mathrm{supp}(\alpha)\cap{\Phi}=\{\alpha_0\}$, but this
yields $\bar\alpha \prec 0$.
Indeed, $\bar\beta\prec{0}$ for 
all non real roots $\beta$ in the support of $\alpha$,
because $C$ is $V$-fit
for $(\mathfrak{g}_0,\mathfrak{q}_{{\Phi}})$, and we assumed that
$\bar\alpha_0\prec{0}$.
This gives a 
contradiction, since 
$\bar\alpha\in\mathcal{Q}^n_{{\Phi}}$,
that consists of $C$-positive roots.
\end{proof}
We recall (see \cite[Lemma 5.7]{MN05}) that, given any $CR$ algebra
$(\mathfrak{g}_0,\mathfrak{q})$, there is a unique maximal complex
subalgebra $\mathfrak{q}'$ of $\mathfrak{g}$ with
$\mathfrak{q}\subset\mathfrak{q}'\subset(\mathfrak{q}+\bar{\mathfrak{q}})$.
The $CR$ algebra $(\mathfrak{g}_0,\mathfrak{q}')$ is 
(see \cite[\S 5C]{MN05}) the 
\emph{basis of the weakly nondegenerate reduction}
of $(\mathfrak{g}_0,\mathfrak{q})$.
\par
We obtain:
\begin{thm}[Weakly nondegenerate reduction]\label{thm:fd} 
Let $(\mathfrak{g}_0,\mathfrak{q}_{\Phi})$
be a parabolic $CR$ algebra, with $\Phi\subset\mathcal{B}$,
where $\mathcal{B}$ is the set of $C$-positive simple roots for
a $V$-fit Weyl chamber $C$.
 Set:
\begin{equation}\label{eq:fj}
\Psi=\{\alpha\in{\Phi}\,|\, \bar\alpha \succ 0\}.
\end{equation}
\par
Then the parabolic $CR$ algebra
$(\mathfrak{g}_0,\mathfrak{q}_{\Psi})$ is the basis of
the weakly nondegenerate reduction of 
$(\mathfrak{g}_0,\mathfrak{q}_{{\Phi}})$. \par
Let $M_{{\Phi}}$, $M_{\Psi}$ be the ${{\mathbf{G}_0}}$-orbits
corresponding to $(\mathfrak{g}_0,\mathfrak{q}_{{\Phi}})$,
$(\mathfrak{g}_0,\mathfrak{q}_{\Psi})$, respectively. Then the
${{\mathbf{G}_0}}$-orbit 
$M_{\Psi}$ is holomorphically nondegenerate and the
${{\mathbf{G}_0}}$-equivariant fibration
$M_{\Phi}\xrightarrow{\pi}M_{\Psi}$
is a $CR$ fibration with 
simply connected (but not necessarily connected)
complex fibers.\par
In particular, $M_{\Phi}$ is holomorphically nondegenerate if and only
if:
\begin{equation}\label{eq:fk}
\bar\alpha \succ 0\quad\forall\alpha\in\Phi.
\end{equation}
\end{thm}
\begin{proof}
All complex subalgebras $\mathfrak{q}'$ containing
$\mathfrak{q}_{\Phi}$ are parabolic and of the form 
$\mathfrak{q}'=\mathfrak{q}_{\Psi}$
for some $\Psi\subset\Phi$. Thus, by Proposition \ref{prop:fc}, condition
\eqref{eq:fk} is necessary and sufficient for
$M_{\Phi}$ to be holomorphically nondegenerate. \par
In general, we observe that, 
if $C$ is $V$-fit for $(\mathfrak{g}_0,\mathfrak{q})$,
with $\mathfrak{q}=\mathfrak{q}_{\Phi}$, and $\Psi$ is defined by
\eqref{eq:fj}, then the chamber $C$ is $V$-fit also for
$(\mathfrak{g}_0,\mathfrak{q}_{\Psi})$, and 
therefore $M_{\Psi}$ is holomorphically nondegenerate by the
argument above.
Proposition \ref{prop:fc}
tells us that $M_{\Phi}\xrightarrow{\pi}M_{\Psi}$ 
is a $CR$ fibration with a
holomorphically nondegenerate basis and a complex fiber.
Then the statement follows by the uniqueness of the
weakly nondegenerate reduction (see \cite[Lemma 5.7]{MN05})
and from the last statement in Theorem \ref{thm:eh}. 
\end{proof}
\section{A structure theorem for 
$\mathbf{G}_0$ -orbits}\label{sec:h}
Let $M$ be any $\mathbf{G}_0$-orbit, with associated parabolic $CR$
algebra $(\mathfrak{g}_0,\mathfrak{q})$.  It is easy to construct a
smooth $\mathbf{G}_0$-equivariant fibration, with complex fibers, of
$M$ onto a real flag manifold $M'$ of $\mathbf{G}_0$.  Fix indeed a
Cartan pair $(\vartheta,\mathfrak{h}_0)$ adapted to
$(\mathfrak{g}_0,\mathfrak{q})$, and let $\mathcal{R}$ be the roots of
$\mathfrak{g}$ with respect to the complexification $\mathfrak{h}$ of
$\mathfrak{h}_0$. Choose any element $A$ of the facet in
$\mathfrak{h}_{\mathbb{R}}$ corresponding to $\mathfrak{q}$
(cf. \cite[VIII \S 4]{Bou75}). Then the facet of $B=A+\bar{A}$ defines
a parabolic $\mathfrak{q}'$ with
$\mathfrak{q}\cap\bar{\mathfrak{q}}\subset
\mathfrak{q}'=\bar{\mathfrak{q}}'\subset(\mathfrak{q}+\bar{\mathfrak{q}})$.
The $\mathbf{G}_0$-equivariant fibration $M\to{M}'$ of $M$ onto the
parabolic $CR$ manifold corresponding to
$(\mathfrak{g}_0,\mathfrak{q}')$ has all the required properties.  A
different 
canonical construction is given by Wolf (cf. \S\ref{sec:k}). 
The construction below is also canonical, and  
allows a better control of the structure of
the complex fibers.
\par\smallskip 
Starting from a parabolic $CR$ algebra
$(\mathfrak{g}_0,\mathfrak{q})$, we construct recursively a sequence
of parabolic $CR$ algebras $(\mathfrak{g}_0,\mathfrak{q}^{(h)})$ by
setting, for all $h\geq 0$\,:
\begin{equation}\label{eq:hb}
\left\{\begin{aligned}
\big(\mathfrak{g}_0,\mathfrak{q}^{(0)}\big)\;\; &=\text{
the weakly nondegenerate reduction of $(\mathfrak{g}_0,\mathfrak{q})$}\\
\big(\mathfrak{g}_0,\mathfrak{q}^{(h+1)}\big)&=\text{
the weakly nondegenerate reduction of $(\mathfrak{g}_0,\mathfrak{q}^{(h)}_w)$}.
\end{aligned}\right.
\end{equation}
We recall that $(\mathfrak{g}_0,\mathfrak{q}^{(h)}_w)$ is
the $CR$-weakening of $(\mathfrak{g}_0,\mathfrak{q}^{(h)})$,
defined in Theorem~\ref{thm:ec}. The 
weakly nondegenerate reduction
was described in Theorem~\ref{thm:fd}. Denote by 
$M^{(h)}$ and by $M^{(h)}_w$ the ${{\mathbf{G}_0}}$-orbits
associated to the $CR$ algebras $(\mathfrak{g}_0,\mathfrak{q}^{(h)})$ and
$(\mathfrak{g}_0,\mathfrak{q}^{(h)}_w)$, respectively.
We know that there is a $\mathbf{G}_0$-equivariant diffeomorphism
$f_h:M_w^{(h)}\to{M}^{(h)}$, 
which is also a $CR$ map, but in general not
a $CR$ diffeomorphism. Let $\psi_h:M^{(h)}\to{M}^{(h+1)}$ the
composition of the weakly nondegenerate reduction
$\pi_h:M_w^{(h)}\to{M}^{(h+1)}$ and of the inverse of $f_h$.
We have a commutative diagram:
\vspace{-3pt}
\begin{equation*}
  \begin{CD}
    M^{(h)}_w@>{f_h}>>M^{(h)}\\
@V{\pi_h}VV  @VV{\psi_h}V\\
M^{(h+1)}@=M^{(h+1)}. 
  \end{CD}\vspace{6pt}
\end{equation*}
Denoting by $\pi:M\to{M}^{(0)}$ the weakly nondegenerate reduction,
we obtain a sequence of smooth ${{\mathbf{G}_0}}$-equivariant maps:
\begin{equation}\label{eq:hc}\begin{CD}
M@>{\pi}>>M^{(0)}@>{\psi_0}>>M^{(1)}@>{\psi_1}>>
M^{(2)}@>>>\cdots
\end{CD}
\end{equation}
In general, the maps $\psi_h$ \textit{are not} $CR$.
Since
\begin{equation}\label{eq:hd}
\mathrm{dim}\,M^{(h)}\leq
\mathrm{dim}\,M^{(h+1)}\quad  \forall h\geq 0,
\end{equation}
and, by Proposition~\ref{prop:ee}, equality 
in \eqref{eq:hd} holds if and only if
$(\mathfrak{g}_0,{\mathfrak{q}}^{(h)}_w)$
is totally real, 
the sequence $\{\mathfrak{q}^{(h)}\}$
stabilizes and we have,
for some \mbox{integer $m\geq 0$},
\begin{equation}\label{eq:he}  
\mathfrak{q}^{(h)}=\mathfrak{q}^{(h)}_w=\mathfrak{q}^{(m)}=\mathfrak{e}, 
\quad\forall h\geq{m}.
\end{equation}
\begin{dfn}\label{dfn:ha}
The limit manifold $M_{\mathfrak{e}}$ of \eqref{eq:hc}, 
which is the $\mathbf{G}_0$-orbit associated to the $CR$ algebra
$(\mathfrak{g}_0,\mathfrak{e})$, is called the \emph{real core} of
$M$.
\end{dfn}
\begin{thm}[Structure theorem]\label{thm:ha}
Let $(\mathfrak{g}_0,\mathfrak{q})$ be a parabolic $CR$ algebra,
$M$ the corresponding ${{\mathbf{G}_0}}$-orbit, and $M_{\mathfrak{e}}$,
with $CR$ algebra $(\mathfrak{g}_0,\mathfrak{e})$, its real core.  
Then:
\begin{enumerate}
\item $(\mathfrak{g}_0,\mathfrak{e})$ is totally real,
i.e. $\bar{\mathfrak{e}}=\mathfrak{e}$;
\item $M_{\mathfrak{e}}$ is a real flag
manifold and the ${{\mathbf{G}_0}}$-equivariant projection:
\begin{equation}\label{eq:ha}
\begin{CD}
M@>\psi>> M_{\mathfrak{e}}
\end{CD} 
\end{equation}
has complex fibers,
with finitely many connected components, that are
simply connected;
\item each connected component $F_0$ of a fiber $F$
of \eqref{eq:ha} is a tower of 
holomorphic fibrations,
in which each fiber is the Cartesian product of
a Euclidean complex space and an open orbit of
a real form in a complex flag manifold.
\end{enumerate}
\end{thm}
\begin{proof}
We already established (1) and the first statement of (2).
The weakly nondegenerate reduction
${M}^{(h)}_w\to{M}^{(h+1)}$ is a $CR$ fibration for which we know,
by Theorem~\ref{thm:eh} and Theorem~\ref{thm:fd}, that the 
fiber has finitely many components and each connected component is
a Cartesian product of a Euclidean complex space
and an open orbit of a real form in a
complex flag manifold, hence we have also (2) and (3).
\end{proof} 
\section{Connected components of the fibers}\label{sec:g}
We keep the notation of \S\ref{sec:c}.
Let $(\vartheta,\mathfrak{h}_0)$
be a Cartan pair, adapted to the parabolic $CR$ algebra 
$(\mathfrak{g}_0,\mathfrak{q})$ of the
$\mathbf{G}_0$-orbit $M$. Denote by
$\mathbf{H}_0=\mathbf{Z}_{{{\mathbf{G}_0}}}(\mathfrak{h}_0)$ 
the Cartan subgroup
of ${{\mathbf{G}_0}}$ corresponding to the Cartan subalgebra $\mathfrak{h}_0$.
\begin{dfn}\label{def:ga}
The \emph{analytic, or real, Weyl group}
$\mathbf{W}({{\mathbf{G}_0}},\mathfrak{h}_0)$ is the quotient 
by $\mathbf{H}_0$ of
the normalizer of $\mathfrak{h}_0$ in~${{\mathbf{G}_0}}$:
\begin{equation}
\mathbf{W}({{\mathbf{G}_0}},\mathfrak{h}_0)=
\mathbf{N}_{{{\mathbf{G}_0}}}(\mathfrak{h}_0)/
\mathbf{Z}_{{{\mathbf{G}_0}}}(\mathfrak{h}_0).
\end{equation}
\end{dfn}
Since we took ${{\mathbf{G}_0}}$ connected, the analytic Weyl group
$\mathbf{W}({{\mathbf{G}_0}},\mathfrak{h}_0)$
only depends, modulo isomorphisms, 
upon the pair $(\mathfrak{g}_0,\mathfrak{h}_0)$, and not
on the particular choice of ${{\mathbf{G}_0}}$. We also have (see e.g.
\cite[p.489]{Kn:2002})
\begin{equation}
\mathbf{W}({{\mathbf{G}_0}},\mathfrak{h}_0)
=\mathbf{N}_{\mathbf{K}_0}(\mathfrak{h}_0)/
\mathbf{Z}_{\mathbf{K}_0}(\mathfrak{h}_0),
\end{equation}
where $\mathbf{K}_0$ is the compact analytic subgroup of ${{\mathbf{G}_0}}$
with Lie algebra $\mathfrak{k}_0=\{X\in\mathfrak{g}_0\mid
\vartheta(X)=X\}$. Since, by complexification,
$\mathbf{N}_{{{\mathbf{G}_0}}}(\mathfrak{h}_0)$ 
acts on $\mathfrak{h}_{\mathbb{R}}=
\mathfrak{h}_0^-\oplus{i}\mathfrak{h}_0^+$, and then, by duality, on
$\mathfrak{h}_{\mathbb{R}}^*$ and, by restriction, 
on the root system $\mathcal{R}=\mathcal{R}(\mathfrak{g},\mathfrak{h})$, 
we obtain a homomorphism, 
that in fact is an
inclusion:
\begin{equation}
\mathbf{W}({{\mathbf{G}_0}},\mathfrak{h}_0)
\hookrightarrow \mathbf{W}(\mathcal{R}),
\end{equation}
of the analytic Weyl group 
$\mathbf{W}({{\mathbf{G}_0}},\mathfrak{h}_0)$ into the Weyl group 
 $\mathbf{W}(\mathcal{R})$
of the root system
$\mathcal{R}$.
\par
We shall also consider the Weyl group:
\begin{equation}
\mathbf{W}(\mathbf{L}_0,\mathfrak{h}_0)=
\mathbf{N}_{\mathbf{L}_0}(\mathfrak{h}_0)/
\mathbf{Z}_{\mathbf{L}_0}(\mathfrak{h}_0)
\end{equation}
of the reductive subgroup
$\mathbf{L}_0$. Since 
$\mathbf{Z}_{\mathbf{L}_0}(\mathfrak{h}_0)=\mathbf{H}_0$, this 
can be viewed as a subgroup 
of the analytic Weyl group $\mathbf{W}({{\mathbf{G}_0}},\mathfrak{h}_0)$.
\par
Let $\mathbf{S}_0$ 
be the analytic semisimple subgroup 
of $\mathbf{L}_0$ with Lie
algebra $\mathfrak{s}_0=[\mathfrak{l}_0,\mathfrak{l}_0]$
(see Corollary \ref{cor:ce}). 
The elements of $\mathbf{S}_0$ 
centralize $\mathfrak{z}_0$. Thus,
for the Cartan subalgebra
$\mathfrak{t}_0=\mathfrak{h}_0\cap\mathfrak{s}_0$ of $\mathfrak{s}_0$,
since $\mathfrak{h}_0=\mathfrak{t}_0\oplus\mathfrak{z}_0$,
we get:
\begin{equation}\label{eq:ge}
\mathbf{N}_{\mathbf{S}_0}(\mathfrak{t}_0)=
\mathbf{N}_{\mathbf{S}_0}(\mathfrak{h}_0)\subset
\mathbf{N}_{\mathbf{L}_0}(\mathfrak{h}_0).
\end{equation}
The Cartan subgroup 
$\mathbf{T}_0$ of $\mathbf{S}_0$ corresponding to $\mathfrak{t}_0$
is:
\begin{equation}\label{eq:gf}
\mathbf{T}_0=\mathbf{Z}_{\mathbf{S}_0}(\mathfrak{t}_0)=
\mathbf{Z}_{\mathbf{S}_0}(\mathfrak{h}_0)=\mathbf{H}_0\cap\mathbf{S}_0.
\end{equation}
Thus the inclusion \eqref{eq:ge} yields, by passing to the quotients,
an inclusion of the Weyl groups:
\begin{equation}\label{eq:gg}
\mathbf{W}(\mathbf{S}_0,\mathfrak{t}_0)=\frac{\mathbf{N}_{\mathbf{S}_0}(
\mathfrak{t}_0)}{\mathbf{Z}_{\mathbf{S}_0}(
\mathfrak{t}_0)}\simeq\frac{\mathbf{N}_{\mathbf{S}_0}(
\mathfrak{t}_0)\,\mathbf{Z}_{\mathbf{L}_0}(\mathfrak{h}_0)}{
\mathbf{Z}_{\mathbf{L}_0}(
\mathfrak{h}_0)}
\subset
\mathbf{W}(\mathbf{L}_0,\mathfrak{h}_0).
\end{equation}
\begin{prop}\label{prop:gb} Let $M$ be a $\mathbf{G}_0$-orbit,
with associated parabolic $CR$ algebra
$(\mathfrak{g}_0,\mathfrak{q})$. 
Let $(\vartheta,\mathfrak{h}_0)$ be a Cartan pair adapted to
$(\mathfrak{g}_0,\mathfrak{q})$, $\mathbf{L}_0$ the $\vartheta$-invariant
reductive Levi factor of the isotropy $\mathbf{I}_0$ of $M$,
and $\mathbf{S}_0$ the maximal analytic semisimple Lie subgroup of
$\mathbf{L}_0$.
Then:
\begin{enumerate}
\item $\mathbf{N}_{\mathbf{S}_0}(\mathfrak{h}_0)$ is a 
closed normal
subgroup of $\mathbf{N}_{\mathbf{L}_0}(\mathfrak{h}_0)$ and
the natural inclusion $\mathbf{N}_{\mathbf{L}_0}(\mathfrak{h}_0)
\hookrightarrow\mathbf{I}_0$ yields a group isomorphism:
\begin{equation}\label{eq:gh}
\pi_0(\mathbf{N}_{\mathbf{L}_0}(\mathfrak{h}_0)/
\mathbf{N}_{\mathbf{S}_0}(\mathfrak{h}_0))
\xrightarrow{\;\sim\;}\pi_0(\mathbf{I}_0).
\end{equation}
\item We have the exact sequences:
\begin{gather}\label{eq:gi}\begin{CD}
\pmb{1}@>>>\frac{\mathbf{H}_0\;}{\mathbf{T}_0}@>>>
\frac{\mathbf{N}_{\mathbf{L}_0}(\mathfrak{h}_0)}{\mathbf{N}_{\mathbf{S}_0}
(\mathfrak{t}_0)}@>>>\frac{\mathbf{W}(\mathbf{L}_0,
\mathfrak{h}_0)}{
\mathbf{W}(\mathbf{S}_0,\mathfrak{t}_0)}@>>>\pmb{1},
\end{CD}
\\[10pt]
\label{eq:gj} \begin{CD}
\pmb{1}@>>>\pi_0\big(\!\frac{\mathbf{H}_0\;}{\mathbf{T}_0}\!\big)
@>>>
\pi_0(\mathbf{I}_0)
@>>>
\frac{\mathbf{W}(\mathbf{L}_0,\mathfrak{h}_0)}{
\mathbf{W}(\mathbf{S}_0,\mathfrak{t}_0)}@>>>\pmb{1}.
\end{CD}
\end{gather}
\end{enumerate}
\end{prop}
\begin{proof}
By \eqref{eq:cc}, we have 
$\pi_0(\mathbf{L}_0)\simeq\pi_0(\mathbf{I}_0)$.
\par
Let $\mathfrak{g}_0=\mathfrak{k}_0\oplus\mathfrak{p}_0$ be the
$\vartheta$-invariant
Cartan decomposition of $\mathfrak{g}_0$,
and
let $\mathfrak{a}_0$ be any maximal Abelian subalgebra of
$\mathfrak{p}_0$ that contains 
$\mathfrak{h}_0^-=\mathfrak{h}_0\cap\mathfrak{p}_0$.
Its centralizer $\mathbf{Z}_{\mathbf{K}_{00}}(\mathfrak{a}_0)$
in the maximal compact subgroup $\mathbf{K}_{00}$
of $\mathbf{L}_0$
intersects all connected components of $\mathbf{L}_0$ (see e.g.
\cite[Proposition 7.33]{Kn:2002}). Thus a fortiori
$\mathbf{Z}_{\mathbf{K}_{00}}(\mathfrak{h}_0^-)$,
containing $\mathbf{Z}_{\mathbf{K}_{00}}(\mathfrak{a}_0)$, intersects all
connected components of $\mathbf{L}_0$. The toroidal part
$\mathfrak{h}_0^+=\mathfrak{h}_0\cap\mathfrak{k}_0$ 
of $\mathfrak{h}_0$ is a Cartan subalgebra of the compact
Lie algebra
$\mathbf{Z}_{\mathfrak{k}_{00}}(\mathfrak{h}_0^-)$. Let
$g_0\in\mathbf{Z}_{\mathbf{K}_{00}}(\mathfrak{h}_0^-)$. Then
$\mathrm{Ad}(g_0)(\mathfrak{h}_0^+)$ is also a Cartan
subalgebra of $\mathbf{Z}_{\mathfrak{k}_{00}}(\mathfrak{h}_0^-)$.
Since all Cartan subalgebras of a compact Lie algebra are
conjugate by an inner automorphism, we can find
an element $g_1$, in the connected component $\mathbf{K}_{00}^0$
of the identity of $\mathbf{K}_{00}$, such that
$\mathrm{Ad}(g_1g_0)(\mathfrak{h}_0^+)=\mathfrak{h}_0^+$.
The element
$g_1g_0\in\mathbf{Z}_{\mathbf{L}_0}(\mathfrak{h}_0^-)
\cap\mathbf{N}_{\mathbf{L}_0}(\mathfrak{h}_0^+)$
is in the same connected component of $g_0$ in $\mathbf{K}_{00}$.
This shows that 
$\mathbf{Z}_{\mathbf{L}_0}(\mathfrak{h}_0^-)
\cap\mathbf{N}_{\mathbf{L}_0}(\mathfrak{h}_0^+)$, and therefore
also $\mathbf{N}_{\mathbf{L}_0}(\mathfrak{h}_0)$,
intersect all connected components of $\mathbf{L}_0$.
\par
Let $\mathbf{Z}_0^0$ be the analytic subgroup 
of $\mathbf{L}_0$ with Lie algebra
$\mathfrak{z}_0$. Then 
the connected component $\mathbf{L}_0^0$ of
 the identity in $\mathbf{L}_0$ is the direct product
\mbox{$\mathbf{S}_0\bowtie\mathbf{Z}_0^0$.}
It follows that the 
intersection $\mathbf{N}_{\mathbf{L}_0}(\mathfrak{h}_0)\cap
\mathbf{L}_0^0$ is 
\mbox{$\mathbf{N}_{\mathbf{S}_0}(\mathfrak{h}_0)\bowtie\mathbf{Z}_0^0$,} 
and this yields~\eqref{eq:gh}. 
\par
The exactness of \eqref{eq:gj} is a consequence of that 
of  \eqref{eq:gi} and of the isomorphism \eqref{eq:gh}. It
will suffice then to prove the exactness of \eqref{eq:gi}. \par
By the definition 
of the Weyl group $\mathbf{W}(\mathbf{L}_0,\mathfrak{h}_0)$,
we have an exact sequence\,:
\begin{equation}\begin{CD}\label{eq:gk}
\pmb{1}@>>>\mathbf{H}_0@>>>\mathbf{N}_{\mathbf{L}_0}(\mathfrak{h}_0)
@>>>\mathbf{W}(\mathbf{L}_0,\mathfrak{h}_0)@>>>\pmb{1}.
\end{CD}
\end{equation}
The image in $\mathbf{W}(\mathbf{L}_0,\mathfrak{h}_0)$ 
of the subgroup
$\mathbf{N}_{\mathbf{S}_0}(\mathfrak{t}_0)$
of $\mathbf{N}_{\mathbf{L}_0}(\mathfrak{h}_0)$, under the
projection into the quotient, 
using the identification
\eqref{eq:gg}, is 
$\mathbf{W}(\mathbf{S}_0,\mathfrak{t}_0)$.
Finally, the Cartan subgroup $\mathbf{T}_0$ of
$\mathbf{S}_0$ is the intersection 
$\mathbf{H}_0\cap\mathbf{N}_{\mathbf{S}_0}(\mathfrak{h}_0)$. Then
\eqref{eq:gi} follows from \eqref{eq:gk} by the 
elementary group homomorphism
theorems. \end{proof}
\begin{rmk}
Note that the number of connected components of the isotropy
subgroup $\mathbf{I}_0$ depends on the choice of ${{\mathbf{G}_0}}$.
The exact sequence  \eqref{eq:gj} exhibits this number 
as a product of a term, $|\pi_0(\mathbf{H}_0/\mathbf{T}_0)|$,
that genuinely depends on ${{\mathbf{G}_0}}$, and another
term, $\left|\left(\mathbf{W}(\mathbf{L}_0,\mathfrak{h}_0)/
\mathbf{W}(\mathbf{S}_0,\mathfrak{t}_0)\right)\right|$,
which is the same for all possible choices of 
 the connected linear Lie group ${{\mathbf{G}_0}}$
with Lie algebra~$\mathfrak{g}_0$.
Indeed, 
modulo isomorphisms, the groups
$\mathbf{W}(\mathbf{L}_0,\mathfrak{h}_0)$ and
$\mathbf{W}(\mathbf{S}_0,\mathfrak{t}_0)$ only depend on the
Lie algebras $\mathfrak{g}_0$, $\mathfrak{l}_0$ and $\mathfrak{h}_0$
(see e.g. \cite[p.489]{Kn:2002}).
\end{rmk}
Our main application of Proposition \ref{prop:gb} is
counting the
number of connected components of the fibers of a
${{\mathbf{G}_0}}$-equivariant fibration \eqref{eq:eb}. We have:
\begin{thm} \label{thm:gc}
Let $M$, $M'$ be $\mathbf{G}_0$-orbits, with associated
parabolic $CR$ algebras
$(\mathfrak{g}_0,\mathfrak{q})$,
$(\mathfrak{g}_0,\mathfrak{q}')$, and isotropy
subalgebras $\mathfrak{i}_0\supset\mathfrak{i}_0'$,
respectively.
Let $(\vartheta,\mathfrak{h}_0)$ be a Cartan pair adapted to
$(\mathfrak{g}_0,\mathfrak{q}')$, with
$\mathfrak{h}_0$ a maximally noncompact Cartan subalgebra
of~$\mathfrak{i}_0'$.
Then the group of connected
components of the typical fiber $F$ of the $\mathbf{G}_0$-equivariant
fibration $\phi:M'\rightarrow{M}$~is
\begin{equation}\label{eq:gl}
\pi_0(F)\simeq \mathbf{W}(\mathbf{L}_0,\mathfrak{h}_0)/
\mathbf{W}(\mathbf{S}_0,\mathfrak{t}_0),
\end{equation}  
where $\mathbf{L}_0$ is the reductive 
$\vartheta$-invariant Levi factor of $\mathbf{I}_0$, having
maximal semisimple analytic subgroup
$\mathbf{S}_0$, with
Lie algebra $\mathfrak{s}_0$, and $\mathfrak{t}_0=\mathfrak{h}_0\cap
\mathfrak{s}_0$.
\end{thm}
\begin{proof}
Since $\mathfrak{h}_0$ is a maximally noncompact 
Cartan subalgebra of $\mathfrak{i}_0'$, by
Theorem \ref{thm:cf} we have the isomorphism
$\pi_0(\mathbf{H}_0/\mathbf{T}_0')\simeq\pi_0({\mathbf{I}_0'})$,
where $\mathbf{T}'_0$ is the centralizer of $\mathbf{H}_0$ in
$\mathbf{S}'_0$, yielding a commutative diagram
\begin{equation*}
  \begin{CD}
@.    \pi_0(\mathbf{H}_0/\mathbf{T}'_0)@>{\simeq}>>\pi_0(\mathbf{I}'_0)\;\\
@. @VVV  @VVV\\
1 @>>>\pi_0(\mathbf{H}_0/\mathbf{T}_0)@>>>\pi(\mathbf{I}_0).
  \end{CD}
\end{equation*}
Since the map $\pi_0(\mathbf{H}_0/\mathbf{T}'_0)\to\pi_0(\mathbf{H}_0/
\mathbf{T}_0)$ is surjective,
from \eqref{eq:gj} we obtain the exact sequence:
\begin{equation}
\begin{CD}
\pi_0({\mathbf{I}_0'})
@>>>
\pi_0(\mathbf{I}_0)@>>>
\displaystyle{\frac{\mathbf{W}(\mathbf{L}_0,\mathfrak{h}_0)}{\mathbf{W}
(\mathbf{S}_0,\mathfrak{t}_0)}}@>>>\pmb{1},
\end{CD}
\end{equation}
yielding
\eqref{eq:gl}.
\end{proof}
\section{The fundamental group} \label{sec:i}
We use
Theorem \ref{thm:ha} and the results of \S\ref{sec:g}
to compute the fundamental group of the $\mathbf{G}_0$-orbits,
extending the results of \cite[\S 8]{AMN06}.
We have:
\begin{thm}\label{thm:ia}
Let $M$ be a ${{\mathbf{G}_0}}$-orbit,
$(\mathfrak{g}_0,\mathfrak{q})$ 
its associated parabolic $CR$ algebra, 
and
${M}_{\mathfrak{e}}$ its real core (see Definition \ref{dfn:ha}), 
with $CR$ algebra
$(\mathfrak{g}_0,\mathfrak{e})$.\par
Let
$(\vartheta,\mathfrak{h}_0)$ be a Cartan pair adapted to
$(\mathfrak{g}_0,\mathfrak{q})$, with
$\mathfrak{h}_0$ maximally noncompact in
$\mathfrak{i}_0$.
Set
$\mathbf{L}_0^{\mathfrak{e}}$, $\mathbf{S}^{\mathfrak{e}}_0$ 
for the reductive Levi factor of the isotropy $\mathbf{I}_0^{\mathfrak{e}}$
of ${M}_{\mathfrak{e}}$ and for its maximal analytic semisimple subgroup,
respectively. \par
Then we have an exact sequence:
\begin{equation}\label{eq:ia}
\begin{CD}
\pmb{1}@>>>\pi_1(M)@>>>\pi_1(M_{\mathfrak{e}})@>>>\frac{\mathbf{W}(
\mathbf{L}^{\mathfrak{e}}_0,\mathfrak{h}_0)}{\mathbf{W}(
\mathbf{S}^{\mathfrak{e}}_0,\mathfrak{h}_0\cap\mathfrak{s}^{\mathfrak{e}}_0)}
@>>>\pmb{1}.
\end{CD}
\end{equation}
In particular, the image of the map $\pi_1(M)\to\pi_1(M_{\mathfrak{e}})$
in \eqref{eq:ia}
is a normal subgroup
of $\pi_1(M_{\mathfrak{e}})$, with a finite index.\par
The map $\mathrm{H}_1(M,\mathbb{Q})\to\mathrm{H}_1(M_{\mathfrak{e}}
,\mathbb{Q})$
of the rational homologies, 
induced by the 
$\mathbf{G}_0$-equivariant projection $M\to{M}_e$,
is an isomorphism.
\end{thm}
\begin{proof} By (3) of Theorem \ref{thm:ha}, the fundamental group of
the typical fiber $F$ of \eqref{eq:ha} is trivial. Thus,
since $M$ is connected, the exact homotopy
sequence of a locally trivial fiber bundle yields the short 
exact sequence:
\begin{equation}\label{eq:ic}
\begin{CD}
\pmb{1}@>>>\pi_1(M)@>>>\pi_1(M_{\mathfrak{e}})@>>>\pi_0(F)@>>>\pmb{1}.
\end{CD}
\end{equation}
The exactness of \eqref{eq:ia}
follows then from the isomorphism \eqref{eq:gl}
established in
Theorem \ref{thm:gc}.\par
Since $F=\mathbf{I}_0^{\mathfrak{e}}/\mathbf{I}_0$, we have
the exact sequence:
\begin{equation}\label{eq:id}
\begin{CD}
\pi_0(\mathbf{I}_0)
@>>>\pi_0(\mathbf{I}_0^{\mathfrak{e}})@>>>\pi_0(F)@>>>\pmb{1}.
\end{CD}
\end{equation}
By Theorem \ref{thm:cf}, the groups $\pi_0(\mathbf{I}_0)$
and $\pi_0(\mathbf{I}_0^{\mathfrak{e}})$ are Abelian. Hence
$\pi_0(F)$, being in a one-to-one correspondence with a quotient
of finite Abelian groups, may be given
the structure of a finite Abelian group, for which 
the maps $\pi_0(\mathbf{I}_0^{\mathfrak{e}})\to\pi_0(F)$
and 
$\pi_1(M_{\mathfrak{e}})\to\pi_0(F)$ are group epimorphisms.
Thus the image of 
$\pi_1(M)\to\pi_1(M_{\mathfrak{e}})$ in \eqref{eq:ia}
is a normal subgroup of $\pi_1(M_{\mathfrak{e}})$,
being the kernel of a group homomorphism. 
The last assertion follows from the fact that the kernel of the
homomorphism $\pi_1(M)\to\pi_1(M_{\mathfrak{e}})$ is a torsion subgroup.
\end{proof} 
\begin{rmk}
We see from \eqref{eq:ia} that the fundamental group $\pi_1(M)$
only depends on the totally real parabolic 
$CR$ algebra $(\mathfrak{g}_0,\mathfrak{e})$ and from the
maximally noncompact Cartan subalgebra $\mathfrak{h}_0$ of
$\mathfrak{i}_0$.
\end{rmk}
We show by the following proposition 
that, vice versa, for each real flag manifold
$M_{\mathfrak{e}}$ of $\mathbf{G}_0$,
and every Cartan pair $(\vartheta,\mathfrak{h}_0)$
adapted to $(\mathfrak{g}_0,\mathfrak{e})$, we can find a
$\mathbf{G}_0$-orbit whose fundamental group is given by 
\eqref{eq:ia}. We have indeed
\begin{prop}\label{prop:ib}
Let $(\vartheta,\mathfrak{h}_0)$ be an adapted Cartan pair for
a totally real
parabolic $CR$ algebra $(\mathfrak{g}_0,\mathfrak{y})$. Then we can
find a parabolic subalgebra $\mathfrak{q}$ such that:
\begin{align}\label{eq:ie}
&\mathfrak{q}\subset\mathfrak{y}=\mathfrak{q}+
\bar{\mathfrak{q}},\\\label{eq:if}
&\mathfrak{h}_0\;\text{is a maximally noncompact Cartan subalgebra of
$\mathfrak{i}_0=\mathfrak{q}\cap\mathfrak{g}_0$.}
\end{align}
Denote by $M$ the $\mathbf{G}_0$-orbit with $CR$ algebra
$(\mathfrak{g}_0,\mathfrak{q})$. Then we have the exact sequence
\begin{equation}
  \label{eq:ih}
  \begin{CD}
    \pmb{1}@>>>\pi_1(M)@>>>\pi_1(M_{\mathfrak{y}})@>>>
\frac{\mathbf{W}(\mathbf{L}_{0}^{\mathfrak{y}},\mathfrak{h}_0)}{
\mathbf{W}(\mathbf{S}_0^{\mathfrak{y}},\mathfrak{h}_0
\cap\mathfrak{s}_0^{\mathfrak{y}})}@>>>\pmb{1},
  \end{CD}
\end{equation}
where $M_{\mathfrak{y}}$, the $\mathbf{G}_0$-orbit
associated to $(\mathfrak{g}_0,\mathfrak{y})$,
 is a 
real flag manifold, with isotropy subgroup $\mathbf{I}_0^{\mathfrak{y}}$,
$\mathbf{L}^{\mathfrak{y}}_0$ is the reductive part of 
$\mathbf{I}_0^{\mathfrak{y}}$, and $\mathbf{S}^{\mathfrak{y}}_0$
the maximal analytic
semisimple subgroup of $\mathbf{L}^{\mathfrak{y}}_0$,
with Lie algebra $\mathfrak{s}_0^{\mathfrak{y}}$.
\end{prop}
Note that $M_{\mathfrak{y}}$ does not necessarily coincide with
the real core $M_{\mathfrak{e}}$ of $M$.
\begin{proof}
Let $\mathcal{R}$ be the root system of $\mathfrak{g}$
with respect to the complexification
$\mathfrak{h}$ of $\mathfrak{h}_0$ and let $C$ be a $V$-fit
Weyl chamber of $\mathcal{R}$ for $(\mathfrak{g}_0,\mathfrak{y})$.
Let $\mathfrak{y}=\mathfrak{q}_{\Phi}$ for a subset $\Phi$
of the set
$\mathcal{B}$ of the $C$-positive simple roots.
Set
\begin{equation}
  \label{eq:ig}
  \Psi=\Phi\cup\{\alpha\in\mathcal{B}\,|\,\bar\alpha
 \prec 0\}
\end{equation}
and take $\mathfrak{q}=\mathfrak{q}_{\Psi}$. 
If $\alpha \succ 0$ is imaginary, then 
$\mathrm{supp}(\alpha)$ contains some $\beta_0$ with
$\bar\beta_0 \prec 0$.
Thus, by \eqref{eq:ig}, 
$\beta_0\in\Psi$, and therefore $\alpha\in\mathcal{Q}_{\Psi}^n$,
showing that
$\mathcal{Q}^r_{\Psi}$ does not contain any imaginary root.
This implies that $\mathfrak{h}_0$ is maximally noncompact in
$\mathfrak{i}_0$. 
Moreover, the equality in \eqref{eq:ie} is valid because
of \eqref{eq:fi} of Proposition \ref{prop:fc}.
Hence
both 
\eqref{eq:ie} and \eqref{eq:if} are satisfied.\par
The exactness of \eqref{eq:ih} follows because 
the $\mathbf{G}_0$-equivariant $CR$ map $M\to{M}_{\mathfrak{y}}$
is the weakly nondegenerate reduction, and hence the fibers 
are
simply connected by Theorem \ref{thm:fb}, so that the argument
in the proof of Theorem \ref{thm:ia} applies.
\end{proof}
In the last part of this section, we shall give an explicit  
description
of the fundamental group $\pi_1(M)$ of the $\mathbf{G}_0$ 
orbit $M$.\par
Fix a Cartan pair
$(\vartheta,\mathfrak{h}_0)$ adapted to
the parabolic $CR$ algebra $(\mathfrak{g}_0,\mathfrak{q})$
of $M$. Let $M_{\mathfrak{e}}$, with
$CR$ algebra $(\mathfrak{g}_0,\mathfrak{e})$, be the real core
of $M$. 
Keeping the Cartan involution $\vartheta$
and the corresponding Cartan decomposition
$\mathfrak{g}_0=\mathfrak{k}_0\oplus\mathfrak{p}_0$ fixed, we denote by 
$\mathfrak{h}_0^{\mathfrak{e}}$ a
$\vartheta$-invariant 
maximally noncompact Cartan subalgebra of $\mathfrak{i}_0^{\mathfrak{e}}$.
We choose, as we can, $\mathfrak{h}_0^{\mathfrak{e}}$ in such a way that
$\mathfrak{h}_0^{\mathfrak{e}}
\cap\mathfrak{k}_0\subset
\mathfrak{h}_0^+$ and 
$\mathfrak{h}_0^{\mathfrak{e}}
\cap\mathfrak{p}_0\supset
\mathfrak{h}_0^-$. 
\par
Let $\mathcal{R}^{\mathfrak{e}}$ be the root system of 
$\mathfrak{g}$ with respect to the complexification
$\mathfrak{h}^\mathfrak{e}$ of $\mathfrak{h}_0^{\mathfrak{e}}$,
and 
$\mathcal{E}$ the parabolic set
of $\mathfrak{e}$  in $\mathcal{R}^{\mathfrak{e}}$. 
We recall that, for a real root $\alpha\in\mathcal{R}^{\mathfrak{e}}$,
the real eigenspace:
\begin{equation}
\mathfrak{g}_0^{\alpha}=\{X\in\mathfrak{g}_0\,|\, [H,X]=\alpha(H)\,X,
\quad\forall H\in\mathfrak{h}_0^{\mathfrak{e}}\cap\mathfrak{p}_0\}
\end{equation}
is not trivial. Its real dimension is called the \emph{multiplicity}
of~$\alpha$.\par
Let $C\in\mathfrak{C}(\mathcal{R}^{\mathfrak{e}},\mathcal{E})$ 
be an $S$-fit Weyl chamber for 
$(\mathfrak{g}_0,\mathfrak{e})$,
denote by $\mathcal{B}$ the basis of $C$-positive simple roots in
$\mathcal{R}^{\mathfrak{e}}$,
 and $\Phi^{\mathfrak{e}}=\mathcal{B}\cap\mathcal{E}^{n}$.
The roots of $\mathcal{B}$ correspond to the nodes of a
Satake diagram and, in particular,
$\bar\alpha \succ 0$ for all non imaginary
$\alpha\succ 0$.\par
\smallskip
We utilize 
\cite{Wig98} to describe 
the fundamental group $\pi_1(M_{\mathfrak{e}})$ in terms of a  
set $\Gamma$ of 
generators, given by \eqref{eq:ik},
and by the relations 
\eqref{eq:il} 
below:
\begin{align}\label{eq:ik} 
&\Gamma=\{\xi_{\alpha}\,|\,\alpha\in\mathcal{B}
\cap\mathcal{R}^{\mathfrak{e}}_{\mathrm{re}}
\;\text{has multiplicity $1$}\} \\\label{eq:il}
&\xi_{\alpha}=1 \quad\text{if $\alpha\notin
\Phi^{\mathfrak{e}}$},\quad
\xi_{\alpha}\xi_{\beta}=\xi_{\beta}\xi_{\alpha}^{(-1)^{(\alpha|\beta^{\lor})}}
\quad\forall\xi_{\alpha},\xi_{\beta}\in\Gamma.
\end{align}
In \eqref{eq:il} we use 
the standard notation $\beta^{\lor}=2\beta/\|\beta\|^2$.\par
From Theorem \ref{thm:ia} and this description
of $\pi_1(M_{\mathfrak{e}})$, we get:
\begin{cor}\label{cor:id}
Let $M$ be a $\mathbf{G}_0$-orbit, with associated parabolic
$CR$ algebra $(\mathfrak{g}_0,\mathfrak{q})$, and $M_{\mathfrak{e}}$
its real core.
If $\mathfrak{g}_0$ is a real semisimple Lie algebra such that all its
simple ideals are\footnote{Here we follow,  
for  labeling the simple real Lie algebras,
\cite[Table $\mathrm{V\! I}$, Chapter X]{Hel78}.} 
\begin{equation}\tag{a}\label{eq:it}
{\begin{aligned}
 &\text{either of the complex
type, or compact, or of one of the real types}\\ 
&\text{$\mathrm{AI\! I}$, $\mathrm{AI\! I\! I a}$,
$\mathrm{AI\! V}$, $\mathrm{BI\! I}$, $\mathrm{CI\! I}$,
$\mathrm{DI\! I}$, $\mathrm{DI\! I\! I b}$,
$\mathrm{EI\! I\! I }$, $\mathrm{EI\! V}$,
$\mathrm{FI\! I}$,}   
  \end{aligned}}
\end{equation}
then all ${{\mathbf{G}_0}}$-orbits are simply connected.\par
If we allow the simple ideals to be either of the types listed
in \eqref{eq:it} or
of the real types:
\begin{equation}
\tag{b}
  \label{eq:iu}
 \text{
$\mathrm{AI\! I\! I b}$ and $\mathrm{DI\! I\! I a}$},\end{equation}
then the map
$\pi_1(M)\to\pi_1(M_{\mathfrak{e}})$ of \eqref{eq:ic} is an isomorphism.
\end{cor}
\begin{proof}
Every ${{\mathbf{G}_0}}$-orbit splits into the Cartesian product of
${{\mathbf{G}_0}}_i$-orbits, each corresponding to a simple ideal
${\mathfrak{g}_0}_i$ of $\mathfrak{g}_0$ (see e.g. \cite[p.490]{AMN06}).
Thus we can reduce to the case where $\mathfrak{g}_0$ is simple.
Consider a maximally noncompact Cartan subalgebra 
$\mathfrak{h}_0^{\mathfrak{e}}$
of $\mathfrak{g}_0$. Let $\mathcal{R}^{\mathfrak{e}}$ be the root system of
$\mathfrak{g}$ with respect to the complexification  
$\mathfrak{h}^{\mathfrak{e}}$ of $\mathfrak{h}^{\mathfrak{e}}_0$.
The first assertion follows from the fact that,
in the cases listed in $(a)$, 
if $C$ is any
$S$-fit chamber for $(\mathfrak{g}_0,\mathfrak{e})$,
then $\mathcal{B}$ does not contain any simple real root
with multiplicity one. Hence
we have
$\pi_1(M_{\mathfrak{e}})=\pmb{1}$ 
and thus also $M$ 
is simply connected.\par
The last statement is a consequence of the fact that
in the cases listed in \eqref{eq:iu}, the quotient 
$\mathbf{H}_0/\mathbf{Z}({{\mathbf{G}_0}})$ of a Cartan subgroup
$\mathbf{H}_0$ 
of ${{\mathbf{G}_0}}$ by its center $\mathbf{Z}({{\mathbf{G}_0}})$ is always
connected.
Then, by the exact sequence \eqref{eq:id}, the fiber $F$ is connected,
and therefore \eqref{eq:ic} yields 
an isomorphism of
the fundamental groups of $M$ and $M_{\mathfrak{e}}$.
\end{proof}
\begin{rmk}
We note that Proposition \ref{prop:ee} and Theorem \ref{thm:fd} provide
an effective construction of $\mathfrak{e}$, and hence of
$M_{\mathfrak{e}}$, starting from a representation 
$\mathfrak{q}=\mathfrak{q}_{\Phi}$ in terms of the root system
associated to a Cartan pair $(\vartheta,\mathfrak{h}_0)$ adapted
to $(\mathfrak{g}_0,\mathfrak{q})$.
\end{rmk}
Fix a Chevalley basis $\{X_{\alpha}\}_{\alpha\in\mathcal{R}^{\mathfrak{e}}}
\cup\{H_{\alpha}\}_{\alpha\in\mathcal{B}}$
of $\mathfrak{g}$ (see \cite{Bou75}), with:
\begin{equation*}
X_{\alpha}\in\mathfrak{g}^{\alpha}, \quad
\tau(X_{\alpha})=X_{-\alpha}, \quad
[H_{\alpha},X_{\pm\alpha}]=\pm 2X_{\alpha},\quad
[X_{\alpha},X_{-\alpha}]=-H_{\alpha}.
\end{equation*}
When $\alpha\in\mathcal{R}^{\mathfrak{e}}_{\mathrm{re}}$, we choose, as
we can,
$X_{\alpha}\in\mathfrak{g}_0$.\par
Let $\xi_\alpha\in\Gamma$.
The element
$\exp(i\pi{H}_{\alpha})$ belongs 
to the normalizer of $\mathfrak{h}_0$ in $\mathbf{L}_0^{\mathfrak{e}}$, 
and
the map in \eqref{eq:ia}
transforms
$\xi_{\alpha}$ into the equivalence class of  $\exp(i\pi{H}_{\alpha})$
in
${\mathbf{W}(\mathbf{L}_0^{\mathfrak{e}},\mathfrak{h}_0)}/{
\mathbf{W}(\mathbf{S}^{\mathfrak{e}}_0,\mathfrak{h}_0
\cap\mathfrak{s}^{\mathfrak{e}}_0)}$.
\par
The Cartan subalgebra $\mathfrak{h}_0$ is obtained from 
$\mathfrak{h}_0^{\mathfrak{e}}$ by the Cayley transform with respect to a set
$\alpha_1,\hdots,\alpha_m$ of pairwise strongly orthogonal real roots
in $\mathcal{E}^r$. With 
$\mathfrak{h}^{\mathfrak{e}}_{\mathbb{R}}=(\mathfrak{h}^{\mathfrak{e}}_0\cap
\mathfrak{p}_0)+i(\mathfrak{h}^{\mathfrak{e}}_0\cap
\mathfrak{k}_0)$
and 
$\mathfrak{h}_{\mathbb{R}}=\mathfrak{h}_0^-
\oplus\, i\,\mathfrak{h}_0^+$, the Cayley transform maps
$\mathfrak{h}^{\mathfrak{e}}_{\mathbb{R}}$ onto 
$\mathfrak{h}_{\mathbb{R}}$~by\,:
\begin{equation}\label{eq:in}
\lambda
:\mathfrak{h}^{\mathfrak{e}}_{\mathbb{R}}
\ni H \longrightarrow
H+\tfrac{i}{2}{\sum}_{j=1}^m{\alpha_j(H)\,(iH_{\alpha_j}+X_{\alpha_j}+
X_{-\alpha_j})}\in\mathfrak{h}_{\mathbb{R}}.
\end{equation}
Since 
$\alpha(H)=0$ for all $H\in\mathfrak{z}_0^{\mathfrak{e}}$ and
$\alpha\in\mathcal{E}^r=\bar{\mathcal{E}}^r$,
the Cayley transform is the identity on
$\mathfrak{z}_0^{\mathfrak{e}}\subset{\mathfrak{h}_0^{\mathfrak{e}}}
\cap\mathfrak{h}_0$.
\par
For a real $\beta\in\Phi^{\mathfrak{e}}$,
the action of $\exp(i\pi{H}_{\beta})$ on 
$\mathfrak{h}_{\mathbb{R}}$ is described by
\begin{equation} \begin{aligned}
&\mathrm{Ad}(\exp(i\pi{H}_{\beta}))\big(
H+
\tfrac{i}{2}
{\sum}_{j=1}^m{\alpha_j(H)\,(iH_{\alpha_j}+X_{\alpha_j}+
X_{-\alpha_j})}\big)\\
&\qquad=
H+\tfrac{i}{2}{\sum}_{j=1}^m{\alpha_j(H)\,\left(iH_{\alpha_j}+
e^{i\pi(\alpha_j|\beta^{\lor})}(X_{\alpha_j}+
X_{-\alpha_j})\right)}.\end{aligned}
\end{equation}
By duality, the Cayley transform defines a map 
$\lambda^*:\mathfrak{h}_{\mathbb{R}}^*
\to{\mathfrak{h}^{\mathfrak{e}}_{\mathbb{R}}}^*$,
that gives by restriction a bijection 
$\mathcal{R}\to\mathcal{R}^{\mathfrak{e}}$
of the two root systems. 
Set:
\begin{equation}
[\lambda^*]^{-1}:\mathcal{R}^{\mathfrak{e}}\ni\alpha
\longrightarrow \alpha^{\lambda}\in\mathcal{R}.
\end{equation}
Then we obtain:
\begin{lem}\label{lem:if}
Let $\beta$ be any real root in $\Phi^{\mathfrak{e}}$. Then
$\mathrm{Ad}(\exp(i\pi{H}_{\beta}))$ defines 
in the Weyl group 
of 
$\mathcal{R}$ the element:
\begin{equation}\label{eq:io}
s_{\alpha_1^{\lambda}}^{(\alpha_1|\beta^{\lor})}\circ\cdots\circ
s_{\alpha_m^{\lambda}}^{(\alpha_m|\beta^{\lor})} \in\mathbf{W}(\mathcal{R}),
\end{equation}
where $s_{\alpha^\lambda_j}$ is the symmetry with respect to
$\alpha^\lambda_j\in\mathcal{R}$. \qed
\end{lem}
Let $\mathfrak{s}^{\mathfrak{e}}_{0\,(j)}$, for $j=1,\hdots,p$, 
be the simple ideals of $\mathfrak{s}_0^{\mathfrak{e}}$.
For each $j=1,\hdots,p$, 
denote by $\mathcal{R}_{(j)}$
the set of roots 
$\gamma\in\mathcal{R}=\mathcal{R}(\mathfrak{g},\mathfrak{h})$
that are of the form $\gamma=\alpha^{\lambda}$ for some $\alpha\in
\mathcal{R}^{\mathfrak{e}}$ 
for which the eigenspace $\mathfrak{g}^{\alpha}$ 
is contained
in the complexification $\mathfrak{s}_{(j)}^{\mathfrak{e}}$ of 
$\mathfrak{s}_{0\,(j)}^{\mathfrak{e}}$. The
$\mathcal{R}_{(j)}$'s are disjoint. Let $\mathbf{S}_{0\,(j)}^{\mathfrak{e}}$ 
be the analytic Lie subgroup of
$\mathbf{S}_0^{\mathfrak{e}}$ with Lie algebra 
${\mathfrak{s}}_{0\,(j)}^{\mathfrak{e}}$.
\par
For each $j=1,\hdots,p$, let $\mathcal{A}_{(j)}$ be the subset
of the set $\{\alpha_1,\hdots,\alpha_m\}$ of 
roots of $\mathcal{R}^{\mathfrak{e}}$, used to define
the Cayley transform \eqref{eq:in}, 
 consisting of those 
for which $\alpha_j^{\lambda}\in\mathcal{R}_{(j)}$. Since:
\begin{equation}\label{eq:ip}\begin{aligned}
\mathbf{S}_0^{\mathfrak{e}}&={\mathbf{S}}_{0\,(1)}^{\mathfrak{e}}\bowtie
\cdots\bowtie
{\mathbf{S}}_{0\,(p)}^{\mathfrak{e}}  
\qquad\text{and}\qquad\\
\mathbf{W}(\mathbf{S}_0^{\mathfrak{e}},\mathfrak{h}_0\cap
\mathfrak{s}_0^{\mathfrak{e}})
&=\mathbf{W}(
{\mathbf{S}}_{0\,(1)}^{\mathfrak{e}},\mathfrak{h}_0
\cap\mathfrak{s}_{0\,(1)}^{\mathfrak{e}})
\bowtie
\cdots\bowtie
\mathbf{W}(
{\mathbf{S}}_{0\,(p)}^{\mathfrak{e}},\mathfrak{h}_0
\cap\mathfrak{s}_{0\,(p)}^{\mathfrak{e}})
\end{aligned}\end{equation}
we obtain:  
\begin{thm}\label{thm:ig}
With the notation above: $\pi_1(M)$ is the subgroup
of $\pi_1(M_{\mathfrak{e}})$ consisting of the elements of the
form:
   \begin{equation}\label{eq:iq}
\xi=\xi_{\beta_1}^{k_1}\cdots\xi_{\beta_{\ell}}^{k_{\ell}},
\end{equation}    
 where
$\ell$ is a positive integer, and $k_1,\hdots,k_{\ell}\in\mathbb{Z}$,
$\xi_{\beta_1},\hdots,\xi_{\beta_{\ell}}\in\Gamma$ satisfy
one of the two equivalent conditions:
\begin{equation}
  \label{eq:ir}
 s_{\alpha_1}^{\sum_{j=1}^{\ell}{k_j(\alpha_1|\beta_j^{\lor})}}
\circ\cdots\circ
s_{\alpha_{m}}^{\sum_{j=1}^{\ell}{k_j(\alpha_m|\beta_j^{\lor})}}\in
\mathbf{W}(
\mathbf{S}_0^{\mathfrak{e}},\mathfrak{h}_0\cap\mathfrak{s}_0^{\mathfrak{e}}),
\qquad\text{or}
\end{equation}
\begin{equation}\label{eq:is} \left\{{\begin{aligned}
&s_{\alpha_{i_1}}^{\sum_{j=1}^{\ell}{k_j(\alpha_{i_1}^h|\beta_j^{\lor})}}
\circ\cdots\circ
s_{\alpha_{i_r}}^{\sum_{j=1}^{\ell}{k_j(\alpha_{i_{r_h}}^h|\beta_j^{\lor})}}\in
\mathbf{W}(
{\mathbf{S}}_{0\,(h)}^{\mathfrak{e}},\mathfrak{h}_0\cap
\mathfrak{s}_{0\,(h)}^{\mathfrak{e}})
\\
&\text{where}\quad
\mathcal{A}_{(h)}=\{\alpha_{i_1}^h,\hdots,\alpha_{i_{r_h}}^h\}
,\;\text{for}\quad h=1,\hdots,p.\qed
\end{aligned}}\right.
\end{equation}
\end{thm}
\begin{rmk}\label{rmk:ih}
When
$\mathbf{L}_0$ is a reductive real linear group, and $\mathfrak{h}_0$
a Cartan subalgebra of its Lie algebra $\mathfrak{l}_0$, 
the \emph{analytic} 
Weyl group $\mathbf{W}(\mathbf{L}_0,\mathfrak{h}_0)$
has been
explicitly computed (see e.g. \cite{Vog82, adams-2007}).
Thus Theorem \ref{thm:ig} yields an effective way to compute the
fundamental group of a ${{\mathbf{G}_0}}$-orbit $M$.\par
For the convenience of the reader, we give below a description
of that part of the Weyl group 
$\mathbf{W}(\mathbf{S}_0^{\mathfrak{e}},\mathfrak{h}_0
\cap\mathfrak{s}_0^{\mathfrak{e}})$,
that is needed to understand \eqref{eq:ir}, \eqref{eq:is}. 
We take,
as we can, $\mathbf{G}$ simply connected. Then also
the complexification ${\mathbf{S}^{\mathfrak{e}}}$ of
$\mathbf{S}^{\mathfrak{e}}_0$ is simply connected,
because $\mathbf{S}^{\mathfrak{e}}$ is the semisimple Levi
factor of a parabolic subgroup of $\mathbf{G}$. 
Indeed, each integral character of $\mathfrak{s}^{\mathfrak{e}}$
lifts to an integral character of $\mathfrak{g}$ and therefore
defines a linear representation of $\mathbf{G}$,
giving, by restriction,
a linear representation of $\mathbf{S}^{\mathfrak{e}}$.
It is well known that this property characterizes simple connectivity.
This is the situation where we can apply the results of
\cite{adams-2007, Vog82}.
\par
Let $\mathfrak{h}$ be the complexification of $\mathfrak{h}_0$, and
$\mathcal{S}$ the root system of 
${\mathfrak{s}}^{\mathfrak{e}}$ with respect to the
Cartan subalgebra 
$\mathfrak{t}=\mathfrak{h}\cap{\mathfrak{s}}^{\mathfrak{e}}$.
Denote by
$\mathcal{S}_{\mathrm{im,c}}$ and by $\mathcal{S}_{\mathrm{im,n}}$
the sets of compact and noncompact imaginary roots
of $\mathcal{S}$, respectively.
For a Weyl chamber $C\in\mathfrak{C}(\mathcal{S})$, set
$\rho_0=\frac{1}{2}\sum\{\beta\in\mathcal{S}_{\mathrm{im,c}}|
\beta\succ0\}$. Let $E$ be the subset of
$\mathcal{S}_{\mathrm{im,n}}$ of the 
$C$-positive noncompact imaginary roots that
are orthogonal to $\rho_0$. 
The elements of $E$ are pairwise 
strongly orthogonal, thus the subgroup $\mathbf{W}(E)$ generated
by the symmetries $s_{\alpha}$, for $\alpha\in{E}$, is isomorphic
to $\mathbb{Z}_2^p$, where $p\geq{0}$ is the number of elements
of $E$.
Let $\Lambda$ be the co-root lattice,
generated by the elements $\alpha^{\lor}=2\alpha/\|\alpha\|^2$ for
$\alpha\in\mathcal{S}$, 
and $\Lambda^*$ 
the sublattice generated by the $\alpha^{\lor}$ for
$\alpha\in\mathcal{S}_{\mathrm{im}}$. 
Let $\mathfrak{t}^*_{\mathbb{R}}$ be the linear span of $\mathcal{S}$ and
$\varpi$ the orthogonal projection of $\mathfrak{t}^*_{\mathbb{R}}$
onto the linear span of $\mathcal{S}_{\mathrm{im}}$. 
\par
Let $\mathbf{W}^*$ be
the subgroup of the Weyl group $\mathbf{W}(\mathcal{S})$
generated by the symmetries $s_{\alpha}$ for
$\alpha\in\mathcal{R}_{\mathrm{im,n}}$.
Then we have:
\begin{equation}
\mathbf{W}^*\cap
\mathbf{W}(\mathbf{S}_0^{\mathfrak{e}},\mathfrak{h}_0
\cap\mathfrak{s}_0^{\mathfrak{e}})
=\ker\left(f:\mathbf{W}(E)\to\Lambda^*/2\varpi(\Lambda)\right)
\end{equation}
where $f$ is the homomorphism that maps $s_{\alpha}$ to 
$\alpha^{\lor}$.
\par
We note that the elements \eqref{eq:iq}, having been constructed
from real roots of the root system $\mathcal{R}^{\mathfrak{e}}$,  
belong to the subgroup of $\mathbf{W}(\mathcal{R})$ generated by
the symmetries with respect to noncompact imaginary roots that
are orthogonal to $\rho_0$, and therefore 
to $\mathbf{W}^*$.
\end{rmk}
\section{The Mostow fibration}\label{sec:j}
Let $M={{\mathbf{G}_0}}/\mathbf{I}_0$ be the homogeneous space of a
Lie group ${{\mathbf{G}_0}}$, with $\mathbf{I}_0$ 
a closed subgroup of ${{\mathbf{G}_0}}$.
We assume that both
${{\mathbf{G}_0}}$ and
$\mathbf{I}_0$ have finitely many connected components.
We fix maximal compact subgroups $\mathbf{K}_0$ and
$\mathbf{K}_{00}$ of ${{\mathbf{G}_0}}$ and 
of $\mathbf{I}_0$, respectively, with $\mathbf{K}_{00}\subset
\mathbf{K}_{0}$.
In \cite{Most55} and
\cite{Most62} 
G.D. Mostow proved that there exist closed Euclidean
subspaces $F$ and $E$ in ${{\mathbf{G}_0}}$ such that:
\begin{gather*}
\left.\begin{gathered}
\mathbf{K}_0\times{F}\times{E}\ni(k,f,e)\to k\cdot{f}\cdot{e}\in{{\mathbf{G}_0}}
\\
\quad \text{and}\quad
\mathbf{K}_{00}\times{E}\ni(k,e)\to k\cdot{e}\in\mathbf{I}_0
\end{gathered}\right\}
\text{are diffeomorphisms onto},\\[3pt]
\mathrm{ad}(k)(F)=F\quad\forall{k}\in\mathbf{K}_{00}.
\end{gather*}
In particular, $M$ is isomorphic, as a $\mathbf{K}_0$-space,
to the manifold $\mathbf{K}_0\times_{\mathbf{K}_{00}}F$, i.e.
to the quotient of the Cartesian product $\mathbf{K}_0\times{F}$
with respect to the equivalence relation that identifies
$(k,f)$ and $(k\cdot{k}_0^{-1},\mathrm{ad}(k_0)(f))$ if
$k\in\mathbf{K}_0$, $f\in{F}$, and $k_0\in\mathbf{K}_{00}$.
Let $N$ be the homogeneous space $\mathbf{K}_0/\mathbf{K}_{00}$. 
\begin{dfn}
The $\mathbf{K}_0$-equivariant fibration
$M\to{N}$ defined by the commutative diagram
\begin{equation}
  \label{eq:ja}\begin{CD}
 @. \mathbf{K}_0\times{F}@>>> \mathbf{K}_0\\
@. @VVV   @VVV\\
M@.\simeq\mathbf{K}_0\times_{\mathbf{K}_{00}}{F}@>>> N
\end{CD}
\end{equation}
is called \emph{the Mostow fibration} of $M$.
\end{dfn}
We come back now to the situation where $M$ is a
${{\mathbf{G}_0}}$-orbit 
in a complex flag manifold $X$, with
associated parabolic $CR$ algebra $(\mathfrak{g}_0,\mathfrak{q})$. 
We keep the notation of the previous sections. In particular,
we denote by  
$\mathbf{K}$ the
complexification of $\mathbf{K}_0$, and by
$\mathfrak{k}_0$, $\mathfrak{k}$ the Lie algebras of $\mathbf{K}_0$,
$\mathbf{K}$, respectively. Then
the basis $N$ of the Mostow
fibration is the intersection of $M$ with the complex
${\mathbf{K}}$-orbit $M^*$ in 
${X}=\mathbf{G}/\mathbf{Q}$,
that corresponds to $M$ in the 
Matsuki duality
(see \cite{Mats88}, \cite{BrLo02}). 
Note that $N$ is a $\mathbf{K}_0$-homogeneous generic $CR$-submanifold
of $M^*$, with associated $CR$ algebra
$CR$ algebra $(\mathfrak{k}_0,\mathfrak{q}\cap\mathfrak{k})$.
\begin{prop} \label{prop:jb}
The $CR$ manifold $M$ and the
basis $N$ of its Mostow fibration
$M\to{N}$ have the same $CR$-codimension. In particular, the
fibers of the Mostow fibration $M\to{N}$ have an even dimension.
\end{prop}
\begin{proof}
The $CR$ codimensions of $M$ and $N$ are (cf. Remark \ref{rmk:ac})
\begin{align*}
  CR\text{-}\mathrm{codim}\,{M}&=\mathrm{dim}_{\mathbb{C}}\mathfrak{g}-
\mathrm{dim}_{\mathbb{C}}(\mathfrak{q}+\bar{\mathfrak{q}})
\qquad\text{and}
\\
CR\text{-}\mathrm{codim}\,{N}&=\mathrm{dim}_{\mathbb{C}}\mathfrak{k}-
\mathrm{dim}_{\mathbb{C}}([\mathfrak{q}\cap\mathfrak{k}]
+[\bar{\mathfrak{q}}\cap\mathfrak{k}]).
\end{align*}
Thus, we need to prove that
\begin{equation}\label{eq:jg}
\mathrm{dim}_{\mathbb{C}}\mathfrak{g}=
\mathrm{dim}_{\mathbb{C}}(\mathfrak{q}+\bar{\mathfrak{q}})  
+\mathrm{dim}_{\mathbb{C}}\mathfrak{k}-
\mathrm{dim}_{\mathbb{C}}([\mathfrak{q}\cap\mathfrak{k}]
+[\bar{\mathfrak{q}}\cap\mathfrak{k}]).
\end{equation}
First we show that
\begin{equation}
  \label{eq:ji}
  (\mathfrak{q}+\bar{\mathfrak{q}})\cap\mathfrak{k}=
(\mathfrak{q}\cap\mathfrak{k})+(\bar{\mathfrak{q}}\cap
\mathfrak{k}).
\end{equation}
Fix any $\vartheta$-invariant Cartan subalgebra 
$\mathfrak{h}_0$ of $\mathfrak{i}_0$. Let
$\mathfrak{h}$ be the complexification of $\mathfrak{h}_0$ and
$\mathcal{R}=\mathcal{R}(\mathfrak{g},\mathfrak{h})$ 
be the root system of $\mathfrak{g}$ with respect to its
Cartan subalgebra $\mathfrak{h}$. 
Recall that
$\mathfrak{k}=\{X+\vartheta(X)\mid
X\in\mathfrak{g}\}$ and that, if $X_{\alpha}\in\mathfrak{g}^{\alpha}$, then
$\vartheta(X_{\alpha})\in\mathfrak{g}^{-\bar\alpha}$.
To prove \eqref{eq:ji},
it suffices to show that, if $X_{\alpha}\neq{0}$ and
$(X_{\alpha}+\vartheta(X_{\alpha}))\in(\mathfrak{q}+\bar{\mathfrak{q}})$,
then 
$\alpha$ and $-\bar\alpha$ both together belong either to
$\mathcal{Q}$, or to $\bar{\mathcal{Q}}$.
The assertion is trivially true when $\bar\alpha=-\alpha$.
When $\bar\alpha\neq -\alpha$, the fact that
$(X_{\alpha}+\vartheta(X_{\alpha}))\in(\mathfrak{q}+\bar{\mathfrak{q}})$
implies that $\{\alpha,-\bar\alpha\}\subset\mathcal{Q}\cup
\bar{\mathcal{Q}}$. If $\alpha\in\mathcal{Q}\setminus\bar{\mathcal{Q}}$
and $-\bar\alpha\in\bar{\mathcal{Q}}\setminus\mathcal{Q}$,
then neither $\bar\alpha$, nor $-\bar\alpha$ would belong to
$\mathcal{Q}$, contradicting the fact that $\mathcal{Q}$ is
a parabolic set. Analogously, we rule out the case where
$\alpha\in\bar{\mathcal{Q}}\setminus\mathcal{Q}$ and
$-\bar\alpha\in\mathcal{Q}\setminus\bar{\mathcal{Q}}$.
This proves \eqref{eq:ji}.
Since
\begin{equation*}
\mathrm{dim}_{\mathbb{C}}(\mathfrak{q}+\bar{\mathfrak{q}})  
+\mathrm{dim}_{\mathbb{C}}\mathfrak{k}-
\mathrm{dim}_{\mathbb{C}}([\mathfrak{q}
+\bar{\mathfrak{q}}]\cap\mathfrak{k})=
\mathrm{dim}_{\mathbb{C}}(\mathfrak{q}
+\bar{\mathfrak{q}}+\mathfrak{k}),
\end{equation*}
to complete the proof
that $M$ and $N$ have the same $CR$ codimension it suffices to
verify that:
\begin{equation}\label{eq:jb}
 \mathfrak{q}+\bar{\mathfrak{q}}+{\mathfrak{k}}=\mathfrak{g},
\end{equation}
and, to this aim, that all root spaces $\mathfrak{g}^{\alpha}$
are contained in
the left hand side of
\eqref{eq:jb}. 
If $\mathfrak{g}^{\alpha}\not\subset\mathfrak{q}+\bar{\mathfrak{q}}$,
then $\alpha$ is either real or complex, and 
$\mathfrak{g}^{-\bar\alpha}\subset\mathfrak{q}\cap\bar{\mathfrak{q}}$.
With $X_{\alpha}\in\mathfrak{g}^{\alpha}$, we have
$\vartheta(X_{\alpha})\in\mathfrak{g}^{-\bar\alpha}\subset
\mathfrak{q}\cap\bar{\mathfrak{q}}$ and
$(X_{\alpha}+
\vartheta({X}_{\alpha}))\in{\mathfrak{k}}$. Thus
$X_{\alpha}$ belongs to the left hand side of \eqref{eq:jb}.
This completes the proof.
\end{proof}
\begin{thm}\label{thm:jc}
Let $M$ be a ${{\mathbf{G}_0}}$-orbit,
with associated parabolic $CR$ algebra
$(\mathfrak{g}_0,\mathfrak{q})$, and $M\to{N}$ its Mostow fibration.
Let $M_{\mathfrak{e}}$ be the real core of $M$ (see Definition
\ref{dfn:ha}). Then there is a sequence of $\mathbf{K}_0$-equivariant
fibrations:
\begin{equation}
  \label{eq:jc}
 N={N}^{(0)}\to{N}^{(1)}\to\cdots\to{N}^{(m-1)}\to{N}^{(m)}=M_{\mathfrak{e}},
\end{equation}
in which, for each $h\geq 1$, fiber $L^{(h)}$ 
of the $\mathbf{K}_0$-equivariant
fibration ${N}^{(h-1)}\to{N}^{(h)}$ is diffeomorphic to a
disjoint union of a finite number of copies of a 
complex flag manifold. 
\end{thm}
\begin{proof} We consider the sequence \eqref{eq:hc}. For each
$h=0,1,\hdots,m$ we take $N^{(h)}$ to be the basis of the Mostow
fibration of the ${{\mathbf{G}_0}}$-orbit $M^{(h)}$. 
Let $\mathbf{I}_0^{(h)}$ be the isotropy subgroup of $M^{(h)}$
and $\mathbf{K}^{(h)}_{00}$ its maximal $\vartheta$-invariant compact 
subgroup. Being complex, by Theorem \ref{thm:eh},
each connected component of the fiber 
$F_h=F'_h\times{F}''_h$ of the
${{\mathbf{G}_0}}$-equivariant fibration
of $M^{(h-1)}_{w}\to{M}^{(h)}$ 
is $CR$-diffeomorphic to
the product of a complex Euclidean space 
$F''_h$ and of the disjoint union $F'_h$ of
finitely many copies of an open orbit
$\Omega^{(h)}$ 
in a 
complex flag manifold $X_h$. Then $L^{(h)}$ is
the basis of the Mostow fibration
$F'_h\to{L}^{(h)} $. Thus $L^{(h)}$ has finitely many connected components,
each diffeomorphic to  
a complex flag manifold, being equal to the Matsuki-dual
of an open orbit.
\end{proof}
Thus we have:
\begin{cor}\label{cor:jd} With the notation of Theorem \ref{thm:jc}:
  \begin{equation}
    \label{eq:je}
    \chi(N)=0\Longleftrightarrow \chi(M_{\mathfrak{e}})=0,
  \end{equation}
where $\chi$ is the Euler-Poincar\'e characteristic.\par
Let $\{M_{j}|j\in{J}\}$ be the set of all 
${{\mathbf{G}_0}}$-orbits in the complex flag manifold 
${X}=\mathbf{G}/\mathbf{Q}$. Then:
\begin{equation}
  \label{eq:jf}
\chi({X})=\sum_{j\in{J}}{\chi(N_{j})},  
\end{equation}
where $N_{j}$ is the basis of the Mostow fibration
$M_{j}\to{N}_{j}$.
\end{cor}
\begin{proof}
We keep the notation of Theorem \ref{thm:jc}. 
The Euler-Poincar\'e characteristic $\chi(N)$ is the product
$\chi(N)=\chi(M_{\mathfrak{e}})\cdot\chi(L^{(1)})\cdots
\chi(L^{(m)})$ and $\chi(L^{(h)})>0$ for $h=1,\hdots,m$, because each
$L^{(h)}$ is a complex flag manifold.\par
The last assertion follows by considering a cell decomposition
$\mathcal{C}$
of ${X}$ in which every cell is contained in some ${{\mathbf{G}_0}}$-orbit
$M_{j}$. Since the fibers of the Mostow fibration are Euclidean,
we may obtain a cell decomposition of $M_{j}$ from a cell
decomposition 
of the basis $N_{j}$ of its Mostow fibration.
Since the dimension of the fibers is even, the contribution of the
cells of $M_{j}$ to the alternated sum 
$\chi({X})=\sum_{c\in\mathcal{C}}(-1)^{\mathrm{dim}(c)}$
of the cells contained 
in $\chi({X})$ is exactly $\chi(N_{j})$.
\end{proof}
\section{Real core and
algebraic arc components}
\label{sec:k}
In this section we compare our construction of \S\ref{sec:h} with
the \emph{algebraic arc components} of \cite[\S\,8]{Wolf69}.
They were defined in the following way.
\begin{dfn}
  \label{dfn:ka} Let $M$ be a $\mathbf{G}_0$-orbit, with associated 
parabolic $CR$ algebra $(\mathfrak{g}_0,\mathfrak{q})$. Let
  $(\vartheta,\mathfrak{h}_0)$ be an adapted Cartan pair and
  $\mathcal{R}=\mathcal{R}(\mathfrak{g},\mathfrak{h})$.
  Define
  \begin{align}
    \label{eq:ka}
    \delta&=
{\sum}_{\alpha\in\mathcal{Q}^n\cap\bar{\mathcal{Q}}^n}
    {\alpha},\\\label{eq:kb}
    \mathcal{Q}_a&=\big\{\alpha\in\mathcal{R}\mid (\delta|\alpha)\geq 0
\big\},\\
\label{eq:kc}
    \mathfrak{q}_a&= \mathfrak{h}\oplus {\sum}_{\alpha\in\mathcal{Q}_a} 
{\mathfrak{g}^{\alpha}}.
  \end{align}
The $\mathbf{G}_0$-orbit $M_a$, corresponding to the 
parabolic $CR$ algebra
$(\mathfrak{g}_0,\mathfrak{q}_a)$, is called the \emph{space of algebraic
arc components} of $M$.
\end{dfn}
\begin{lem}
  \label{lem:kb}
  If $(\mathfrak{g}_0,\mathfrak{q})$ and
  $(\mathfrak{g}_0,\mathfrak{q}')$ are parabolic $CR$ algebras with
  $\mathfrak{q}+\bar{\mathfrak{q}}= \mathfrak{q}'+\bar{\mathfrak{q}}'$,
  then $\mathfrak{q}_a=\mathfrak{q}'_a$. This holds in particular when
  $(\mathfrak{g}_0,\mathfrak{q}')$ is the weakly nondegenerate
  reduction of $(\mathfrak{g}_0,\mathfrak{q})$.
\end{lem}
\begin{proof}
  Indeed $\mathcal{Q}^n\cap\bar{\mathcal{Q}}^n=\mathcal{R}
\setminus({-\mathcal{Q}}\cup{-\bar{\mathcal{Q}}})=
\mathcal{R}
\setminus({-\mathcal{Q}}'\cup{-\bar{\mathcal{Q}}'})
=\mathcal{Q}'{}^n
\cap\bar{\mathcal{Q}}'{}^n$.
\end{proof}
\begin{lem}\label{lem:kc}
Let $(\mathfrak{g}_0,\mathfrak{q})$ be a 
parabolic $CR$ algebra. Fix an
adapted Cartan pair $(\vartheta,\mathfrak{h}_0)$.
Let $\mathcal{R}$ be the root system of $\mathfrak{g}$ with respect to
the complexification $\mathfrak{h}$ of $\mathfrak{h}_0$ and
$\mathcal{Q}\subset\mathcal{R}$ the parabolic set of $\mathfrak{q}$.
Consider the complement of $\mathfrak{q}\cap\bar{\mathfrak{q}}$
in $\mathfrak{q}$ given by
\begin{equation}
  \label{eq:kd}
  \mathfrak{x}={\sum}_{\alpha\in\mathcal{Q}\setminus\bar{\mathcal{Q}}}{
\mathfrak{g}^{\alpha}}.
\end{equation}
If $(\mathfrak{g}_0,\mathfrak{q})$ is weakly nondegenerate, then
\begin{equation}
  \label{eq:ke}
  \{Z\in\mathfrak{x}\mid [Z,\bar{\mathfrak{q}}]\subset
(\mathfrak{q}+\bar{\mathfrak{q}})\}\subset\mathfrak{q}^n.
\end{equation}
\end{lem}
\begin{proof}
Set 
${\mathcal{Q}}^{-n}=\mathcal{R}\setminus{\mathcal{Q}}=
\{-\alpha\mid
\alpha\in\mathcal{Q}^n\}$ and
$\bar{\mathcal{Q}}^{-n}=\{\bar\alpha\mid\alpha\in{\mathcal{Q}}^{-n}\}$.
We apply an
argument similar to that of the proof of
\cite[Proposition~12.3]{AMN06}. 
The left hand side of \eqref{eq:ke} decomposes into a direct sum
of root spaces $\mathfrak{g}^{\alpha}$. 
Since 
\begin{equation*}
(\mathcal{Q}\setminus\bar{\mathcal{Q}})=(\mathcal{Q}^r\cap
\bar{\mathcal{Q}}^{-n})\cup(\mathcal{Q}^n\cap
\bar{\mathcal{Q}}^{-n}),
\end{equation*}
it suffices to show that 
the left hand side of \eqref{eq:ke}
does not contain $\mathfrak{g}^{\alpha}$
if $\alpha\in\mathcal{Q}^r\cap\bar{\mathcal{Q}}^{-n}$. 
By the assumption that 
$(\mathfrak{g}_0,\mathfrak{q})$ is weakly nondegenerate, 
by \cite[Theorem 6.2]{MN05} there
exist roots $\beta_1,\ldots,\beta_k\in\bar{\mathcal{Q}}$ such that:
\begin{enumerate}
\item[$(i)$]
$\gamma_i=\alpha+\sum_{j=1}^i\beta_i\in\mathcal{R}$ for
$j=1,\ldots,k$,
\item[$(ii)$]
$\gamma_k\notin\mathcal{Q}\cup\bar{\mathcal{Q}}$. 
\end{enumerate}
If we take
a sequence $(\beta_1,\ldots,\beta_k)$ 
satisfying $(i)$ and $(ii)$
with $k$ minimal, then $(i)$ and $(ii)$ hold true for
all sequences $(\beta_{s_1},\hdots,\beta_{s_k})$
obtained by a permutation of the $\beta_i$'s. 
At least one of the
$\beta_i$'s does not belong to $\mathcal{Q}$, so we can assume
that $\beta_1\in\mathcal{Q}^{-n}$. Then
$\gamma_1\in\mathcal{Q}^{-n}$, because $\alpha\in\mathcal{Q}^r$,
and $\mathcal{Q}^r\cup\mathcal{Q}^{-n}$ is a parabolic set.
We also have
$\gamma_1\in\bar{\mathcal{Q}}^{-n}$, because
$\beta_2,\ldots,\beta_k\in\bar{\mathcal{Q}}$ and
$\gamma_k=\gamma_1+\sum_{j=2}^k\beta_k\in\bar{\mathcal{Q}}^{-n}$. This
shows that $k=1$ and hence $\mathfrak{g}^{\alpha}$ is not contained
in the left hand side of \eqref{eq:ke}.
\end{proof}
\begin{lem}
  \label{lem:kd}
  If  $(\mathfrak{g}_0,\mathfrak{q})$ is weakly nondegenerate, then
  \begin{equation}
    \label{eq:kf}
\mathfrak{q}_w\cap\bar{\mathfrak{q}}_w  \subset 
\mathfrak{q}_a\subset\mathfrak{q}_w+\bar{\mathfrak{q}}_w.
  \end{equation}
\end{lem}
\begin{proof}
We keep the notation introduced in the proof of the previous lemma.
\par
Let us consider the normalizer of $\mathfrak{q}^n\cap\bar{\mathfrak{q}}^n$
in $\mathfrak{g}$:
  \begin{equation}
    \label{eq:kg}
\mathbf{N}_{\mathfrak{g}}(\mathfrak{q}^n\cap\bar{\mathfrak{q}}^n)
=\{Z\in\mathfrak{g}\mid [Z,\mathfrak{q}^n\cap\bar{\mathfrak{q}}^n]
\subset\mathfrak{q}^n\cap\bar{\mathfrak{q}}^n\}.
\end{equation}
The first inclusion in \eqref{eq:kf} is obvious, because 
$\mathfrak{q}_w\cap\bar{\mathfrak{q}}_w
=\mathfrak{q}\cap\bar{\mathfrak{q}}$, and
$\mathfrak{q}_a$ contains 
$\mathbf{N}_{\mathfrak{g}}(\mathfrak{q}^n\cap\bar{\mathfrak{q}}^n)\supset
\mathfrak{q}\cap\bar{\mathfrak{q}}$.
Since $\mathbf{N}_{\mathfrak{g}}(\mathfrak{q}^n\cap\bar{\mathfrak{q}}^n)$
is the direct sum of $\mathfrak{h}$ and of root spaces, we can decompose
$\mathbf{N}_{\mathfrak{g}}(\mathfrak{q}^n\cap\bar{\mathfrak{q}}^n)$
into the direct sum of its nilradical 
$\mathbf{N}_{\mathfrak{g}}^n(\mathfrak{q}^n\cap\bar{\mathfrak{q}}^n)$
and of a $\tau$-invariant reductive complement
$\mathbf{N}_{\mathfrak{g}}^r(\mathfrak{q}^n\cap\bar{\mathfrak{q}}^n)$ 
(recall that $\tau=\sigma\circ\vartheta$ is the conjugation with respect
to the compact form $\mathfrak{u}_0=\mathfrak{k}_0+i\mathfrak{p}_0$
of $\mathfrak{g}$). Wolf proved (see \cite[Theorem 8.5(2)]{Wolf69})
that
\begin{equation}
  \label{eq:kh}
\mathbf{N}_{\mathfrak{g}}^n(\mathfrak{q}^n\cap\bar{\mathfrak{q}}^n)
\subset\mathfrak{q}_a^n,\quad
\mathbf{N}_{\mathfrak{g}}^r(\mathfrak{q}^n\cap\bar{\mathfrak{q}}^n)
\subset\mathfrak{q}_a^r,  
\end{equation}
so that in particular
\begin{equation}
  \label{eq:ki}
\mathbf{N}_{\mathfrak{g}}^n(\mathfrak{q}^n\cap\bar{\mathfrak{q}}^n)=
 \mathbf{N}_{\mathfrak{g}}(\mathfrak{q}^n\cap\bar{\mathfrak{q}}^n) \cap
\mathfrak{q}_a^n,\quad
\mathbf{N}_{\mathfrak{g}}^r(\mathfrak{q}^n\cap\bar{\mathfrak{q}}^n)=
 \mathbf{N}_{\mathfrak{g}}(\mathfrak{q}^n\cap\bar{\mathfrak{q}}^n) \cap
\mathfrak{q}_a^r.
\end{equation}
To complete the proof of \eqref{eq:kf},
it suffices to show that
$\mathcal{Q}_a\subset(\mathcal{Q}_w\cup\bar{\mathcal{Q}}_w)$.
Since (see \cite[Theorem 8.5(2)]{Wolf69}) 
$\mathfrak{q}_a$
is contained in 
$\mathfrak{q}+\bar{\mathfrak{q}}$ and is invariant by conjugation,
we can assume, by
substituting, if needed, $\bar\alpha$ to $\alpha$, 
that $\alpha\in\mathcal{Q}$. 
Assume by contradiction that
$\alpha\notin\mathcal{Q}_w\cup\bar{\mathcal{Q}}_w$.
Since we have the partition
\begin{equation*}
  \mathcal{R}\setminus\big(\mathcal{Q}_w\cup\bar{\mathcal{Q}}_w\big)=
\big(\mathcal{Q}^r\cap\bar{\mathcal{Q}}^{-n}\big)\cup
\big(\bar{\mathcal{Q}}^r\cap{\mathcal{Q}}^{-n}\big)\cup
\big({\mathcal{Q}}^{-n}\cap\bar{\mathcal{Q}}^{-n}\big),
\end{equation*}
we obtain that
$\alpha\in\mathcal{Q}\setminus\big(\mathcal{Q}_w\cup\bar{\mathcal{Q}}_w\big)
=\mathcal{Q}^r\cap\bar{\mathcal{Q}}^{-n}$.
This implies that
$\mathfrak{g}^{-\alpha}$ is contained in 
$\mathbf{N}_{\mathfrak{g}}(\mathfrak{q}^n\cap\bar{\mathfrak{q}}^n)
\subset\mathfrak{q}_a$. Then $\mathfrak{g}^{\pm\alpha}\subset\mathfrak{q}^r_a$,
and hence, by \eqref{eq:ki}, 
$\mathfrak{g}^{\pm\alpha}\subset
\mathbf{N}_{\mathfrak{g}}^r(\mathfrak{q}^n\cap\bar{\mathfrak{q}}^n)$.
But $\mathbf{N}_{\mathfrak{g}}^r(\mathfrak{q}^n\cap\bar{\mathfrak{q}}^n)$
also normalizes $(\mathfrak{q}^{-n}\cap\bar{\mathfrak{q}}^{-n})$,
and hence $(\mathfrak{q}+\bar{\mathfrak{q}})$.
By Lemma \ref{lem:kc} this implies that 
${\alpha}\in\mathcal{Q}^n\subset\mathcal{Q}_w$, yielding a contradiction.
\end{proof}  
\begin{exam}\label{exam:kf}
  Denote by $\mathcal{F}^7_{d_1,d_2,\hdots,d_r}$ the complex manifold
  consisting of the flags $(\ell_1,\ell_2,\hdots,\ell_r)$ where
  $\ell_j$ is a $d_j$-dimensional linear complex subspace of
  $\mathbb{C}^7$ and
  $\ell_1\subsetneq\ell_2\subsetneq\cdots\subsetneq\ell_r$.\par
  Let $(\mathrm{e}_1,\ldots,\mathrm{e}_7)$ be a basis of the real form
  $\mathbb{R}^7$ of $\mathbb{C}^7$, and
  $(\epsilon_1,\ldots,\epsilon_7)$ the basis of $\mathbb{C}^7$ given
  by
  \begin{equation*}
    \begin{gathered}
      \epsilon_1=\mathrm{e}_1+i\mathrm{e}_7,\quad
      \epsilon_2=\mathrm{e}_2,\quad
      \epsilon_3=\mathrm{e}_3+i\mathrm{e}_6,\quad
      \epsilon_4=\mathrm{e}_4,\\
      \epsilon_5=\mathrm{e}_5,\quad
      \epsilon_6=\mathrm{e}_3-i\mathrm{e}_6,\quad
      \epsilon_7=\mathrm{e}_1-i\mathrm{e}_7.
    \end{gathered}
  \end{equation*}
Let $\mathbf{Q}$ be the Borel subgroup of 
$\mathbf{G}=\mathbf{SL}(7,\mathbb{C})$ that stabilizes the complete
  flag
  \begin{equation*}
    \varepsilon=\big(\langle\epsilon_1\rangle_{\mathbb{C}},
    \langle\epsilon_1,\epsilon_2\rangle_{\mathbb{C}}, \ldots,
    \langle\epsilon_1,\ldots,\epsilon_6\rangle_{\mathbb{C}}\big)\in
    \mathcal{F}^7_{1,2,3,4,5,6}.
  \end{equation*}
Set $\mathbf{G}_0=\mathbf{SL}(7,\mathbb{R})$, and
consider the orbit
  $M=\mathbf{G}_0\cdot\varepsilon\subset\mathcal{F}^7_{1,2,3,4,5,6} $, 
  with associated $CR$ algebra
  $(\mathfrak{g}_0,\mathfrak{q})$, where
  $\mathfrak{g}_0=\mathfrak{sl}(7,\mathbb{R})$ and $\mathfrak{q}$ is
  the Lie algebra of $\mathbf{Q}$.  The weakly nondegenerate reduction
  of $M$ is the $\mathbf{G}_0$-orbit
  $M^{(0)}\subset\mathcal{F}^7_{2,4}$ through the flag
  \begin{equation*}
    \varepsilon^0=\big(
    \langle\epsilon_1,\epsilon_2\rangle_{\mathbb{C}},
    \langle\epsilon_1,\epsilon_2,\epsilon_3,\epsilon_4\rangle_{\mathbb{C}}
    \big)\in\mathcal{F}^7_{2,4}.
  \end{equation*}
  \par
  Continuing with the construction of \S~\ref{sec:h} we obtain that
  $M^{(1)}\subset\mathcal{F}^7_{1,3,5,6}$ is the $\mathbf{G}_0$-orbit
  through the flag
  \begin{equation*}
    \varepsilon^1=\big(\langle\epsilon_2\rangle_{\mathbb{C}},
    \langle\epsilon_2,\epsilon_1,\epsilon_4\rangle_{\mathbb{C}}, \langle
    \epsilon_2, \epsilon_1, \epsilon_4, \epsilon_3, \epsilon_7
    \rangle_{\mathbb{C}}, \langle \epsilon_2, \epsilon_1, \epsilon_4,
    \epsilon_3, \epsilon_7,\epsilon_6 \rangle_{\mathbb{C}} \big)
    \in\mathcal{F}^7_{1,3,5,6},
  \end{equation*}
  and $M^{(2)}\subset\mathcal{F}^7_{1,2,4,6} $ is the
  $\mathbf{G}_0$-orbit through the flag
  \begin{equation*}
    \varepsilon^{2}=\big(\langle\epsilon_2\rangle_{\mathbb{C}},
    \langle\epsilon_2,\epsilon_4\rangle_{\mathbb{C}}, \langle
    \epsilon_2, \epsilon_4, \epsilon_1, \epsilon_7
    \rangle_{\mathbb{C}}, \langle \epsilon_2, \epsilon_4, \epsilon_1,
    \epsilon_7, \epsilon_3,\epsilon_6 \rangle_{\mathbb{C}} \big)
    \in\mathcal{F}^7_{1,2,4,6}. 
  \end{equation*}
  Being totally real, $M^{(2)}$ coincides with
  $M_{\mathfrak{e}}$.\par
  Denote by $\lambda^{\! *}:\mathcal{R}\to\mathcal{R}^{\mathfrak{e}}$ the
  Cayley transform with respect to the roots $\epsilon_1-\epsilon_7$
  and $\epsilon_3-\epsilon_6$ According to \eqref{eq:ik},
\eqref{eq:il},
  the fundamental group of $M_{\mathfrak{e}}$ is generated by
  $\xi_1=\xi_{\lambda^{\! *}(\epsilon_2-\epsilon_4)}$,
  $\xi_2=\xi_{\lambda^{\! *}(\epsilon_4-\epsilon_1)}$,
  $\xi_4=\xi_{\lambda^{\! *}(\epsilon_7-\epsilon_3)}$ and
  $\xi_6=\xi_{\lambda^{\! *}(\epsilon_6-\epsilon_5)}$, with relations
  $\xi_i^2=1$ and $\xi_i \xi_j=\xi_j \xi_i$ for $i,j=1,2,4,6$.\par
  The semisimple part $\mathfrak{s}_0^{\mathfrak{e}}$ of the isotropy
  subgroup of $M_{\mathfrak{e}}$ is the subgroup corresponding to the
  root subsystem $\{\pm(\epsilon_1-\epsilon_7),
  \pm(\epsilon_3-\epsilon_6)\}$, the subset $E$ in Remark~\ref{rmk:ih}
  is $\{\epsilon_1-\epsilon_7, \epsilon_3-\epsilon_6\}$ and the map $f:
  \mathbf{W}(E)\to\:\Lambda^{*}/2\varpi(\Lambda)$ is a bijection. By
  Theorem~\ref{thm:ig} the fundamental group of $M$ is generated by
  $\xi_1$ and $\xi_2\xi_4\xi_6$ and is isomorphic to
  $\mathbb{Z}_2^2$.\par 
The space of algebraic arc components of $M$ is the
  $\mathbf{G}_0$-orbit $M_a\subset\mathcal{F}^7_{1,4,6}$ through the
  flag
  \begin{equation*}
    \varepsilon^a=\big(\langle\epsilon_2\rangle_{\mathbb{C}},
    \langle \epsilon_2, \epsilon_4, \epsilon_1, \epsilon_7
    \rangle_{\mathbb{C}}, \langle \epsilon_2, \epsilon_4, \epsilon_1,
    \epsilon_7, \epsilon_3,\epsilon_6 \rangle_{\mathbb{C}} \big)
    \in\mathcal{F}^7_{1,4,6}, 
  \end{equation*}
  showing that in this case
  $\mathfrak{e}\subsetneq\mathfrak{q}_a$. Furthermore, the algebraic
  arc components are not simply connected, and their fundamental group
  is isomorphic to $\mathbb{Z}_2$.
\end{exam}
\begin{exam}\label{exam:kg}
We keep the notation of the previous example. We consider now the
basis $(w_1,\hdots,w_7)$ of $\mathbb{C}^7$ given by
\begin{equation*}
  \begin{gathered}
    w_1=e_1+ie_7,\quad w_2=e_2,\quad w_3=e_3, \quad w_4=e_4\\
w_5=e_5,\quad w_6=e_6,\quad w_7=e_1-ie_7.
  \end{gathered}
\end{equation*}
With $\mathbf{G}_0=\mathbf{SL}(7,\mathbb{R})$, we consider the
$\mathbf{G}_0$-orbit $M\subset\mathcal{F}^7_{2,3}$ through the flag
\begin{equation*}
  \omega=\big(\langle{w_1,w_2}\rangle,\langle{w_1,w_2,w_3}\rangle\big)
\in \mathcal{F}^7_{2,3}.
\end{equation*}
Then $M_{\mathfrak{e}}=M^{(1)}\subset\mathcal{F}^7_{1,4}$ is the
$\mathbf{G}_0$-orbit through the flag
\begin{equation*}
  \omega^{(1)}=\big(\langle{w_2}\rangle,\langle{w_1,w_2,w_3,w_7}\rangle
\big) 
\in\mathcal{F}^7_{1,4}.
\end{equation*}
On the other hand, the space $M_a$ of the algebraic arc
components of $M$ is the $\mathbf{G}_0$-orbit 
$M_a\subset\mathcal{F}^7_{1,2,4}$ through the flag
\begin{equation*}
  \omega^{(a)}=\big(\langle{w_2}\rangle,\langle{w_2,w_3}\rangle,
\langle{w_1,w_2,w_3,w_7}\rangle\big)
\in\mathcal{F}^7_{1,2,4}.
\end{equation*}
Thus we obtain in this case that $\mathfrak{q}_a\subsetneq\mathfrak{e}$.
\end{exam}
\begin{prop}
  \label{prop:kg}
  Let $(\mathfrak{g}_0,\mathfrak{q})$ be a parabolic $CR$ algebra,
$(\vartheta,\mathfrak{h}_0)$ an adapted Cartan pair, $\mathcal{R}$
the root system of $\mathfrak{g}$ with respect to the complexification
of $\mathfrak{h}_0$, and $\mathcal{Q}$ the parabolic set of 
$\mathfrak{q}$. 
Assume that
there is an $S$-fit Weyl chamber in $\mathfrak{C}(\mathcal{R},\mathcal{Q})$
with
\begin{equation}
  \label{eq:kj}
  \bar\alpha\succ 0 \quad\text{if}\quad\alpha\in\mathcal{R}_{\mathrm{cpx}}
\quad\text{and}\quad\alpha\succ 0.
\end{equation}
Then:
\begin{align}
\label{eq:km}
&\mathfrak{q}_a=\mathfrak{e},\\
  \label{eq:kk}
&\text{$\mathfrak{q}_w+\bar{\mathfrak{q}}_w$ is a Lie subalgebra of
$\mathfrak{g}$,}\\
\label{eq:kl}
&
\text{if moreover $(\mathfrak{g}_0,\mathfrak{q})$ is weakly nondegenerate, then
$\mathfrak{e}=\mathfrak{q}_w+\bar{\mathfrak{q}}_w$.}
\end{align}\par
In particular \eqref{eq:km} holds true when 
$(\mathfrak{g}_0,\mathfrak{q})$ is the parabolic $CR$ algebra
  corresponding to a closed $\mathbf{G}_0$-orbit. \end{prop}
\begin{proof}
Fix an adapted Cartan pair $(\vartheta,\mathfrak{h}_0)$ 
  and choose an $S$-fit Weyl
  chamber $C$ for $(\mathfrak{g}_0,\mathfrak{q})$,
such that \eqref{eq:kj} is satisfied.  We have
  $\mathfrak{q}_w+\bar{\mathfrak{q}}_w=
  \mathfrak{q}^n+\bar{\mathfrak{q}}^n +
  (\mathfrak{q}\cap\bar{\mathfrak{q}})$. 
Hence, to prove \eqref{eq:kk},
we only need to show that, if a root $\gamma$ is the sum
$\gamma=\alpha+\beta$ of a root 
  $\alpha\in\mathcal{Q}^n$ and a root 
  $\beta\in\bar{\mathcal{Q}}^n\setminus\mathcal{Q}$, then
  $\gamma\in\mathcal{Q}_w\cup\bar{\mathcal{Q}}_w$.  Because
of \eqref{eq:kj}, $\beta$ is a negative imaginary root.
If also $\alpha$ is imaginary,
  then $\gamma$ is imaginary and hence belongs to
  $\mathcal{Q}_w\cup\bar{\mathcal{Q}}_w$. If $\alpha$ is not
  imaginary, then $\gamma$ is positive and not imaginary, and
  thus belongs to $\mathcal{Q}\cap\bar{\mathcal{Q}}$. \par
Statement \eqref{eq:kl} follows from \eqref{eq:kk}.\par
While proving \eqref{eq:km}, 
  by Lemma~\ref{lem:kb} we can assume that
  $(\mathfrak{g}_0,\mathfrak{q})=(\mathfrak{g}_0,\mathfrak{q}^{(0)})$
  is weakly nondegenerate. 
Since by \eqref{eq:kl} we have in this case
$\mathfrak{q}_w+\bar{\mathfrak{q}}_w=\mathfrak{q}^{(1)}=\mathfrak{e}$,
the inclusion $\mathfrak{q}_a\subset\mathfrak{e}$
follows by  Lemma~\ref{lem:kd}.\par
We claim that
$\mathfrak{e} \subset
\mathbf{N}_{\mathfrak{g}}(\mathfrak{q}^n\cap\bar{\mathfrak{q}}^n) +
\sum_{\alpha\in\mathcal{R}_{\mathrm{im}}} \mathfrak{g}^{\alpha}$. Indeed
let $\alpha\in\mathcal{R}$ be a root for which
$\mathfrak{g}^{\alpha}\subset\mathfrak{e}$ but $\mathfrak{g}^{\alpha}
\not\subset
\mathbf{N}_{\mathfrak{g}}(\mathfrak{q}^n\cap\bar{\mathfrak{q}}^n)$. 
Since $\mathfrak{e}=\mathfrak{q}^n+\bar{\mathfrak{q}}^n+(\mathfrak{q}\cap
\bar{\mathfrak{q}})$, 
then either $\alpha\in\mathcal{Q}^n\setminus\bar{\mathcal{Q}}$,
or $\alpha\in\bar{\mathcal{Q}}^n\setminus{\mathcal{Q}}$.
By \eqref{eq:kj}, in both cases we have
$\alpha\in\mathcal{R}_{\mathrm{im}}$. Since
$\mathbf{N}_{\mathfrak{g}}(\mathfrak{q}^n\cap\bar{\mathfrak{q}}^n)
\subset \mathfrak{q}_a$ by \cite[Theorem~8.5(2)]{Wolf69}, and
$\sum_{\alpha\in\mathcal{R}_{\mathrm{im}}} \mathfrak{g}^{\alpha}
\subset \mathfrak{q}_a$ because $\mathfrak{q}_a$ is the
complexification of a real parabolic subalgebra of $\mathfrak{g}_0$,
we obtain that $\mathfrak{e}\subset\mathfrak{q}_a$.\par
The last statement is a consequence of the fact (see e.g. \cite{AMN08})
that for a closed orbit there are $S$-fit Weyl chambers for which
\eqref{eq:kj} holds true.
\end{proof}
\bibliographystyle{amsplain}
\renewcommand{\MR}[1]{}
\providecommand{\bysame}{\leavevmode\hbox to3em{\hrulefill}\thinspace}
\providecommand{\MR}{\relax\ifhmode\unskip\space\fi MR }
\providecommand{\MRhref}[2]{%
  \href{http://www.ams.org/mathscinet-getitem?mr=#1}{#2}
}
\providecommand{\href}[2]{#2}

\end{document}